\documentclass[11pt]{article}
\usepackage[margin=1in]{geometry}
\usepackage{amsmath,amssymb,amsthm,verbatim}

\title{A Structure Theorem for Poorly Anticoncentrated Gaussian Chaoses and Applications to the Study of Polynomial Threshold Functions}
\author{Daniel M. Kane\\ Department of Mathematics \\ Stanford University\\ dankane@math.stanford.edu}

\newcommand{\R}{\mathbb{R}}

\newcommand{\jac}{\textrm{Jac}}
\newcommand{\var}{\textrm{Var}}

\newcommand{\pr}{\textrm{Pr}}
\newcommand{\sgn}{\textrm{sgn}}
\newcommand{\E}{\mathbb{E}}
\newcommand{\dotp}[2]{\left\langle #1,#2\right\rangle}

\newcommand{\dlots}{\ldots}
\newcommand{\etc}{\textrm{etc}}
\newcommand{\Inf}{\textrm{Inf}}
\newcommand{\gns}{\mathbb{GNS}}
\newcommand{\as}{\mathbb{AS}}
\newcommand{\ns}{\mathbb{NS}}
\newcommand{\gas}{\mathbb{GAS}}
\newcommand{\cov}{\textrm{Cov}}
\newcommand{\tth}{\textrm{th}}

\newtheorem{thm}{Theorem}
\newtheorem{prop}[thm]{Proposition}
\newtheorem{cor}[thm]{Corollary}
\newtheorem{lem}[thm]{Lemma}
\newtheorem{conj}[thm]{Conjecture}

\newtheorem*{defn}{Definition}
\newtheorem*{rmk}{Remark}

\begin{document}
\maketitle

\begin{abstract}
We prove a structural result for degree-$d$ polynomials.  In particular, we show that any degree-$d$ polynomial, $p$ can be approximated by another polynomial, $p_0$, which can be decomposed as some function of polynomials $q_1,\ldots,q_m$ with $q_i$ normalized and $m=O_d(1)$, so that if $X$ is a Gaussian random variable, the probability distribution on $(q_1(X),\ldots,q_m(X))$ does not have too much mass in any small box.

Using this result, we prove improved versions of a number of results about polynomial threshold functions, including producing better pseudorandom generators, obtaining a better invariance principle, and proving improved bounds on noise sensitivity.
\end{abstract}

\section{Introduction}\label{IntroSec}

\subsection{Polynomial Threshold Functions}\label{PTFSec}

A polynomial threshold function (PTF) is a function of the form $f(X)=\sgn(p(X))$ for some polynomial $p(X)$.  We say that $f$ is a degree-$d$ polynomial threshold function of $p$ is of degree at most $d$.  Polynomial threshold functions are a fundamental class of functions with applications to many fields such as circuit complexity \cite{curciutApp}, communication complexity \cite{commApp} and learning theory \cite{learningApp}.

We present a new structural result for degree-$d$ polynomials that allows us to obtain improved versions of a number of results relating to polynomial threshold functions. Our result allows us to define a new notion of regularity for polynomials for which we can prove an improved version of the Invariance Principle of \cite{MOO}.  We also obtain a regularity lemma (along the lines of the main theorem of \cite{reg}) for this new notion of regularity.  Although neither of these theorems will be directly comparable to their classical versions (due to the different notions of regularity), the combination of our regularity lemma and invariance principle produces a marked improvement over previous work.  These results in turn allow us to prove better bounds on the noise sensitivity of polynomial threshold functions (improving on the bounds of \cite{sensitivity} for fixed $d\geq 3$) and provide us with an improved analysis of the pseudorandom generators of \cite{MZ} and \cite{GPRG}.

\subsection{Anticoncentration and Diffuse Decompositions}\label{ADDSec}

Many of the analytic techniques for dealing with polynomial threshold functions (most notably the replacement method (see \cite{Freplacement,LReplacement})) work well for dealing with smooth functions of polynomials.  In order to get these techniques to yield useful results for threshold functions, it is often necessary to approximate the threshold function by a smooth one.  In order to obtain one needs to know that with high probability that the value of $p(X)$ does not lie too close to zero.  Results of this form have become known as a \emph{anticoncentration} results. Such a result was proved by Carbery and Wright in \cite{anticoncentration}.  They prove that for $p$ a degree-$d$ polynomial and $X$ a random Gaussian that
\begin{equation}\label{CWEqn}
\pr(|p(X)| \leq \epsilon|p|_2) = O(d\epsilon^{1/d}).
\end{equation}

This bound has proved to be an essential component of many theorems about polynomial threshold functions.  Unfortunately, presence of the $\epsilon^{1/d}$ term above often leads to results that have poor $\epsilon$-dependence for moderately large values of $d$, and the lack of a stronger form of Equation (\ref{CWEqn}) has proved to be a bottleneck for a number of results on polynomial threshold functions.  One might hope to overcome this difficulty by proving an improved version of Equation (\ref{CWEqn}).  In particular, a generic polynomial $p$ can be thought of as a sum of largely uncorrelated monomials, and thus, one might expect that $p(X)$ be Gaussian distributed.  Thus, while Equation (\ref{CWEqn}) tells us little more than the fact that the distribution of $p(X)$ has no point masses, one might expect the stronger condition that $p(X)$ has bounded probability density function to hold.  Unfortunately, this is not the case in general.  For example, if $p$ is the $d^{\tth}$ power of a linear polynomial, the probability that $|p(X)|<\epsilon$ will in fact be proportional to $\epsilon^{1/d}$.  On the other hand, this counterexample is not as great an obstacle as it first appears to be.  While, in this case, the probability distribution of $p(X)$ does have poor analytic properties, this is because $p$ can be written as a composition of a well-behaved (in this case linear) polynomial, and a simple, yet poorly-behaved polynomial (the $d^{\textrm{th}}$ power).  The fact that the value of $p$ is governed by the value of this linear polynomial will allow one to overcome the difficulties posed by poor anticoncentration in most applications.

In fact, this principle applies more generally.  In particular, as we shall show, any polynomial $p$ may be approximately represented as the composition of a simple polynomial (i.e. a polynomial dependent on few input variables) and an analytically well-behaved polynomial (i.e. one with good anticoncentration properties).  In order to make this claim rigorous, we provide the following definitions:

\begin{defn}
Given a degree-$d$ polynomial $p:\R^n\rightarrow\R$, we say that a set of polynomials $(h,q_1,\ldots,q_m)$ is a \emph{decomposition of $p$ of size $m$} if $q_i:\R^n\rightarrow \R$, and $h:\R^m\rightarrow\R$ are polynomials so that
\begin{itemize}
\item $p(x) = h(q_1(x),\ldots,q_m(x))$
\item For every monomial $c\prod x_i^{a_i}$ appearing in $h$, we have that $\sum a_1 \deg(q_i) \leq d$
\end{itemize}
\end{defn}

In other words, a decomposition of $p$ is a way of writing $p$ as a composition of a simple polynomial, $h$, with another polynomial $Q=(q_1,\ldots,q_m)$.  The second condition above tells us that if we expanded out the polynomial $h(q_1(x),\ldots,q_m(x))$, we would never have to write any terms of degree more than $d$.

\begin{defn}
We say that a tuple of polynomials $(q_1,\ldots,q_m):\R^n\rightarrow\R^m$, is an \emph{$(\epsilon,N)$-diffuse set} if for every $(a_1,\ldots,a_m)\in \R^m$ and Gaussian random variable $X$ we have that
$$
\pr_X(|q_i(X)-a_i| \leq \epsilon \textrm{ for all }i) \leq \epsilon^m N,
$$
and $\E[|q_i(X)|^2] \leq 1$ for all $i$.
\end{defn}

We note that while an anticoncentration result need only tell us that the probability distribution of $p(X)$ contains no point masses, an $(\epsilon,N)$-diffuse set of polynomials will have the probability density function of the vector $(q_1(X),\ldots,q_m(X))$ average no more than $N$ on any small box.  This is a much stronger notion of ``analytically well-behaved''.  Combining the two definitions above, we define the notion of a \emph{diffuse decomposition}.

\begin{defn}
Given a polynomial $p$ we say that $(h,q_1,\ldots,q_m)$ is an \emph{$(\epsilon,N)$-diffuse decomposition of $p$ of size $m$} if $(h,q_1,\ldots,q_m)$ is a decomposition of $p$ of size $m$ and if $(q_1,\ldots,q_m)$ is an $(\epsilon,N)$-diffuse set.
\end{defn}

It is not obvious that diffuse decompositions should exist in any useful cases.  The main result of this paper will be to show that not only can any polynomial be approximated by a polynomial with a diffuse decomposition, but that the parameters of this decomposition are sufficient for use in a wide variety of applications.

\begin{thm}[The Diffuse Decomposition Theorem]\label{DDTheorem}
Let $\epsilon,c$ and $N$ be positive real numbers and $d$ a positive integer.  Let $p(X)$ be a degree-$d$ polynomial.  Then there exists a degree-$d$ polynomial $p_0$ with $|p-p_0|_2 < O_{c,d,N}(\epsilon^N)|p|_2$ so that $p_0$ has an $(\epsilon,\epsilon^{-c})$-diffuse decomposition of size at most $O_{c,d,N}(1)$.
\end{thm}

It should be noted that if $p$ is a polynomial with a diffuse decomposition, $(h,q_1,\ldots,q_m)$, then the distribution of $p(X)$ will be determined in large part by the polynomial $h$, as the distribution for $(q_1(X),\ldots,q_m(X))$ is controlled by the diffuse property.  Thus, Theorem \ref{DDTheorem} may be thought of as a structural result for Gaussian chaoses.  Theorem \ref{DDTheorem} may also be thought of as a continuous analogue of theorems of Green-Tao (\cite{GT}) and Kaufman-Lovett (\cite{KL}) which say that a polynomial over a finite field can be decomposed in terms of lower degree polynomials whose output distributions on random inputs are close to uniform.

\begin{rmk}
The bound on the size of the decomposition in Theorem \ref{DDTheorem} is effective, but may be quite large.  Working through the details of the proof would lead to a bound of $A(d+O(1),N/c)$, where $A(m,n)$ is the Ackermann function.  The author believes that a polynomial in $(dN/c)$ should be sufficient, but does not know of a proof for this improved bound.
\end{rmk}

\subsection{Applications of the Main Theorem}\label{ApplicationsSec}

Theorem \ref{DDTheorem} has several applications that we will discuss.  The existence of diffuse decompositions allows us to make better use of the replacement method and achieve a tighter analysis of the pseudorandom generators for polynomial threshold functions presented in \cite{GPRG} and \cite{MZ}.  We can also use this theory to improve on the Invariance Principle of \cite{MOO}.  In particular, we come up with a new notion of regularity for a polynomial, so that for highly regular polynomials their evaluation at random Gaussian variables and at random Bernoulli variables are close in cdf distance.  We then show that an arbitrary polynomial can be written as a decision tree of small depth almost all of whose leaves are either regular or have constant sign with high probability.  These theorems of ours will produce a qualitative improvement over the analogous theorems of \cite{MOO} and \cite{reg}.  Finally, we make use of this technology to prove new bounds on the noise sensitivity of polynomial threshold functions. Each of these applications will be discussed in more detail in the relevant section of this paper.

\subsection{Overview of the Paper}\label{OutlineSec}

In Section \ref{BasicStuffSec}, we introduce a number of basic concepts that will be used throughout the paper.  Section \ref{MainSec} will contain the proof of Theorem \ref{DDTheorem} along with some associated lemmas.  In Section \ref{DDFactsSec}, we discuss some basic facts about diffuse decompositions that will prove useful to us later on.  In Section \ref{GaussianPRGSec}, we discuss our application to pseudorandom generators for polynomial threshold functions of Gaussians.  In Section \ref{InvarianceSec}, we state and prove our versions of the invariance principle and regularity lemma.  In Section \ref{GLSec}, we discuss our results relating to noise sensitivity problems.  In Section \ref{BernoulliPRGSec}, we discuss our results for pseudorandom generators for polynomial threshold functions with Bernoulli inputs.  Finally in Section \ref{ConclusionSec}, we provide some closing remarks.

\section{Basic Results and Notation}\label{BasicStuffSec}

\subsection{Basic Notation}

We will use the notation $O_a(N)$ to denote a quantity whose absolute value is bounded above by $N$ times some constant depending only on $a$.

Throughout this paper, the variables $G,X,Y,Z,X^i,Y^i,Z^i,\textrm{etc.}$ will be used to denote multidimensional Gaussian random variables unless stated otherwise.  The coordinates of these variables will be denoted using subscripts.  Thus, $X^i_j$ will denote the $j^{\tth}$ coordinate of the variable $X^i$.

We also recall here the definition of a polynomial threshold function
\begin{defn}
A function $f:\R^n\rightarrow\{\pm 1 \}$ is a (degree-$d$) \emph{polynomial threshold function} (or PTF) if it is of the form $f(x) = \sgn(p(x))$ for some (degree-$d$) polynomial $p$.
\end{defn}

\subsection{Basic Facts about Polynomials of Gaussians}

We recall some basic facts about polynomials of Gaussians.  We begin by recalling the $L^t$-norm of a function.

\begin{defn}
For a function $p:\R^n\rightarrow\R$, we let
$$
|p|_t = \left( \E_X[|p(X)|^t ] \right)^{1/t}.
$$
\end{defn}

We now recall some basic distributional results about polynomials evaluated at random Gaussians.

\begin{lem}[Carbery and Wright]\label{anticoncentrationLem}
If $p$ is a degree-$d$ polynomial then
$$
\pr(|p(X)| \leq \epsilon|p|_2) = O(d\epsilon^{1/d}),
$$
where the probability is over $X$, a standard $n$-dimensional Gaussian.
\end{lem}

We will make use of the hypercontractive inequality.  The proof follows from Theorem 2 of \cite{hypercontractivity}.

\begin{lem}\label{hypercontractiveLem}
If $p$ is a degree-$d$ polynomial and $t>2$, then
$$
|p|_t \leq \sqrt{t-1}^d |p|_2.
$$
\end{lem}

In particular this implies the following Corollary:

\begin{cor}[Weak Anticoncentration]\label{WeakAnticoncentrationCor}
Let $p$ be a degree-$d$ polynomial in $n$ variables.  Let $X$ be a family of standard Gaussians.  Then
$$
\pr\left(|p(X)|\geq |p|_2/2\right) \geq 9^{-d}/2.
$$
\end{cor}
\begin{proof}
This follows immediately from the Paley–Zygmund inequality (\cite{PZ}) applied to $p^2$.
\end{proof}

We also have the following concentration bound.

\begin{cor}\label{ConcentrationCor}
If $p$ is a degree-$d$ polynomial and $N>0$, then
$$
\pr_X(|p(X)| > N|p|_2) = O\left(2^{-(N/2)^{2/d}} \right).
$$
\end{cor}
\begin{proof}
Apply the Markov inequality and Lemma \ref{hypercontractiveLem} with $t = (N/2)^{2/d}$.
\end{proof}

\subsection{Multilinear Algebra}

The conventions and results discussed in the remainder of this section will be used primarily in Section \ref{MainSec}, and sparingly in the rest of the paper.

We will later need to make some fairly complicated constructions making use of multilinear algebra.  We take this time to review some of the basic definitions and go over some of the notation that we will be using.  We recall that a $k$-tensor is an element of a $k$-fold tensor product of vector spaces $A\in V_1\otimes\cdots\otimes V_k$.  Equivalently, it may be thought of as the $k$-linear form $V_1\times\cdots \times V_k \rightarrow \R$ given by $(v_1,\ldots,v_k) \rightarrow \dotp{A}{v_1\otimes\cdots\otimes v_k}$ (assuming that each of the $V_i$ come equipped with an inner product).  If the $V_i$ come with isomorphisms to $\R^{n_i}$, then we can associate $A$ with the sequence of coordinates $A_{i_1\cdots i_k} = A(e_{i_1},\ldots,e_{i_k})$.

We recall Einstein summation notation which says that if we are given a product of tensors with stated indices that it is implied that we sum over any shared indices.  In particular if $A$ is a $k_1$-tensor and $B$ a $k_2$-tensor than the expression
$$
A_{i_1,i_2,\ldots,i_m,j_1,j_2,\ldots,j_{k_1-m}}B_{i_1,i_2,\ldots,i_m,j_{k_1-m+1},j_{k_1-m+2},\ldots,j_{k_1+k_2-2m}}
$$
denotes the $(k_1+k_2-2m)$-tensor $C$ with coordinates
\begin{align*}
C  _{j_1,j_2, \ldots,j_{k_1+k_2-2m}} = \sum_{i_1,i_2,\ldots,i_m}A_{i_1,i_2,\ldots,i_m,j_1,j_2,\ldots,j_{k_1-m}}\cdot B_{i_1,i_2,\ldots,i_m,j_{k_1-m+1},j_{k_1-m+2},\ldots,j_{k_1+k_2-2m}}.
\end{align*}
Note that if there are no overlapping indices that this product simply denotes the tensor product of $A$ and $B$.  If on the other hand, all indices overlap, this denotes the dot product of $A$ and $B$.  We will also sometimes group several coordinates into a single coordinate of larger dimension.  We will try to use upper case letters for indices to indicate that this is happening.

We define the $L^2$ norm of a tensor $A$ to be the square root of the sum of the squares of its coordinates.  If $A$ is a $k$-tensor we have the equivalent definitions:
\begin{align*}
|A|_2^2 & = \dotp{A}{A}\\
& = \sum_{i_1,\ldots,i_k} |A_{i_1,\ldots,i_k}|^2\\
& = \E_{X^1,\ldots,X^k}[ |A_{i_1,\ldots,i_k} X^1_{i_1} X^2_{i_2}\cdots X^k_{i_k}|^2].
\end{align*}
For tensor-valued functions $A(X)$ we define the $L^2$-norm by
$$
|A|_2^2 := \E_X[|A(X)|_2^2].
$$

We will also need the notion of a wedge product of tensors over some subset of their coordinates.  In particular, if $A$ is a rank-$(k+m)$ tensor with its first $k$ indices corresponding to spaces of the same dimension, we define
$$
\bigwedge_{i_1,\ldots,i_k} A_{i_1,\ldots,i_k,j_1,\ldots,j_m} := \sum_{\sigma \in S_k} (-1)^\sigma A_{i_{\sigma(1)},\ldots,i_{\sigma(k)},j_1,\ldots,j_m}.
$$
Note the important special case here where $A$ is a tensor product of $k$ different 1-tensors $A_{i_1,\ldots,i_k}=A^1_{i_1}\cdots A^k_{i_k}$.  It is then the case that
$$
\left(\bigwedge_{i_1,\ldots,i_k} A^1_{i_1}\cdots A^k_{i_k}\right) B^1_{i_1}\cdots B^k_{i_k} = \det\left(\dotp{A^i}{B^j} \right).
$$

We will think of the derivative operator as taking functions on $\R^n$ whose values are $k$-tensors to functions on $\R^n$ whose values are $(k+1)$-tensors.  In particular, given a tensor valued function $A_S(x)$, we define the tensor $D_iA_S(x)$ to have $(i,S)$-coordinate $\frac{\partial A_S(x)}{\partial x_i}$.  Note that this implies that for a vector $X$ that $D_i X_i A_S$ is simply the standard directional derivative $D_X A_S$.

Lastly, note that if $p$ is a homogeneous, degree-$d$  polynomial that it has an associated $d$-tensor $A$ given by $A_{i_1,\dlots,i_d}:= D_{i_1}\cdots D_{i_d} p$ (note that this $d^{\tth}$ order derivative is independent of the point at which it is being evaluated).  Note that $A$ is determined by the property that it is a symmetric tensor (it is invariant under any permutation of coordinates) so that for any vector $X$, $A(X,X,\ldots,X) = d! p(X)$.

\subsection{Strong Anticoncentration}

Strong anticoncentration was an idea first exposed by the author in \cite{GPRG}. It is a heuristic which states that a polynomial is generally not much smaller than its derivative. We will need to make use of a generalization of this to sets of several tensor-valued polynomials.  In particular we will prove the following proposition:

\begin{prop}[Strong Anticoncentration]\label{StrongAnticoncentrationProp}
For $1\leq i \leq k$ let $A^i_{S_i}(x)$ be a degree-$d_i$, tensor-valued polynomial on $\R^n$ (i.e. a tensor whose coefficients are degree-$d_i$ polynomials on $\R^n$).  Let $1/2>\epsilon > 0$.  We have that
\begin{align*}
\pr & \left(\prod_{j=1}^k |A^j_{S_j}(X)|_2  < \epsilon \left|\bigwedge_{i_1,\ldots,i_k} \prod_{j=1}^k D_{i_j}A^j_{S_j}(X) \right|_2 \right)
\leq \epsilon 2^{O(d_1+d_2+\cdots + d_k)}O(\sqrt{k})^{k+1} \log(\epsilon^{-1})^k.
\end{align*}
\end{prop}

In order to prove Proposition \ref{StrongAnticoncentrationProp} we will need to following lemma:

\begin{lem}\label{SALem}
For $1\leq i\leq k$ let $p^i$ be a degree $d_i$ polynomial on $\R^n$ and let $\delta,\epsilon_i>0$.  Then
\begin{align*}
\pr_{X,Y^1,\ldots,Y^k} & \left(|p^i(X)| < \epsilon_i \textrm{ for all }i, \textrm{ and } |\det(D_{Y^j}p^i(X))|>\delta \right) \leq \frac{2^{k+1}\prod_{i=1}^k d_i \prod_{i=1}^k \epsilon_i}{\delta V_k}
\end{align*}
where $V_k = \frac{2\pi^{k+1/2}}{\Gamma((k+1)/2)}$ is the volume of the unit $k$-sphere.
\end{lem}
\begin{proof}
Define the function $f:S^k\rightarrow \R^k$ by letting
$$f(a_0,a_1,\ldots,a_k)_i := p^i(a_0 X + a_1 Y^1 + a_2 Y^2 + \ldots + a_k Y^k).$$
Notice that the matrix with coefficients $D_{Y^j}p^i(X)$ is simply the Jacobian of $f$ at the point $(1,0,0,\ldots,0)$.  Notice that if we replace the random variables $X,Y^1,\ldots,Y^k$ by linear combinations of each other by making an orthonormal change of coordinates, that they are still independent Gaussians and thus,the probability in question is unchanged.  We claim that for any fixed values of $X,Y^i$ that the probability over a random such change of variables that $$|p^i(X)| < \epsilon_i \textrm{ for all }i, \textrm{ and } |\det(D_{Y_j}p^i(X))|>\delta$$ is at most $\frac{2^k\prod_{i=1}^k d_i \prod_{i=1}^k \epsilon_i}{\delta V_k}.$  Such a statement would clearly imply our lemma.

Note that making such a random change of variables is equivalent to precomposing $f$ with a random element of the orthogonal group $O(k+1)$.  Thus, it suffices to bound
$$
\pr_{x\in S^k} \left( f(x) \in R, \textrm{ and } |\det(\jac(f(x)))|>\delta\right),
$$
where $R\subset \R^k$ is given by $\prod_i [-\epsilon_i,\epsilon_i]$.
Let $T$ be the set of $x\in S^k$ so that $f(x) \in R, \textrm{ and } |\det(\jac(f(x)))|>\delta.$  We know by the change of variables formula for integration that
\begin{equation}\label{integrationEqn}
\int_T |\det(\jac(f(x)))|dx = \int_{\R^n} |f^{-1}(y)| dy.
\end{equation}
We note that the right hand side of Equation (\ref{integrationEqn}) is $\int_R |f^{-1}(y)| dy$.  By Bezout's Theorem, the integrand is at most $2\prod_{i=1}^k d_i$ except on a set of measure 0.  Thus, $\int_{\R^n} |f^{-1}(y)| dy \leq 2^{k+1}\prod_{i=1}^k d_i \prod_{i=1}^k \epsilon_i.$  On the other hand the left hand side of Equation (\ref{integrationEqn}) is at least $\delta \textrm{Vol}(T)=\delta V_k \pr_{x\in S^k}(x\in T)$.  Thus, $$
\pr_{x\in S^k}(x\in T) \leq \frac{2^{k+1}\prod_{i=1}^k d_i \prod_{i=1}^k \epsilon_i}{\delta V_k}.
$$
\end{proof}

\begin{cor}\label{SACor}
For polynomials $p^i:\R^n\rightarrow\R$ of degree $d_i$ for $1\leq i \leq k$ and for $1/2>\epsilon>0$,
\begin{align*}
\pr_{X,Y^1,\ldots,Y^k}&\left( \prod_{i=1}^k |p^i(X)| < \epsilon \left|\det \left( D_{Y^i} p^j \right) \right| \right)  \leq \epsilon 2^{O(d_1+d_2+\cdots + d_k)}O(\sqrt{k})^{k+1} \log(\epsilon^{-1})^k.
\end{align*}
\end{cor}
\begin{proof}
We note that the problem in question is invariant under scalings of the $p^i$, and therefore we may assume that $|p^i|_2=1$ for all $i$.  We note by Lemma \ref{anticoncentrationLem} and Corollary \ref{ConcentrationCor} that we may ignore the case where some $|p^i(X)| < \epsilon^{d_i}$ or where some $|p^i(X)|>\epsilon^{-1}$ (as the probability that such an event happens for any $i$ is at most $O(\sum_i d_i\epsilon)$).  For each $i$ we may partition the interval $[\epsilon^{d_i},\epsilon^{-1}]$ into $O(d_i\log(\epsilon^{-1}))$ many intervals each of whose endpoints differ by at most a factor of 2.  Up to a factor of $O(\log(\epsilon^{-1}))^k \prod_i d_i$, it suffices to bound the probability that each of the $|p^i(X)|$ lies in a specified such interval and that $\prod_{i=1}^k |p^i(X)| < \epsilon \left|\det \left( D_{Y^i} p^j \right) \right|.$  If the upper endpoints of these intervals are $\epsilon_i$, then this probability, is at most the probability that
$$
|p^i(X)| < \epsilon_i \textrm{ for all }i, \textrm{ and } |\det(D_{Y^j}p^i(X))|>2^k\epsilon \prod \epsilon_i.
$$
By Lemma \ref{SALem}, the above probability is at most $\epsilon 2^{O(d_1+d_2+\cdots + d_k)}O(\sqrt{k})^{k+1}$.  Multiplying by $O(\log(\epsilon^{-1}))^k \prod_i d_i$, yields our bound.
\end{proof}

\begin{proof}[Proof of Proposition \ref{StrongAnticoncentrationProp}]
For $Z$ a tensor of the same dimension as $A^j$, let $f^j_Z:\R^n\rightarrow \R$ be the function $f^j_Z(x) = \dotp{A^j(x)}{Z}$.
Note that
$$
\left|\bigwedge_{i_1,\ldots,i_k} \prod_{j=1}^k D_{i_j}A^j_{S_j}(X) \right|_2^2
=
\E_{Y^1,\ldots,Y^k,Z^1,\ldots,Z^k}\left[\left|\det\left(D_{Y^i}f^j_{Z^j}(X)\right)\right|^2\right].
$$
Furthermore,
$$
\left(\prod_{j=1}^k |A^j_{S_j}(X)|_2\right)^2 = \E_{Z^1,\ldots,Z^k}\left[\left|\prod_{j=1}^k f^j_{Z^j}(X) \right|^2 \right].
$$

Now suppose that for some choice of $X$ that
\begin{equation}\label{SAFailEqn}
\prod_{j=1}^k |A^j_{S_j}(X)|_2^2  < \epsilon^2 \left|\bigwedge_{i_1,\ldots,i_k} \prod_{j=1}^k D_{i_j}A^j_{S_j}(X) \right|_2^2.
\end{equation}
We have by Corollary \ref{WeakAnticoncentrationCor} that with probability at least $2^{O(k)}$ over the random Gaussians $Y^1,\ldots,Y^k$ and $Z^1,\ldots,Z^k$ that the left hand side of Equation (\ref{SAFailEqn}) is at least
$$
\left|\prod_{j=1}^k f^j_{Z^j}(X) \right|^2/2.
$$
By the Markov bound, we have that except for a probability of at most $2^{O(k)}$
the right hand side of Equation (\ref{SAFailEqn}) is at most
$$
\epsilon^2 2^{O(k)}\left|\det\left(D_{Y^i}f^j_{Z^j}(X)\right)\right|^2.
$$
Thus, whenever Equation (\ref{SAFailEqn}) holds, with probability at least $2^{O(k)}$ over $Y^i$ and $Z^i$ we have that
$$
\left|\prod_{j=1}^k f^j_{Z^j}(X) \right| \leq \epsilon 2^{O(k)}\left|\det\left(D_{Y^i}f^j_{Z^j}(X)\right)\right|.
$$
But by Corollary \ref{SACor} the probability of this happening (even for fixed $Z^i$) is at most
$$
\epsilon 2^{O(d_1+d_2+\cdots + d_k)}O(\sqrt{k})^{k+1} \log(\epsilon^{-1})^k.
$$
Thus, the probability of Equation (\ref{SAFailEqn}) holding is at most $2^{O(k)}$ times as much, which is still
$$
\epsilon 2^{O(d_1+d_2+\cdots + d_k)}O(\sqrt{k})^{k+1} \log(\epsilon^{-1})^k.
$$
\end{proof}

\subsection{Orthogonal Polynomials}

Here we review some basic facts about orthogonal polynomials.  Recall that the Hermite polynomials are an orthonormal basis for polynomials in one variable with respect to the Gaussian inner product.  In particular, they are defined by the properties that
\begin{itemize}
\item $H_n:\R\rightarrow \R$ is a degree-$n$ polynomial
\item $\E[H_n(X)H_m(X)] = \delta_{n,m}$ where $X$ is a one-dimensional Gaussian random variable
\end{itemize}
Furthermore, we have the relation that $H_n'(x) = \sqrt{n} H_{n-1}(x)$.  We can extend this theory to polynomials in $n$ variables as follows.  For $a=(a_1,\ldots,a_n)$ a vector of non-negative integers, we define the corresponding polynomial $H_a(x) = \prod_{i=1}^n H_{a_i}(x_i)$ on $\R^n$.  It is easy to check that the total degree of $H_a$ is $|a|_1 := \sum_{i=1}^n a_i$ and that $\E[H_a(X)H_b(X)]=\delta_{a,b}$.

Given a polynomial $p$ in $n$ variables, we can always write $p$ as a linear combination of Hermite polynomials.  In fact, it is easy to check that
$$
p(X) = \sum_{|a|_1 \leq \deg(p)} c_a(p) H_a(X)
$$
where
$$
c_a(p) = \E[p(X)H_a(X)].
$$
We define the $k^{\tth}$ harmonic component of $p$ to be
$$
p^{[k]}:=\sum_{|a|_1 = k} c_a(p) H_a(X).
$$
We say that $p$ is harmonic of degree $k$ if it equals its $k^{\tth}$ harmonic part.

Note that we can compute the derivative of $H_a$ as
$$
D_i H_a(X) = \sqrt{a_i} H_{a-e_i}(X).
$$
This is clearly a vector of polynomials that are harmonic of degree $|a|_1-1$.  Furthermore, we have that
\begin{align*}
\E[(D_i H_a(X)) (D_i H_b(X))] & = \sum_i \E[\sqrt{a_i b_i} H_{a-e_i}(X) H_{b-e_i}(X)]\\
& = \sum_i \sqrt{a_i b_i} \delta_{a-e_i,b-e_i}\\
& = \delta_{a,b} \sum_i \sqrt{a_ib_i}\\
& = \delta_{a,b} \sum_i a_i\\
& = |a|_1 \delta_{a,b}.
\end{align*}
Additionally, for $a\neq b$ each of the components of $D_i H_a$ is a Hermite polynomial orthogonal to the corresponding component of $D_i H_b$.  Iterating this, we can see that
\begin{align*}
\E[(D_{i_1}D_{i_2}\cdots D_{i_k} H_a(X)) & (D_{i_1}D_{i_2}\cdots D_{i_k} H_b(X))]  = |a|_1 (|a|_1-1) \cdots (|a|_1-k+1) \delta_{a,b}.
\end{align*}
Hence we have

\begin{lem}\label{DerivativeSizeLem}
For $p$ a polynomial of degree $d$,
$$
|D_{i_1}\cdots D_{i_k} p(X)|_2^2 \leq d(d-1)\cdots (d-k+1) |p|_2^2
$$
with equality if and only if $p$ is harmonic of degree $d$.
\end{lem}

\subsection{Ordinal Numbers}

A few of our proofs are going to use some basic facts about ordinal numbers that can be written as polynomials in $\omega$ to show that certain recursive procedures terminate.  If $p$ is a polynomial with non-negative integer coefficients we consider the ordinal number $p(\omega)$.  Recall that these numbers have a comparison operation given by $p(\omega)>q(\omega)$ if and only if the leading coefficient of $p-q$ is positive.  We will need the following lemma:

\begin{lem}\label{ordLem}
There is no infinite sequence of polynomials with non-negative integer coefficients, $p_i$ so that $p_1(\omega)>p_2(\omega)>p_3(\omega)>\ldots$.
\end{lem}
\begin{proof}
We prove that all such sequences are finite by induction on $\deg(p_1)$.  If $\deg(p_1)=0$, this sequence is just a decreasing sequence of non-negative integers and must therefore be finite.  Next suppose for sake of contradiction that we have such an infinite decreasing sequence where $\deg(p_1)=d$, and that we know that no such infinite sequences exist with $\deg(p_1)<d$.  Note that the $\omega^d$ coefficient of the $p_i$ is a non-increasing sequence of non-negative integers and therefore must eventually stabilize.  Hence there is some $a$ and $N$ so that for all $n>N$, $p_n(\omega) = a \omega^d + q_n(\omega)$ where $q_n$ is a polynomial with non-negative coefficients of degree at most $d-1$.  It is clear that $q_n(\omega)>q_{n+1}(\omega)>\ldots$, so by the inductive hypothesis, this sequence must be finite.
\end{proof}

Lemma \ref{ordLem} allows us to perform transfinite induction.  In particular, if we have some sequence of statements $S(p)$ indexed by polynomials $p$ in one variable with non-negative integer coefficients, and if furthermore we have that for any $p$,
$$
[S(q) \textrm{ for all } q \textrm{ so that }q(\omega) < p(\omega)] \Rightarrow S(p)
$$
then $S(p)$ will hold for all $p$.  This is true for the following reason.  Suppose for sake of contradiction that $S(p)$ were false for some $p=p_1$.  This would imply by the given property that there was some $p_2$ with $p_2(\omega)<p_1(\omega)$ for which $S(p_2)$ was false.  Similarly, given any $p_i$ for which $S(p_i)$ was false we could find a $p_{i+1}$ with $p_{i+1}(\omega) < p_i(\omega)$ for which $S(p_{i+1})$ was false.  This would give us an infinite sequence of polynomials $p_i$ so that $p_1(\omega)>p_2(\omega)>\ldots$, which would contradict Lemma \ref{ordLem}.

\section{Proof of the Decomposition Theorem}\label{MainSec}

\subsection{Overview of the Proof}

The proof of Theorem \ref{DDTheorem} comes in two steps.  The first is Proposition \ref{DecompositionProp} (below), which states roughly that if $p$ is a degree-$d$ polynomial so that for a random Gaussian $X$ $|p'(X)|$ is small with non-negligible probability, then $p$ can be decomposed as a polynomial with smaller $L^2$ norm, plus a sum of products of lower degree polynomials.  Given this proposition, the proof of Theorem \ref{DDTheorem} is relatively straightforward.  We begin by writing a trivial decomposition of $p$ as $p(x) = \textrm{Id}(p(x))$.  If this is a diffuse decomposition, we are done.  Otherwise, by Proposition \ref{StrongAnticoncentrationProp}, there must be a reasonable probably that $|p'(X)|$ is small.  Thus, Proposition \ref{DecompositionProp} allows us to decompose $p$ in terms of lower-degree polynomials.  This gives us a new decomposition of $p$.  If it is diffuse, we are done, otherwise it is not hard to show that at least one of the polynomials in this decomposition can be decomposed further.  We show that this procedure will eventually terminate by demonstrating an ordinal monovariant which decreases with each step.

In Section \ref{DecompSec}, we state and prove Proposition \ref{DecompositionProp}, and in Section \ref{DDpfSec} complete the proof of Theorem \ref{DDTheorem}.

\subsection{The Decomposition Lemma}\label{DecompSec}

In this section, we will prove the following important proposition that will allow us to write a non-diffuse polynomial in terms of lower-degree polynomials.

\begin{prop}\label{DecompositionProp}
Let $p(X)$ be a degree $d$ polynomial with $|p|_2 \leq 1$ and let $\epsilon,c,N>0$ be real numbers so that
$$
\pr_X(|D_i p(X)|_2 < \epsilon) > \epsilon^N.
$$
Then there exist polynomials $a_i(X)$,$b_i(X)$ of degree strictly less than $d$ with $|a_i(X)|_2|b_i(X)|_2\leq O_{N,c,d}(\epsilon^{-c})|p^{[d]}|_2$ and so that
$$
\left| \left(p(X) - \sum_{i=1}^k a_i(X) b_i(X)\right)^{[d]} \right|_2 < O_{N,c,d}(\epsilon^{1-c}),
$$
where $k = O_{N,c,d}(1)$.  Furthermore, this can be done in such a way that for each $i$, $\deg(a_i)+\deg(b_i)=d$.
\end{prop}
\begin{rmk}
Unlike the constants implied in Theorem \ref{DDTheorem}, the implied constants in Proposition \ref{DecompositionProp} are primitive recursive functions of the parameters. Although we do not bound them explicitly, our techniques show that they are at worst an iterated exponential.
\end{rmk}

Our proof of Proposition \ref{DecompositionProp} will proceed in stages.  First we will show that for such polynomials $p$, there is a reasonable probability (over $X,Y^i$) that $D_i D_{Y^1} D_{Y^2} \cdots D_{Y^{d-1}} p(X)$ will be small.  This is easily seen to reduce to a statement about the rank-$d$ tensor, $A_{i_1\cdots i_d} = D_{i_1}\cdots D_{i_d} p$.  In particular, we know that $A_{i_1\cdots i_d} Y^1_{i_1}\cdots Y^{d-1}_{i_{d-1}}$ has a reasonable probability of being small.  We then prove a structure theorem telling us that such tensors can be approximated as a sum of tensor products of lower-rank tensors.  This in turn will translate into our being able to approximate the degree-$d$ part of $p$ by a sum of products of lower degree polynomials.

We begin with the following proposition:
\begin{prop}\label{SmallDerProp}
Let $c,N>0$ be real numbers and $d$ a positive integer.  Let $\epsilon>0$ be a real number that is sufficiently small given $c,d$ and $N$.  Suppose that $p$ is a degree-$d$ polynomial so that
$$
\pr_X(|D_i p(X)|_2 < \epsilon) > \epsilon^N.
$$
Then we have that
$$
\pr_{X,Y}(|D_i D_Y p(X)|_2 < \epsilon^{1-c}) > \epsilon^{O_{N,c,d}(1)}.
$$
\end{prop}

We begin with the following Lemma:
\begin{lem}\label{SDImpliesLRLem}
Let $N>0$ be a real number and let $d$ and $k$ be positive integers.  Suppose that $A_i(X)$ is a degree-$d$, tensor-valued polynomial so that for some $1/2>\epsilon >0$,
$$
\pr_X(|A_i(X)|_2 < \epsilon) \geq \epsilon^N.
$$
Then the probability over Gaussian $X$ that $|A_i(X)|_2<\epsilon$ and
\begin{align*}
\left|\bigwedge_{i_1\cdots i_k}(D_{j_1}A_{i_1}(X))\cdots (D_{j_k}A_{i_k}(X)) \right|_2 < O_{d,k,N}(\epsilon^{k-N})\log(\epsilon^{-1})^{k}
\end{align*}
is at least $\epsilon^N/2.$
\end{lem}
\begin{proof}
Note that by decreasing $N$, we may assume that
$$
\pr_X(|A_i(X)|_2 < \epsilon) = \epsilon^N.
$$

Note that for any tensor $B_{ij}$
\begin{align*}
\bigwedge_{i_1,\ldots, i_k} B_{i_1j_1}\cdots B_{i_kj_k} & = \sum_{\sigma\in S_k} (-1)^\sigma B_{i_{\sigma(1)}j_1}\cdots B_{i_{\sigma(k)}j_k}\\
& = \sum_{\sigma\in S_k} (-1)^\sigma B_{i_1j_{\sigma^{-1}(1)}}\cdots B_{i_kj_{\sigma^{-1}(k)}}\\
& = \bigwedge_{j_1,\ldots, j_k} B_{i_1j_1}\cdots B_{i_kj_k}.
\end{align*}

Thus, for fixed $X$, we have by Lemma \ref{anticoncentrationLem} that with a probability of at most $1/10$ over $Y^\ell$ we have that
\begin{align*}
\left|\bigwedge_{i_1,\ldots, i_k}(D_{j_1}A_{i_1}(X))\cdots (D_{j_k}A_{i_k}(X)) \right|_2
= & \left|\bigwedge_{j_1,\ldots, j_k}(D_{j_1}A_{i_1}(X))\cdots (D_{j_k}A_{i_k}(X)) \right|_2  \\
> &  \Omega(1/kd)^{kd} \left|Y^1_{i_1}\cdots Y^k_{i_k}\bigwedge_{j_1,\ldots, j_k}(D_{j_1}A_{i_1}(X))\cdots (D_{j_k}A_{i_k}(X)) \right|_2.
\end{align*}
Therefore, it suffices to show that with probability at least $3\epsilon^N/5$ that $|A_i(X)|_2<\epsilon$ and
$$
\left|Y^1_{i_1}\cdots Y^k_{i_k}\bigwedge_{j_1,\ldots, j_k}(D_{j_1}A_{i_1}(X))\cdots (D_{j_k}A_{i_k}(X)) \right|_2< O_{d,k,N}(\epsilon^{k-N})\log(\epsilon^{-1})^{k}.
$$

For fixed $X$, by Corollary \ref{ConcentrationCor} we have that with probability at least $9/10$ that for random $Y^1,\ldots,Y^k$ that $|Y^j_i A_i(X)| < O_k(1) |A_i(X)|_2$ for all $1\leq j \leq k$.  Thus, with probability at least $9\epsilon^N/10$ over $X$ and the $Y^j$, we have that $|Y^j_i A_i(X)| < O_k(\epsilon)$ for all $j$.

On the other hand, Proposition \ref{StrongAnticoncentrationProp} implies that with probability at least $1-\epsilon^N/10$ that
\begin{equation}\label{Z1Eqn}
\left| \bigwedge_{j_1,\ldots,j_k} \prod_{\ell=1}^k D_{j_\ell} Y^\ell_{i_\ell} A_{i_\ell}(X) \right|_2 \leq O_{k,d}(1) \epsilon^{-N} (\log(\epsilon^{-1}))^k \prod_{\ell=1}^k |Y^\ell_i A_i(X)|.
\end{equation}
Recall that with probability at least $9\epsilon^N/10$ we have that $|A_i(X)|_2<\epsilon$ and $|Y^j_i A_i(X)| < O_k(\epsilon)$.  When this holds, the right hand side of Equation (\ref{Z1Eqn}) is at most
$$
O_{k,d}(1) \epsilon^{k-N} \log^k(\epsilon^{-1}).
$$
Hence with probability at least $4\epsilon^N/5$, we have that $|A_i(X)|_2 < \epsilon$ and
$$
\left| \bigwedge_{j_1,\ldots,j_k} \prod_{\ell=1}^k D_{j_\ell} Y^\ell_{i_\ell} A_{i_\ell}(X) \right|_2 < O_{k,d}(1) \epsilon^{k-N} \log^k(\epsilon^{-1}),
$$
as desired.
\end{proof}

Lemma \ref{SDImpliesLRLem} tells us some very strong information about the tensor $D_j A_i(X)$.  In order to understand this better, we will study what it means for a 2-tensor $B_{ij}$ to have $\left|\bigwedge_{i_1,\ldots,i_k} B_{i_1,j_1}\cdots B_{i_k,j_k}\right|_2$ small.  Recall that a 2-tensor can be thought of as a matrix.  We will show that this condition implies that $B_{ij}$ is approximately a matrix of rank at most $k$.

\begin{lem}\label{LowRankLem}
Suppose that $B_{ij}$ is a tensor and suppose that for some integer $k$ and some $\epsilon>0$ that
$$
\left|\bigwedge_{i_1,\ldots,i_k} B_{i_1,j_1}\cdots B_{i_k,j_k}\right|_2 < \epsilon^k.
$$
Then there exist some vectors $V^\ell_i,W^\ell_j$ so that
$$
\left|B_{ij} - \sum_{\ell=1}^{k-1} V^\ell_i W^\ell_j\right|_2 < O_k(\epsilon).
$$
\end{lem}
\begin{proof}
We proceed by induction on $k$.  If $k=1$, we have by assumption that $|B_{ij}|_2<\epsilon$, so we are done.

For larger values of $k$, we may assume that
$$
\left|\bigwedge_{i_1,\ldots,i_{k-1}} B_{i_1,j_1}\cdots B_{i_{k-1},j_{k-1}}\right|_2 \geq \epsilon^{k-1},
$$
or otherwise we would be done by the inductive hypothesis.

Consider random Gaussians $X^1,\ldots,X^k$.  We have that
\begin{align*}
\E\left[\left|\bigwedge_{i_1,\ldots,i_{k-1}} B_{i_1,j_1}\cdots B_{i_{k-1},j_{k-1}}X^1_{j_1}\cdots X^{k-1}_{j_{k-1}}\right|_2^2\right] & =
\left|\bigwedge_{i_1,\ldots,i_{k-1}} B_{i_1,j_1}\cdots B_{i_{k-1},j_{k-1}}\right|_2^2  \geq \epsilon^{2k-2}.
\end{align*}
Similarly,
\begin{align*}
\E\left[\left|\bigwedge_{i_1,\ldots,i_{k}} B_{i_1,j_1}\cdots B_{i_{k},j_{k}}X^1_{j_1}\cdots X^{k}_{j_{k}}\right|_2^2\right] & =
\left|\bigwedge_{i_1,\ldots,i_{k}} B_{i_1,j_1}\cdots B_{i_{k},j_{k}}\right|_2^2  \leq \epsilon^{2k}.
\end{align*}
By Lemma \ref{anticoncentrationLem}, we have that with probability at least $1/2$ that
$$
\left|\bigwedge_{i_1,\ldots,i_{k-1}} B_{i_1,j_1}\cdots B_{i_{k-1},j_{k-1}}X^1_{j_1}\cdots X^{k-1}_{j_{k-1}}\right|_2 \geq \Omega_k(\epsilon^{k-1}).
$$
Furthermore, by the Markov bound, we can find such $X^1,\ldots,X^{k-1}$ so that
$$
\E_{X^k}\left[\left|\bigwedge_{i_1,\ldots,i_{k}} B_{i_1,j_1}\cdots B_{i_{k},j_{k}}X^1_{j_1}\cdots X^{k}_{j_{k}}\right|_2^2\right] \leq 2 \epsilon^{2k}.
$$
Let $V^\ell_i$ be the vector $B_{ij}X^\ell_j$.  We have that
$$
\left| \bigwedge_{i_1,\ldots,i_{k-1}} V^1_{i_1}\cdots V^{k-1}_{i_{k-1}}\right|_2^2 = \Omega_k(\epsilon^{2k-2})
$$
and
$$
\E_{X^k}\left[\left| \bigwedge_{i_1,\ldots,i_{k}} V^1_{i_1}\cdots V^{k}_{i_{k}}\right|_2^2 \right] \leq 2 \epsilon^{2k}.
$$

Notice that the wedge products above are simply standard wedges of vectors.  Note that if we have vectors $u^1,\ldots,u^k$ that
$$
u^1\wedge u^2 \wedge \cdots \wedge u^k = u^1\wedge u^2 \wedge \cdots \wedge u^{k-1} \wedge u^{k,\perp}
$$
where $u^{k\perp}$ is the projection of $u^k$ onto the space perpendicular to $\langle u^1,u^2,\ldots,u^{k-1} \rangle$.  From here, it is easy to see that we have
$$
\frac{|u^1\wedge u^2 \wedge \cdots \wedge u^k|_2}{|u^1\wedge u^2 \wedge \cdots \wedge u^{k-1}|_2} = |u^{k,\perp}|_2.
$$
Therefore, we have that
$$
\E_{X^k}[|V^{k,\perp}_i|_2^2] = O_k(\epsilon^2).
$$
On the other hand, we have that
$$
V^{k,\perp}_i = B^\perp_{ij} X^k_j
$$
where $B^\perp$ is the tensor obtained from $B$ by replacing each row $B_{ij}e_j$ with its projection onto $\langle V^1,V^2,\ldots,V^{k-1}\rangle^\perp$.  In particular, this means that each row of $B^\perp$ can be written as the corresponding row of $B$ plus an element of $\langle V^1,V^2,\ldots,V^{k-1}\rangle$.  This means that for some appropriate vectors $U^\ell$, we have that $B^\perp_{ij} = B_{ij} - \sum_{\ell=1}^{k-1} V^\ell_i W^\ell_j$.  On the other hand, we note that
\begin{align*}
|B^\perp|_2^2 & = \E[|B^\perp_{ij} X_j|_2^2]\\
& = \E[|V^\perp_i|_2^2]\\
& = O_k(\epsilon^2).
\end{align*}
Thus, $|B^\perp|_2 = O_k(\epsilon)$, completing our proof.
\end{proof}

We are now prepared to prove Proposition \ref{SmallDerProp}.
\begin{proof}
Suppose we are given $c,d,N$ and $\epsilon>0$ sufficiently small.  Suppose that we have a degree-$d$ polynomial $P$ so that
$$
\pr_X(|D_ip(X)|_2 < \epsilon) > \epsilon^N.
$$
Note that by Lemma \ref{anticoncentrationLem} that this implies that $\E_X[|D_i p(X)|_2^2] \leq O_{d} (\epsilon^{-2dN})$.  And hence that $\E_X[|D_i D_j p(X)|_2^2] \leq O_d(\epsilon^{-2dN}).$

Let $k$ be an integer so that $k > 2N/ c$.  By Lemma \ref{SDImpliesLRLem} applied to $D_ip(X)$ we have that with probability at least $\epsilon^N/2$ that
$$
\left| \bigwedge_{i_1,\ldots,i_k} \prod_{\ell=1}^k D_{i_\ell}D_{j_\ell} p(X) \right|_2 < O_{c,d,N} (\epsilon^{k(1-c/2)}).
$$
Let $B_{ij}(X)$ be the tensor $D_i D_j p(X)$.  By the above and Corollary \ref{ConcentrationCor} we have that with probability at least $\epsilon^N/3$ over $X$ that $|B(X)|_2 < O_d(\epsilon^{-2dN})$ and
$$
\left| \bigwedge_{i_1,\ldots,i_k} \prod_{\ell=1}^k B_{i_\ell j_\ell}  \right|_2 < O_{c,d,N} (\epsilon^{k(1-c/2)}).
$$

Applying Lemma \ref{LowRankLem} to $B$ at such values of $X$, we have that there are vectors $V^\ell,W^\ell$ so that
$$
\left|B_{ij} - \sum_{\ell=1}^{k-1} V^\ell_i W^\ell_j\right|_2 = O_{c,d,N}(\epsilon^{1-c/2}).
$$
We note that we can replace the $V^\ell$ in such a decomposition with an orthonormal basis for the space that they span by adjusting the $W^\ell$ accordingly.  We then have that
\begin{align*}
\sum_{\ell=1}^k |W^\ell_j |_2^2 & = \left|\sum_{\ell=1}^{k-1} V^\ell_i W^\ell_i\right|_2^2\\
& \leq (|B|_2 + O_{c,d,N}(1))^2 \\
& \leq O_{c,d,N}(\epsilon^{-4dN}).
\end{align*}
Therefore $|W^\ell_j|_2 \leq O_{c,d,N}(\epsilon^{-2dN})$ for each $\ell$.

Now given a random Gaussian vector $Y$, there is a probability of at least $\Omega_{k}(\epsilon^{2dkN+k})$ that $|Y_i V^\ell_i| \leq \epsilon^{2dN+1}$ for each $\ell$.  Furthermore, by Lemma \ref{ConcentrationCor}, for $\epsilon$ sufficiently small the probability that
$$
\left|(B_{ij} - \sum_{\ell=1}^{k-1} V^\ell_i W^\ell_j)Y_i\right|_2 < \epsilon^{1-3c/4}
$$
is much less than this.  Hence for such $X$ (which occur with probability at least $\epsilon^{N}/2$), there is a probability of at least $\epsilon^{O_{c,d,N}(1)}$ over $Y$ that
$$
\left|(B_{ij} - \sum_{\ell=1}^{k-1} V^\ell_i W^\ell_j)Y_i\right|_2 < \epsilon^{1-3c/4}
$$
and
$$
|Y_i V^\ell_i| \leq \epsilon^{2dN+1}
$$
for each $\ell$.  The latter implies that $|Y_i V^\ell_i W^\ell_j|_2 \leq \epsilon$ for each $\ell$, and thus,\begin{align*}
|B_{ij}Y_i|_2 & \leq \left|(B_{ij} - \sum_{\ell=1}^{k-1} V^\ell_i W^\ell_j)Y_i\right|_2 + \sum_{\ell=1}^{k-1} |Y_i V^\ell_i W^\ell_j|_2 < \epsilon^{1-c}.
\end{align*}
Thus, with probability at least $\epsilon^{O_{c,d,N}(1)}$,
$$
|D_i D_Y p(X)|_2 < \epsilon^{1-c}.
$$
\end{proof}

Iterating Proposition \ref{SmallDerProp} will tell us that a polynomial with a reasonable chance of having a small derivative will also have partial higher order derivatives that are small.  Considering the $d^{\tth}$ order derivatives, this reduces to a statement about the rank-$d$ tensor corresponding to our polynomial.  We would like to claim that such tensors can be approximately decomposed as a sum of products of lower rank tensors.  In order to conveniently talk about such products we introduce some notation.  If $S=\{a_1,\ldots,a_k\}$ is a set of natural numbers, we let $U_{i_S}$ denote a tensor on the indices $i_{a_1},i_{a_2},\ldots,i_{a_k}$.

\begin{prop}\label{TensorDecompositionProp}
Let $d$ be an integer, and let $c,N,\epsilon>0$ be real numbers.  Then for all rank-$d$ tensors $A$ with $|A|_2\leq 1$ and
$$
\pr_{X^1,\ldots,X^{d-1}}(|A_{i_1,\ldots,i_d}X^1_{i_1}\cdots X^{d-1}_{i_{d-1}}|_2<\epsilon) > \epsilon^N
$$
Then there exist tensors $U^\ell,V^\ell$, $1\leq \ell \leq k =O_{c,d,N}(1)$ and sets $$\emptyset \subsetneq S(\ell) \subsetneq \{1,2,\ldots,d\},\overline{S(\ell)} = \{1,2,\ldots,d\} - S(\ell)$$ such that $|U^\ell|_2|V^\ell|_2 \leq O_{c,d,N}(|A|_2\epsilon^{-c})$ for all $\ell$ and
$$
\left| A_{i_1\ldots i_d} - \sum_{\ell=1}^k U^\ell_{i_{S(\ell)}} V^\ell_{i_{\overline{S(\ell)}}}\right|_2 = O_{c,d,N}(\epsilon^{1-c}).
$$
\end{prop}
\begin{proof}
We note that it suffices to prove this result for $|A|_2=1$, since the general case would follow from applying this specialized result to $A/|A|_2$.

We will instead prove the stronger claim that given $c,d,N,\epsilon$ that there exists a probability distribution over sequences of tensor-valued polynomials $U^\ell$,$V^\ell$ of degree $O_{c,d,N}(1)$ in the coefficients of $A$, so that for any tensor $A$ satisfying the hypothesis of the proposition that with probability at least $\epsilon^{O_{c,d,N}(1)}$ over our choice of $U^\ell,V^\ell$ in this family that
$$
|U^\ell(A)|_2,|V^\ell(A)|_2 \leq O_{c,d,N}(\epsilon^{-c})
$$
for all $\ell$, and
$$
\left| A_{i_1\ldots i_d} - \sum_{\ell=1}^k U^\ell_{i_{S(\ell)}}(A) V^\ell_{i_{\overline{S(\ell)}}}(A)\right|_2 = O_{c,d,N}(\epsilon^{1-c}).
$$
Given this statement, our proposition can be recovered by picking an appropriate set of $U^\ell$ and $V^\ell$ for our $A$.  We assume throughout this proof that $\epsilon$ is at most a sufficiently small function of $c,d$ and $N$, since otherwise there would be nothing to prove.

We prove this statement by induction on $d$.  For $d=1$, we already have that $|A_{i_1}|_2<\epsilon$, and there is nothing to prove.  Hence we assume that our statement holds for rank-$(d-1)$ tensors.  The basic idea of our proof will be as follows.  By assumption with reasonable probability over $X$, $AX$ will satisfy the inductive hypothesis for a rank-$(d-1)$ tensor.  This means that we can write $U^\ell$ and $V^\ell$ as polynomials in $X$ so that with reasonable probability over $X$, $|AX-\sum U^\ell(X) V^\ell(X)|_2$ is small.  Applying Lemmas \ref{SDImpliesLRLem} and \ref{LowRankLem}, we can show that the derivative of this tensor with respect to $X$ is approximately low-rank.  This means that the tensor
$$
A - \sum_\ell (D_{i_1} U^\ell(X)) V^\ell(X) + U^\ell(X) (D_{i_1} V^\ell(X))
$$
is approximated by a small sum of products of rank-$1$ tensors with rank-$(d-1)$ tensors.  By making some random guesses, these remaining tensors can be written as polynomials in the coefficients of $A$ with reasonable probability.

Suppose that $A$ is a rank-$d$ tensor satisfying the hypothesis of our proposition.  Then with probability at least $\epsilon^{N}$ over a choice of $X^1$, there is a probability of at least $\epsilon^N$ over our choice of $X^2,\ldots,X^{d-1}$ that
$$
|A_{i_1,\ldots,i_d}X^1_{i_1}\cdots X^{d-1}_{i_{d-1}}|_2<\epsilon.
$$
Furthermore, by Corollary \ref{ConcentrationCor}, with probability at least $1-\epsilon^N/2$ we have that $|A_{i_1,\ldots,i_d}X^1_{i_1}|_2 < \epsilon^{c/20}$.  Hence with probability at least $\epsilon^N/2$ over our choice of $X^1$, $\epsilon^{c/20}A_{i_1,\ldots,i_d}X^1_{i_1}$ satisfies the hypotheses of our proposition as a rank-$(d-1)$ tensor.  For each such $X^1$, the induction hypothesis implies that there is a probability of $\epsilon^{O_{c,d,N}(1)}$ over our choice of $U^\ell$,$V^\ell$ that the appropriate conclusion holds.  Therefore, there must be some particular choice of $U^\ell,V^\ell$ so that with probability at least $\epsilon^{O_{c,d,N}(1)}$ over our choice of $X^1$ we have that
$$
|U^\ell(\epsilon^{c/20}AX^1_{i_1})|_2,|V^\ell(\epsilon^{c/20}AX^1_{i_1})|_2 \leq O_{c,d,N}(\epsilon^{-c/20})
$$
and
$$
\left|\epsilon^{c/20}AX^1_{i_1} - \sum_{\ell=1}^k U^\ell_{i_{S(\ell)}}(\epsilon^{c/20}AX^1_{i_1}) V^\ell_{i_{\overline{S(\ell)}}}(\epsilon^{c/20}AX^1_{i_1})\right|_2 = O_{c,d,N}(\epsilon^{1-c/20}).
$$
Letting ${U'}^\ell(X^1):= \epsilon^{-c/40}U^\ell(\epsilon^{c/20}AX^1)$ and ${V'}^\ell(X^1):= \epsilon^{-c/40}V^\ell(\epsilon^{c/20}AX^1)$, we can rephrase the last two equations as
$$
|{U'}^\ell(X^1)|_2,|{V'}^\ell(X')|_2 \leq O_{c,d,N}(\epsilon^{-c/10})
$$
and
$$
\left|AX^1_{i_1} - \sum_{\ell=1}^k {U'}^\ell_{i_{S(\ell)}}(X^1) {V'}^\ell_{i_{\overline{S(\ell)}}}(X^1)\right|_2 = O_{c,d,N}(\epsilon^{1-c/10}).
$$
We will demonstrate that given a correct choice of such ${U'}^\ell$ and ${V'}^\ell$ we can construct new polynomials $U^\ell(A)$, $V^\ell(A)$ that satisfy the necessary conditions with probability at least $\epsilon^{O_{c,d,N}(1)}$.

Let $T_i(X^1)$ be the tensor-valued polynomial whose coefficients are the concatenation of the coefficients of
$$
AX^1_{i_1} - \sum_{\ell=1}^k {U'}^\ell_{i_{S(\ell)}}(X^1) {V'}^\ell_{i_{\overline{S(\ell)}}}(X^1)
$$
and the coefficients of the $\epsilon {U'}^\ell(X^1)$ and $\epsilon {V'}^\ell(X^1)$.  We have that for some $N_1 = O_{c,d,N}(1)$ that with probability at least $\epsilon^{N_1}$ that $|T_i(X^1)|_2 < O_{c,d,N}(\epsilon^{1-c/10})$.  We apply Lemma \ref{SDImpliesLRLem} with $k' > 10 N_1/c$ and then Lemma \ref{LowRankLem} (as in the proof of Proposition \ref{SmallDerProp}) to show that there exist tensors $W^\ell,Z^\ell$ so that
$$
\left| D_j T_i(X^1) - \sum_{\ell=1}^{k'-1} W^\ell_i Z^\ell_j \right|_2 \leq O_{c,d,N}(\epsilon^{1-3c/20}).
$$
We alter $W^\ell$ and $Z^\ell$ to maintain the same sum $\sum_{\ell=1}^{k'-1} W^\ell_i Z^\ell_j$.  This sum can be thought of as a rank $k'-1$ matrix.  Note that by the theory of singular values it can always be expressed in the form $\sum_{\ell=1}^{k'-1} C^\ell W^\ell_i Z^\ell_j$ where $C^\ell$ are positive real numbers and $\{W^\ell\},\{Z^\ell\}$ are orthonormal sets.  Note that the $|C_\ell|$ are no more than $|D_j T_i(X^1)|_2$.  The expectation of this (over $X^1$) is bounded in terms of the sizes of the ${U'}^\ell$ and ${V'}^\ell$.  These in turn can be no larger than $\epsilon^{-O_{d,c,N}(1)}$ by Lemma \ref{anticoncentrationLem} since they are small with reasonable probability.  Thus, by Corollary \ref{ConcentrationCor}, with probability at least $\epsilon^{O_{c,d,N}(1)}$ over our choice of $X^1$, we have
$$
\left| D_j T_i(X^1) - \sum_{\ell=1}^{k'-1} C^\ell W^\ell_i Z^\ell_j \right|\leq O_{c,d,N}(\epsilon^{1-3c/20})
$$
with $\{W^\ell\},\{Z^\ell\}$ are orthonormal sets, and  $C_\ell<\epsilon^{-N_2}$ for all $\ell$ and some $N_2=O_{c,d,N}(1)$.  Furthermore, we may assume that each $C_\ell$ is at least $\epsilon$, since otherwise we could remove the corresponding term in $\sum_{\ell=1}^{k'-1} C^\ell W^\ell_i Z^\ell_j$ without affecting the required properties.

Let
$$
S_{ij} = D_j T_i(X^1) - \sum_{\ell=1}^{k'-1} C^\ell W^\ell_i Z^\ell_j.
$$
Note that if $S_{ij}W^\ell_i$ is non-zero for some $i$, we can add $S_{ij}W^\ell_i$ to $Z^\ell_j$, obtaining a new decomposition of the form specified above with $|S_{ij}|_2$ smaller than it was before.  By taking $|S_{ij}|$ minimal, we can assume that $S_{ij}W^\ell_i=0$ and similarly that $S_{ij}Z^\ell_j=0$ for all $\ell$.

Let $M$ be a sufficiently large constant depending only on $c,d,N$.  Let ${C'}^\ell$ be random real numbers in the range $[\epsilon,\epsilon^{-N_2}]$ and let $Y^\ell$ be random vectors for $1\leq \ell < k'$.  We have that with probability at least $\epsilon^{O_{c,d,N}(1)}$ that for all $a,b$ that $|Y^{a}_i Z^b_i - \delta_{a,b}|< O(\epsilon^M)$, $|S_{ij}Y^\ell_j|_2 \leq O_{c,d,N}(\epsilon^{1-c/5})$, $|AY^\ell_{i_1}|_2\leq O_{c,d,N}(\epsilon^{-c/20})$, and $|{C'}^\ell-C^\ell| < \epsilon^M$.  We construct polynomials $U$ and $V$ (depending on $Y^\ell$ and $C^\ell$) that have the desired properties when these inequalities hold.

We think of $A$ as a linear function that takes a vector $X_{i_1}$ and returns a tensor on the remaining $d-1$ coordinates.
We let $Y_{i_1}$ be the vector
$$
Y_{i_1} = X_{i_1} - \sum_{\ell=1}^{k'-1} Y^\ell_{i_1}( D_X T_j(X^1) ) (D_{Y^\ell} T_j (X^1)) / ({C'}^\ell)^2.
$$
Note that
$$
D_X T_i(X^1) = S_{ij}X_j  + \sum_{\ell=1}^{k'-1} C^\ell W^\ell_i X_j Z^\ell_j.
$$
Similarly,
\begin{align*}
D_{Y^\ell} T_i(X^1) & = S_{ij}Y^\ell_j + \sum_{l=1}^{k'-1} C^l W^l_i Y^{\ell}_j Z^l_j \\ & = C^\ell W^\ell_i + O_{c,d,N}(\epsilon^{1-c/5}) +O_{c,d,N}(\epsilon^{M-N_2}).
\end{align*}
Thus, for $M$ sufficiently large, with high probability over $X$ we have that
$$
( D_X T_j(X^1) ) (D_{Y^\ell} T_j (X^1)) = (C^\ell)^2 X_i W^\ell_i + O_{c,d,N}(\epsilon^{2-2c/5}).
$$
Also, with high probability over $X$ this is at most
\begin{equation}\label{DotPBoundEqn}
({C'}^\ell)^2O_{c,d,N}(\epsilon^{-2c/5}).
\end{equation}
On the other hand,
\begin{align*}
Y_i Z^\ell_i & = X_i Z^\ell_i - \sum_{l=1}^{k'-1} Y^l_{i_1}( D_X T_j(X^1) ) (D_{Y^l} T_j (X^1)) / ({C'}^l)^2\\
& = X_i Z^\ell_i - \sum_{l\neq \ell} O(\epsilon^{M-2N_2-2}) - ((C^\ell)^2 X_i W^\ell_i + O_{c,d,N}(\epsilon^{2-2c/5}))/({C'}^\ell)^2\\
& = O_{c,d,N}(\epsilon^{M-2N_2-2}) + X_i W^\ell_i O_{c,d,N}(\epsilon^{M-2})+O(\epsilon^{2-2c/5}/(C^\ell)^2).
\end{align*}
Hence for $M$ sufficiently large, with high probability over $X$ we have that
$$
|C^\ell Y_i Z^\ell_i| = O_{c,d,N}(\epsilon^{1-2c/5})
$$
for all $\ell$.  If this holds, then
\begin{align*}
D_Y T_i(X^1) & = S_{ij}Y_j + \sum_{\ell=1}^{k'-1} W^\ell_j C^\ell Y_i Z^\ell\\
& = S_{ij}X_j + \sum_{\ell=1}^{k'-1} O_{c,d,N}(\epsilon^{-2c/5}) S_{ij}Y^\ell_j + O_{c,d,N}(\epsilon^{1-2c/5})\\
& = S_{ij}X_j + O_{c,d,N}(\epsilon^{1-3c/5}).
\end{align*}
With high probability over $X$, this is at most $O_{c,d,N}(\epsilon^{1-3c/5})$.  If this is the case, it means that $|D_Y {U'}^\ell|_2, |D_Y {V'}^\ell|_2 < O_{c,d,N}(\epsilon^{-3c/5})$ and that
\begin{align*}
& \left|A_{i_1,\ldots,i_d}Y_{i_1} -\left( \sum_{\ell=1}^k (D_Y {U'}^\ell_{S(\ell)}(X^1)) {V'}^\ell_{\overline{S(\ell)}}(X^1) +  {U'}^\ell_{S(\ell)}(X^1) (D_Y {V'}^\ell_{\overline{S(\ell)}}(X^1))\right) \right|_2  < O_{c,d,N}(\epsilon^{1-3c/5})
\end{align*}

Note also that
$$
AX_{i_1} - AY_{i_1} = \sum_{\ell=1}^{k'-1} AY^\ell_{i_1} ( D_X T_j(X^1) ) (D_{Y^\ell} T_j (X^1)) / ({C'}^\ell)^2.
$$
Thus, we may approximate $AX_{i_1}$ by
\begin{align*}
\sum_{\ell=1}^k (D_Y {U'}^\ell_{S(\ell)}(X^1)) {V'}^\ell_{\overline{S(\ell)}}(X^1) &+  {U'}^\ell_{S(\ell)}(X^1) (D_Y {V'}^\ell_{\overline{S(\ell)}}(X^1))\\ & +
\sum_{\ell=1}^{k'-1} AY^\ell_{i_1} ( D_X T_j(X^1) ) (D_{Y^\ell} T_j (X^1)) / ({C'}^\ell)^2
\end{align*}
and with high probability over $X$ the error is $O_{c,d,N}(\epsilon^{1-3c/5})$.  On the other hand it should be noted that the above is linear in $X$ and thus,may be thought of as a rank-$d$ tensor, $B$, applied to $X$ at one coordinate.  We have that with high probability over $X$ that
$$
|AX-BX|_2 = O_{c,d,N}(\epsilon^{1-3c/5}).
$$
Thus, by Lemma \ref{anticoncentrationLem}, we have that
$$
|A-B|_2 = O_{c,d,N}(\epsilon^{1-3c/5}).
$$
Furthermore, the tensor $B$ is obviously given as a sum of products of pairs of lower-rank tensors on appropriate subsets of the coordinates.  In order to complete our proof we need to show that these lower-rank tensors have size at most $O_{c,d,N}(\epsilon^{-c})$.  We already know that ${V'}^\ell_{\overline{S(\ell)}}(X^1)$, ${U'}^\ell_{S(\ell)}(X^1)$ and $AY^\ell_{i_1}$ are appropriately bounded.  The other tensors are expressed implicitly as linear functions in $X$ with tensor valued outputs. By Lemma \ref{anticoncentrationLem}, it suffices to show for these tensors that with high probability over $X$ that the output is $O_{c,d,N}(\epsilon^{-c})$.  This holds for $D_Y {U'}^\ell_{S(\ell)}(X^1)$ and $D_Y {V'}^\ell_{\overline{S(\ell)}}(X^1)$ since with high probability $S_{ij}Y_j$ is small.  It holds for $( D_X T_j(X^1) ) (D_{Y^\ell} T_j (X^1)) / ({C'}^\ell)^2$ by Equation (\ref{DotPBoundEqn}).

\end{proof}

We are finally ready to prove Proposition \ref{DecompositionProp}.
\begin{proof}
Assume that $\epsilon$ is sufficiently small as a function of $c,d$ and $N$ (for otherwise there is nothing to prove).

Consider such a polynomial $p$.  We claim that for any $k<d$ and any $c'>0$ that
$$
\pr_{X,Y^1,\ldots,Y^k}(|D_i D_{Y^1}\cdots D_{Y^k}p(X)|_2 < \epsilon^{1-c'}) > \epsilon^{O_{k,c',d,N}(1)}.
$$
This is proved by induction on $k$.  The $k=0$ case is given and the inductive step follows immediately from Proposition \ref{SDImpliesLRLem}.  Applying this statement for $k=d-1$, we note that
$$
D_i D_{Y^1}\cdots D_{Y^k}p(X)
$$
is independent of $X$.  Let $A_{i_1,\ldots,i_d} = D_{i_1}\cdots D_{i_d} p(X)$ be the symmetric $d$-tensor associated to $p$.  We have by Lemma \ref{DerivativeSizeLem} that $|A|_2 = \sqrt{d!}|p^{[d]}|_2 \leq \sqrt{d!}$, and thus, $A/d!$ satisfies the hypothesis of Proposition \ref{TensorDecompositionProp}. Hence we can find tensors $U^\ell$ and $V^\ell$ with the properties specified by that proposition so that $|U^\ell|_2|V^\ell|_2 \leq O_{c,d,N}(\epsilon^{-c})|p^{[d]}|_2$.  Since $A$ is symmetric, we have that
\begin{equation}\label{TDecompSymEqn}
\left|A - \sum_{\sigma \in S_d} \sum_{\ell=1}^k U^{\ell}_{i_{\sigma(S(\ell))}}V^{\ell}_{i_{\sigma(\overline{S(\ell)})}} \right|_2 = O_{c,d,N}(\epsilon^{1-c}).
\end{equation}
Note that in the above since the sum over $\sigma$ has already added the permutations of $U^\ell$ over its indices, we may replace $U^\ell$ and $V^\ell$ by their symmetrizations without affecting the above sum.  Let $U^\ell$ be rank $d_\ell$ and $V^\ell$ be rank $d-d_\ell$.  Let $a_\ell(X)$ be the degree-$d_\ell$ harmonic part of  the polynomial $X\rightarrow U^\ell(X,X,\ldots,X)$.  Define $b_\ell(X)$ similarly with respect to $V^\ell$.  By Lemma \ref{DerivativeSizeLem} we have that $|a_\ell|_2|b_\ell|_2\leq d!|U^\ell|_2|V^\ell|_2=O_{c,d,N}(\epsilon^{-c})|p^{[d]}|_2.$  Now consider the tensor given by
$$
D_{i_1}D_{i_2}\cdots D_{i_d} \left[p(X) - \sum_{\ell=1}^k a_\ell(X) b_\ell(X) \right].
$$
This is easily seen to be the tensor given in Equation (\ref{TDecompSymEqn}), and hence has size $O_{c,d,N}(\epsilon^{1-c})$.  On the other hand by Lemma \ref{DerivativeSizeLem}, this can be seen to be $\sqrt{d!}$ times the size of the degree-$d$ harmonic part of the polynomial
$$
p(X) - \sum_{\ell=1}^k a_\ell(X) b_\ell(X).
$$
This completes our proof.
\end{proof}

\subsection{Proof of the Main Theorem}\label{DDpfSec}

We are now prepared to prove the Diffuse Decomposition Theorem.  The basic idea of the proof is fairly simple.  We maintain decompositions of polynomials approximately equal to $p$.  We show using Proposition \ref{DecompositionProp} that if this decomposition is not diffuse that we can replace it by a simpler one by introducing at most a small error.  This new decomposition is simpler in the sense that an associated ordinal number is smaller, and we will use transfinite induction to prove that this process will eventually terminate, yielding an appropriate decomposition.

\begin{proof}[Proof of Theorem \ref{DDTheorem}]
We assume for convenience that $N$ and $c^{-1}$ are integers.  Throughout we will assume that $N,c,d$ and $\epsilon$ are fixed.

We define a \emph{partial decomposition} of our polynomial $p$ to be a set of the following data:
\begin{itemize}
\item A positive integer $m$
\item A polynomial $h:\R^m\rightarrow \R$
\item A sequence of polynomials $(q_1,\ldots,q_m)$ each on $\R^n$ with $|q_i|_2=1$ for each $i$
\item A sequence of integers $(a_1,\ldots,a_m)$ with $a_i$ between 0 and $4\cdot 3^i (N+1) / c - 1$.
\end{itemize}
Furthermore, we require that each $q_i$ is non-constant, and that for any monomial $\prod x_{i}^{\alpha_1}$ appearing in $h$ that $\sum \alpha_i \deg(q_i) \leq d$.

We say that such a partial decomposition has complexity at most $C$ if the following hold:
\begin{itemize}
\item $m \leq C$
\item $|h|_2 \leq C \epsilon^{-1+C^{-1}}$
\item $|p(X) - h(\epsilon^{a_1 c/(2\cdot 3^1)}q_1(X),\epsilon^{a_2 c/(2\cdot 3^2)}q_2(X),\ldots,\epsilon^{a_m c/(2\cdot 3^m)}q_m(X))|_2 \leq C\epsilon^N$
\end{itemize}

Finally, we define the weight of a partial decomposition as follows.  First we define the polynomial
$$
w(x)=\sum_{i=1}^m x^{\deg(q_i)} (4 \cdot 3^i (N+1)/c - a_i).
$$
We then let the weight of the decomposition be $w(\omega)$.

Our result will follow from the following Lemma:
\begin{lem}\label{Z2Lem}
Let $p$ be a degree-$d$ polynomial with a partial decomposition of weight $w$ and complexity at most $C$.  Then there exists a polynomial $p_0$ with an $(\epsilon,\epsilon^{-c})$-diffuse decomposition of size at most $O_{c,d,N,w,C}(1)$ so that $|p-p_0|_2 \leq O_{c,d,N,w,C}(\epsilon^N)$.
\end{lem}
\begin{proof}
We prove this by transfinite induction on $w$.  In particular, we show that either $(h,q_1,\ldots,q_m)$ provides an appropriate diffuse decomposition of a polynomial approximately equal to $p$ or that $p$ has another partial decomposition of complexity $O_{c,d,N,C,w}(1)$ and weight strictly less than $w$ (with finitely many possibilities for the new weight).  The inductive hypothesis will imply that we have an appropriate diffuse decomposition in the latter case.

First note that if some $a_i$ at least $2(N+1) 3^i /c$ that the sum of the coefficients of $q_i$ appearing in $h(\epsilon^{a_1 c/(2\cdot 3^1)}q_1(X),\epsilon^{a_2 c/(2\cdot 3^2)}q_2(X),\ldots,\epsilon^{a_m c/(2\cdot 3^m)}q_m(X))$ is $O_C(\epsilon^N)$.  Thus, these terms can be thrown away without introducing an error of more than $O_{C,d}(\epsilon^N)$.  Doing so to the largest such $q_i$ and shifting all of the larger indices down, perhaps changing the $a_i$, and modifying $h$ appropriately will lead to a new partial decomposition with a new value of $C$ dependent only on $d$ and the old one, and a strictly smaller weight.
Hence we assume that $a_i < 2(N+1) 3^i /c$ for all $i$.

If $\deg(q_{i+1}) > \deg(q_i)$ for some $i$, we may swap $q_i$ and $q_{i+1}$ (making a similar adjustment to $h$ and setting $a_i$ and $a_{i+1}$ to 0) to get a partial decomposition of complexity $C$ and strictly smaller weight.  Hence we may assume that $\deg(q_1)\geq \deg(q_2)\geq \ldots \geq \deg(q_m)$.

Were it the case that for all $x_1,\ldots,x_m$ that
$$
\pr(|q_i(X)-x_i| <\epsilon \textrm{ for all } i) < \epsilon^{m-c},
$$
then we would already have an appropriate diffuse decomposition and would be done.  Hence we may assume that there is a set of $x_i$ so that the above does not hold.  By Proposition \ref{StrongAnticoncentrationProp}, we have that with probability at least $1-\epsilon^{m-c}/2$ that
$$
\prod_{i=1}^m |q_i(X)-x_i| \geq \Omega_{C,d}(\epsilon^{m-c/2})\left|\bigwedge_{j_1,\ldots,j_m} \prod_{i=1}^m D_{j_i} q_i(X) \right|_2.
$$
Thus, with probability at least $\epsilon^{m-c}/2$ both of the above hold, which would imply that
$$
\left|\bigwedge_{j_1,\ldots,j_m} \prod_{i=1}^m D_{j_i} q_i(X) \right|_2 = O_{C,d}(\epsilon^{c/2}).
$$
Now the wedge product above is a wedge product of vectors, and hence its size is unchanged by making a determinant 1 change of basis to the vectors $D_{j_i}q_i$.  Hence, letting $V^i$ be the projection of $D q_i$ onto the orthogonal compliment of the space spanned by the $D q_j$ for $j>i$ we have that the size of the wedge product equals $\prod_{i=1}^m |V^i|_2$.  This means that for some $i$ that $|V^i|_2 \leq O_{C,d} (\epsilon^{c/3^i})$.  Hence for some $i$, we have with probability at least $\Omega_{C,d}(\epsilon^m)$ over $X$ that $|V^i(X)|_2  \leq O_{C,d} (\epsilon^{c/3^i})$, and that this is the largest $i$ for that $X$ for which this holds.  Furthermore, by Lemma \ref{DerivativeSizeLem} and Corollary \ref{ConcentrationCor} we know that when this happens with high probability we also have that the first derivatives of all the $q_i$ have size $O_{C,d}(\log(\epsilon^{-1})^d)$.

When this happens $V^j$ is given by the derivative of $q_j$ minus an appropriate linear combination of the $V^k$ for $k>j$.   Note that for each coefficient, the size of the coefficient times the size of $V^k$ is at most the size of the derivative of $q_j$.  Hence for $k>i$, the size of the coefficient is at most $O_{C,d}(\epsilon^{-c/3^k} \log^d(\epsilon^{-1})).$  From this it is easy to see that $V^i$ is given by a linear combination of the derivatives of the $q_j$ with $j\geq i$ such that the $i^{\tth}$ coefficient is 1 and that all other coefficients have size at most
$$
\prod_{k=i+1}^m O_{C,d}(\epsilon^{-c/3^k} \log^d(\epsilon^{-1})) = O_{C,d}( \epsilon^{-c/(2\cdot 3^i) + c/(2\cdot 3^m)}).
$$
Hence for each such $X$, there are constants $C_j=O_{C,d}( \epsilon^{-c/(2\cdot 3^i) + c/(2\cdot 3^m)})$ (for $j>i$) so that the derivative of $q_i + \sum_j C_j q_j$ at $X$ has size at most $O_{C,d} (\epsilon^{c/3^i}).$  Note that this statement still holds if the $C_j$ are rounded to the nearest multiple of $\epsilon$.  Since there are $\epsilon^{-O_C(1)}$ such possible roundings, there is some set of $C_j$ so that for the polynomial $q(X) = q_i(X)+ \sum_j C_j q_j(X)$, we have that $|D_iq(X)|_2 \leq O_{C,d} (\epsilon^{c/3^i})$ with probability at least $\epsilon^{O_{C,d}(1)}$ over $X$.

We now can apply Proposition \ref{DecompositionProp} to $\Omega_{C,d}(\epsilon^{c/(2\cdot 3^i) - c/(4\cdot 3^m)}) q(X)$.  Let $D=\deg(q_i)$.  Let $Q(X)$ be the degree-$D$ harmonic part of $\Omega_{C,d}(\epsilon^{c/(2\cdot 3^i) - c/(2\cdot 3^m)}) q(X)$.  Proposition \ref{DecompositionProp} tells us that there are polynomials $A_\ell,B_\ell$ of degree strictly less than $D$ with $|A_\ell|_2 |B_\ell|_2$ at most $O_{c,C,d}(\epsilon^{-1/(2C)})|Q|_2$ for each $\ell$, and so that
$Q-\sum_\ell A_\ell B_\ell $ equals a polynomial of degree less than $D$ plus an error of $L^2$ norm at most $O_{c,C,d}(\epsilon^{c/3^i -c/(2\cdot 3^m)})$.  Note that the lower degree polynomial has size at most
$$
|Q|_2 + \sum |A_\ell B_\ell|_2.
$$
By Corollary \ref{ConcentrationCor} and H\"{o}lder's inequality we have that
$$
|A_\ell B_\ell|_2 \leq |A_\ell|_4 |B_\ell|_4 \leq O_{d}(|A_\ell|_2 |B_\ell|_2) = O_{c,C,d}(|Q|_2 \epsilon^{-1/(2C)}).
$$

Consider the $j$ among those for which $\deg(q_j)=D$ for which $C_j \epsilon^{-a_j c/(2\cdot 3^j)}$ is the largest.  $Q(X)$ is then some multiple of the degree-$D$ harmonic part of $q_j \epsilon^{a_j c/(2\cdot 3^j)}$ plus smaller multiples of the degree-$D$ harmonic parts of other $q_k \epsilon^{a_k c/(2\cdot 3^k)}$.

We are now ready to modify our partial decomposition to obtain one of smaller weight.  First we take each of the $q_i$ of degree equal to $D$ and replace $q_i$ by the sum of its harmonic degree-$D$ part and the remainder, introducing each as a new $q_j$.  This increases the complexity by at most a factor of 2, and increases the weight by an ordinal less than $\omega^D$.

Next we note that $q_j \epsilon^{a_j c/(2\cdot 3^j)}$ can be written as a linear combination of the other $q_k \epsilon^{a_k c/(2\cdot 3^k)}$ (with coefficients less than 1) plus a sum of $A_\ell(X)B_\ell(X)$ plus a polynomial of degree less than $D$ plus a degree-$D$ polynomial of size at most $O_{c,C,d}(\epsilon^{(a_j+1) c/(2\cdot 3^j)}) $.  Replacing $q_j$ by a normalized version of this error polynomial, and adding new $q$'s corresponding to the normalized versions of $A_\ell$ and $B_\ell$ and the remaining part of degree less than $D$ and modifying $h$ appropriately, we find that we have a new partial decomposition of complexity $O_{c,C,d}(1)$ and weight smaller by $\omega^D-O_{c,C,d}(\omega^{D-1}).$

One thing that needs to be verified about this construction is that the norm of $h$ is not increased by too much.  We may think of $h$ as a polynomial whose input variables are the $q_i \epsilon^{a_i c/(2\cdot 3^i)}$.  Replacing relevant $q_i$ by the sum of their highest harmonic component plus a lower degree part, merely replaces one of the inputs to $h$ by the sum of two new input variables, and thus, does not increase the size of $h$ by more than a constant factor.  The other step is somewhat more complicated to analyze.  Here we replace the single input variable $q_j \epsilon^{a_j c/(2\cdot 3^j)}$ by a sum of the following:
\begin{itemize}
\item A sum of other variables $q_i \epsilon^{a_i c/(2\cdot 3^i)}$ with coefficients at most 1
\item A new variable corresponding to the error term, with coefficient $O_{c,C,d}(1)$
\item A sum of new terms corresponding to the $A_\ell B_\ell$
\end{itemize}
We claim that the sum of the coefficients of the new terms in $h$ replacing the variable $q_j \epsilon^{a_j c/(2\cdot 3^j)}$ are at most $O_{c,C,d}(\epsilon^{-1/(2C)})$, thus, allowing the complexity to increase by only a bounded amount.  This is clear for the first two contributing factors above.  For the latter two, note that
$$
|Q|_2 = O_{C}(\max\{C_i: \deg(q_i)=D\}) \leq O_{C}(C_j\epsilon^{-a_j c/(2\cdot 3^j)}).
$$
Thus, after rescaling $Q$ so that the coefficient of $q_j \epsilon^{a_j c/(2\cdot 3^j)}$ is 1, we find that $|Q|_2 = O_{C}(1)$.  Thus, when making the replacement, the sum of the coefficients of the $A_\ell B_\ell$ terms and the scaled error term are all $O_{c,C,d}(\epsilon^{-1/(2C)})$.
\end{proof}

Our theorem follows from applying Lemma \ref{Z2Lem} to the partial decomposition $m=1$, $h(x_1) = |p|_2 x_1$, $q_1(X) = p(X)/|p|_2$ and $a_1 = 0$ of complexity 1 and weight $[6 (N+1)/ c ]\omega^d$.
\end{proof}

\section{Basic Facts about Diffuse Decompositions}\label{DDFactsSec}

The primary use of a diffuse decomposition will be that the existence of a diffuse decomposition will allow us to approximate the corresponding threshold function by a smooth function.  In particular, we show:

\begin{prop}\label{ApproximationProp}
Let $(h,q_1,\ldots,q_m)$ by an $(\epsilon,N)$-diffuse decomposition of a degree-$d$ polynomial $p$ for $1/2>\epsilon>0$.  There exists a function $f:\R^m\rightarrow[-1,1]$ so that:
\begin{enumerate}
\item $f(q_1(x),q_2(x),\ldots,q_m(x)) \geq \sgn(p(x))$ pointwise.
\item $\E[f(q_1(X),q_2(X),\ldots,q_m(X))] - \E[\sgn(p(X))] = O_{m,d}(\epsilon N \log(\epsilon^{-1})^{dm/2+1})$.
\item For any $k\geq 0$, $|f^{(k)}|_\infty = O_{m,k}(\epsilon^{-k})$, where $|f^{(k)}|_\infty$ denotes the largest $k^{\tth}$ order mixed partial derivative of $f$ at any point.
\end{enumerate}
\end{prop}

In order to prove this and for some other applications, we will also need the following statement about the distribution of values of $(q_i(X))$ in a diffuse decomposition:
\begin{lem}\label{DDDerLem}
Let $(h,q_1,\ldots,q_m)$ be an $(\epsilon,N)$-diffuse decomposition of a degree-$d$ polynomial for some $1/2>\epsilon>0$.  Letting $x=(q_1(X),q_2(X),\ldots,q_m(X))$ for $X$ a random Gaussian we have that with probability $1-O_{m,d}(N\epsilon \log(\epsilon^{-1})^{dm/2+1})$ that
\begin{equation}\label{SAChainEq}
|h(x)| \geq \epsilon |D_{i_1} h(x)|_2 \geq \epsilon^2 |D_{i_2} D_{i_2} h(x)|_2 \geq \ldots \geq \epsilon^d |D_{i_1}\cdots D_{i_d} h(x)|_2.
\end{equation}
\end{lem}
\begin{proof}
First we note that for some $B=O_m(\log(\epsilon^{-1})^{d/2})$ that by Corollary \ref{ConcentrationCor} that $|q_i(X)|\leq B$ for all $i$ with probability at least $1-\epsilon$.  Hence it suffices to bound the probability that Equation (\ref{SAChainEq}) fails while $|q_i(X)|\leq B$ for all $i$.  We let $R\subset \R^m$ be the region for which this fails.  We bound the probability that $x\in R$ by covering $R$ by axis aligned boxes of side length $2\epsilon$ and using the fact that $(q_1,\ldots,q_m)$ is a diffuse set.  In particular, consider the union of all axis aligned boxes of side length $2\epsilon$ whose endpoints are integer multiples of $2\epsilon$ and which contain some point of $R$.  Call the union of all such boxes $R'$.  Note that since $R'$ is a disjoint union of boxes so that for each such box the probability that $x$ lies in this box is at most $N$ times its volume, we have that $\pr(x\in R) \leq \pr(x \in R') \leq N \textrm{Vol}(R')$.  Let $R''$ be the set of points $y\in R^m$ so that $y$ is within $2\sqrt{m}\epsilon$ of some point in $R$.  Note that $R''\supset R'$.  Thus, it suffices to prove that
$$
\textrm{Vol}(R'') = O_{m,d}(\epsilon \log(\epsilon^{-1})^{dm/2+1}).
$$

Let $Y$ be an $m$-dimensional Gaussian.  Note that $R''$ is contained in a ball of radius $O_m(B)$.  Hence since the probability density function of $BY$ is at least $\Omega_m(B^{-m}dV)$ on this region, we have that $\textrm{Vol}(R'') = O_m(B^m \pr(BY\in R''))$.  Define the polynomial $H(x)=h(Bx)$.  It now suffices to show that with probability at most $O_{d,m}(\epsilon \log(\epsilon^{-1}))$ that $Y$ is within $O_m(\epsilon)$ of a point, $x$ for which
$$
|H(x)| \geq \epsilon |D_{i_1} H(x)|_2 \geq \epsilon^2 |D_{i_2} D_{i_2} H(x)|_2 \geq \ldots \geq \epsilon^d |D_{i_1}\cdots D_{i_d} H(x)|_2
$$
fails to hold.

Note that by Proposition \ref{StrongAnticoncentrationProp} for $k=1$ we have that for any $1/2>\delta>0$ that with probability $1-O_{d,m}(\delta\log(\delta^{-1}))$,
$$
|H(Y)| \geq \delta |D_{i_1} H(Y)|_2 \geq \delta^2 |D_{i_2} D_{i_2} H(Y)|_2 \geq \ldots \geq \delta^d |D_{i_1}\cdots D_{i_d} H(Y)|_2.
$$
If the above holds and $x=Y+z$ for $|z|_2 = O_m(\epsilon)$ we have by Taylor's Theorem that
\begin{align*}
D_{i_1}\cdots D_{i_k} H(x) & = D_{i_1}\cdots D_{i_k} H(Y) + \sum_{t=1}^{d-k} \frac{(D_{i_1}\cdots D_{i_{k+1}}H(Y))z_{i_{k+1}}\cdots z_{i_{k+t}}}{t!}
\end{align*}
Hence we have that
\begin{align*}
&\left|D_{i_1}\cdots D_{i_k} H(x) - D_{i_1}\cdots D_{i_k} H(Y)\right|_2 \\& \ \ \ \ \ \ \ \ \ \  \leq \sum_{t=1}^{d-k} \left|\frac{(D_{i_1}\cdots D_{i_{k+1}}H(Y))z_{i_{k+1}}\cdots z_{i_{k+t}}}{t!}\right|_2\\
&  \ \ \ \ \ \ \ \ \ \  \leq \sum_{t=1}^{d-k} \frac{|(D_{i_1}\cdots D_{i_{k+1}}H(Y))|_2|z|_2^t}{t!}\\
&  \ \ \ \ \ \ \ \ \ \ \leq \sum_{t=1}^{d-k} \frac{|D_{i_1}\cdots D_{i_k} H(Y)|_2 (|z|_2/\delta)^t}{t!}\\
&  \ \ \ \ \ \ \ \ \ \ \leq |D_{i_1}\cdots D_{i_k} H(Y)|_2 (\exp(O_m(|z|_2/\delta))-1).
\end{align*}

Thus, if $\delta = 4\sqrt{m}\epsilon$ and the above holds (which it does with probability $1-O_{d,m}(\epsilon\log(\epsilon^{-1}))$), then for any point $x$ within $2\sqrt{m}\epsilon$ of $Y$ we have that $$\left|D_{i_1}\cdots D_{i_k} H(x) - D_{i_1}\cdots D_{i_k} H(Y)\right|_2\leq |D_{i_1}\cdots D_{i_k} H(Y)|_2(e^{1/2}-1),$$ and thus, Equation (\ref{SAChainEq}) holds.  Thus, $\pr(BY\in R'') \leq O_{d,m}(\epsilon\log(\epsilon^{-1}))$, so $\textrm{Vol}(R'') = O_{d,m}(\epsilon \log(\epsilon^{-1})^{dm/2+1})$, completing our proof.
\end{proof}

\begin{cor}\label{Z3Cor}
Let $(h,q_1,\ldots,q_m)$ be an $(\epsilon,N)$-diffuse decomposition of a degree-$d$ polynomial for $1/2>\epsilon>0$.  Letting $x=(q_1(X),q_2(X),\ldots,q_m(X))$ for $X$ a random Gaussian, the probability that $x$ is within $\epsilon$ of a point $y$ for which $h(y)=0$ is $O_{d,m}(N\epsilon\log(\epsilon^{-1})^{dm/2+1})$.
\end{cor}
\begin{proof}
Note that by the analysis given above, if Equation (\ref{SAChainEq}) holds for $2\epsilon$ then for any $y$ with $|x-y|\leq \epsilon$
$$
|h(x)-h(y)| \leq |h(x)|(e^{1/2}-1) < |h(x)|,
$$
and thus, $h(y)\neq 0$.  Since an $(\epsilon,N)$-diffuse decomposition is also an $(2\epsilon,2^mN)$-diffuse decomposition, this happens with probability at least $1-O_{d,m}(N \epsilon \log(\epsilon^{-1})^{dm/2+1})$ by Lemma \ref{DDDerLem}.
\end{proof}

\begin{proof}[Proof of Proposition \ref{ApproximationProp}]
We construct $f$ in a straightforward manner.  Let $\rho:\R^m\rightarrow\R$ be any smooth, non-negative-valued, function supported on the ball of radius 1 so that
$$
\int_{\R^m} \rho(x)dx =1.
$$
Let $\rho_\epsilon(x) = \epsilon^{-m} \rho(\epsilon^{-1} x)$.  We note that
$$
\int_{\R^m} \rho_\epsilon(x)dx =1.
$$

Let $g:\R^m\rightarrow\R$ be the function
$$
g(x) = \begin{cases} 1 &\textrm{ if there exists a }y\in \R^n \textrm{ so that } |x-y|<\epsilon \textrm{ and }h(y) \geq 0 \\ -1 & \textrm{otherwise} \end{cases}
$$
We let $f$ be the convolution $g * \rho_\epsilon$.

$f$ takes values in $[-1,1]$ because
$$
f(x) = \int_{\R^m} \rho_{\epsilon}(y) g(x-y) dy \leq \int_{\R^m} \rho_{\epsilon}(y)dy = 1
$$
and similarly $f(x) \geq -1$.

$f(q_1,\ldots,q_m)$ is a point-wise upper bound for $\sgn\circ p = \sgn(h(q_1,\ldots,q_m))$ since if $h(x)\geq 0$ then
$$
f(x) = \int_{\R^m} \rho_{\epsilon}(y) g(x-y) dy \int_{B(\epsilon)} \rho_{\epsilon}(y) g(x-y) dy = \int_{B(\epsilon)} \rho_{\epsilon}(y)dy = 1.
$$

Bounds on the derivatives of $f$ come from the fact that
$$
|f^{(k)}|_\infty = |g * \rho_\epsilon^{(k)}|_\infty \leq |g|_\infty |\rho_\epsilon^{(k)}|_1 = O_{m,k}(\epsilon^{-k}).
$$

The second property follows from the fact that $f(x)=\sgn(h(x))$ unless $x$ is within $2\epsilon$ of a point $y$ for which $h(y)=0$.  Noting that an $(\epsilon,N)$-diffuse decomposition of size $m$, is also a $(2\epsilon,2^mN)$-diffuse decomposition, this happens with probability $O_{d,m}(N\epsilon \log(\epsilon^{-1})^{dm/2+1})$.  Since $|f(x)-\sgn(h(x))|$ is never more than 2 this provides the necessary bound.
\end{proof}

Another lemma that will be useful to us is the following:
\begin{lem}\label{hSmallLem}
Let $(h,q_1,\ldots,q_m)$ be an $(\epsilon,N)$-diffuse decomposition of a degree-$d$ polynomial $p$ for $1/2>\epsilon>0$ and $N\epsilon \log(\epsilon^{-1})$ less than a sufficiently small function of $m$ and $d$.  Then $|h|_2 \leq O_{m,d}( N^d  |p|_2)$.
\end{lem}
\begin{proof}
Consider the probability that $|p(X)| > 2|p|_2$.  On the one hand, it is at most $1/4$ by the Markov bound.  We will show that if $|h|_2$ is more than a sufficiently large constant times $N^d|p|_2$, then the probability must be more than this.

We note by Corollary \ref{ConcentrationCor} that with probability at least $7/8$ that each $q_i(X)$ is $O_d(\log(m)^{d/2})$.  We consider the probability that each $q_i(X)$ is at most this size and that $|p(X)|\leq 2|p|_2$.  We bound this probability above by coving the set of $x\in \R^m$ with each $|x_i|=O_d(\log(m)^{d/2})$ so that $|h(x)|\leq 2|p|_2$ with boxes of side length $\epsilon$.  The probability is at most $N$ times the volume of the union of these boxes.  Furthermore, the union of these boxes is contained in the set of $x\in \R^m$ with $|x_i|\leq O_d(\log(m)^{d/2})$ for each $i$ and so that $x$ is within $\epsilon m$ of some point $y$ with $|h(y)| \leq 2|p|_2.$  Call this region $R$.  We note that since $R$ is contained in a ball of radius $O_m(1)$ that the volume of $R$ is bounded by some constant times the probability that a random Gaussian $X$ lies in $R$.

By Proposition \ref{StrongAnticoncentrationProp}, we have that with probability $1-O_{d,m}(\epsilon \log(\epsilon^{-1}))$ over Gaussian $X$ that
$$
|h(X)| > 4\epsilon m |D_{i_1} h(X)|_2 > \ldots > (4\epsilon m)^d |D_{i_1}\cdots D_{i_d} h(X)|_2.
$$
This would imply that for any $y$ within $m\epsilon$ of $X$ that $|h(X)-h(y)|\leq |h(X)|/2$ by means of the Taylor series for $h(y)$.  On the other hand $|h(X)|\geq 4|p|_2$ with probability at least $1-O_{d}((|p|_2/|h|_2)^{1/d})$ by Lemma \ref{anticoncentrationLem}.  Thus, the probability that $|h(X)|\leq 2|p|_2$ is at most
$$
O_{d,m}(\epsilon \log(1+\epsilon^{-1}) + (|p|_2/|h|_2)^{1/d})+1/8.
$$
Hence we have that
$$
1/4 \leq \pr(|p(X)| > 2|p|_2) \leq O_{d,m}(N \epsilon \log(1+\epsilon^{-1}) + N (|p|_2/|h|_2)^{1/d})+1/8.
$$
Thus, if $N\epsilon \log(\epsilon^{-1})$ is less than some sufficiently small function of $d,m$, then $|h|_2 = O_{d,m}(N^d)$.

\end{proof}

Fundamentally, having a diffuse decomposition is useful because it allows us to improve our application of the replacement method.  The following proposition presents this technique in fair generality.

\begin{prop}\label{replacementProp}
Let $p_0:\R^n\rightarrow\R$ be a degree-$d$ polynomial with an $(\epsilon,N)$-diffuse decomposition $(h,q_1,\ldots,q_m)$ for some $1/2>\epsilon>0.$  Let $n_i$ be positive integers so that $n=\sum_{i=1}^\ell n_i$.  We can then consider $p_0$ and each of the $q_i$ as functions on $\R^{n_1}\times\cdots\times\R^{n_\ell}$.

Let $X^1,\ldots,X^\ell$ and $Y^1,\ldots,Y^\ell$ be independent random variables, where $X^j$ and $Y^j$ take values in $\R^{n_j}$ and $Y^j$ is a random Gaussian.  Furthermore, assume that for some integer $k>1$ that for any polynomial $g$ in $m$ variables of degree less than $k$, any $1\leq j\leq \ell$ and any $z^i$ that
$$
\E\left[g(q_i(z^1,\ldots,z^{j-1},X^j,z^{j+1},\ldots,z^\ell)) \right] = \E\left[g(q_i(z^1,\ldots,z^{j-1},Y^j,z^{j+1},\ldots,z^\ell)) \right].
$$

For each $1\leq i \leq m$ and each $1\leq j\leq \ell$ define
$$
Q_{i,j}(x^1,\ldots,x^{j-1},x^{j+1},\ldots,x^\ell) := \E_{Y^j} [q_i(x^1,\ldots,x^{j-1},Y^j,x^{j+1},\ldots,x^\ell)].
$$
Define $T_{i,j}$ to be
\begin{align*}
& \E \left[\left|q_i(Y^1,\ldots,Y^j,X^{j+1},\ldots,X^{\ell}) - Q_{i,j}(Y^1,\ldots,Y^{j-1},X^{j+1},\ldots,X^{\ell}) \right|^k\right] \\ +&\E\left[  \left|q_i(Y^1,\ldots,Y^{j-1},X^{j},\ldots,X^{\ell}) - Q_{i,j}(Y^1,\ldots,Y^{j-1},X^{j+1},\ldots,X^{\ell}) \right|^k\right].
\end{align*}
And let
$$
T:=\sum_{i=1}^m \sum_{j=1}^\ell T_{i,j}.
$$
Then we have that
\begin{align*}
&\left|\pr\left(p_0(X^1,\ldots,X^\ell)\leq 0\right)  -\pr\left(p_0(Y^1,\ldots,Y^\ell)\leq 0\right)\right|  \leq O_{d,m,k}\left(\epsilon^{-k}T +\epsilon N \log(\epsilon^{-1})^{dm/2+1} \right).
\end{align*}

If furthermore, $p$ is a degree-$d$ polynomial so that for some parameters $\delta,\eta>0$
$$
\pr\left(|p(X)-p_0(X)| < \delta |p|_2\right) , \pr\left(|p(Y)-p_0(Y)| < \delta |p|_2\right) \  \geq \ 1-\eta
$$
then
\begin{align*}
&\left|\pr\left(p(X^1,\ldots,X^\ell)\leq 0\right)  -\pr\left(p(Y^1,\ldots,Y^\ell)\leq 0\right)\right|  \leq O_{d,m,k}\left(\epsilon^{-k}T +\epsilon N \log(\epsilon^{-1})^{dm/2+1} + \delta^{1/d} + \eta \right).
\end{align*}
\end{prop}

When considering Proposition \ref{replacementProp}, it might be useful to keep the intended applications in mind.  In Section \ref{GaussianPRGSec}, we will consider the case where the $X^j$ are chosen from $dk$-independent families of Gaussians. In Section \ref{InvarianceSec}, we will consider the case where the $X^j$ are Bernoulli random variables. Finally, in Section \ref{BernoulliPRGSec}, we will consider the case where the $X^j$ are chosen from $4d$-independent families of random Bernoullis.

\begin{proof}
We begin by proving the first of the two bounds, and will then use it to prove the second.  By rescaling $p_0$, we may assume that $|p_0|_2=1$.  Let $X=(X^1,\ldots,X^\ell)$ and $Y=(Y^1,\ldots,Y^\ell)$.  Let $q$ denote the vector valued polynomial $(q_1,\ldots,q_m)$.
We will show that
$$
\pr\left(p_0(X)\leq 0\right) \leq \pr\left(p_0(Y)\leq 0\right) +O_{d,m,k}\left(\epsilon^{-k}T +\epsilon N \log(\epsilon^{-1})^{dm/2+1} \right).
$$
The other inequality will follow analogously.

By Proposition \ref{ApproximationProp} there exists a function $f:\R^m\rightarrow [0,1]$ so that
\begin{enumerate}
\item $f(x)=1$ for all $x$ where $h(x)\leq 0$.
\item $\E[f(q(Y))] = \pr(p_0(Y)\leq 0) +O_{d,m}(\epsilon N \log(\epsilon^{-1})^{dm/2+1})$.
\item $|f^{(k)}|_\infty = O_{m,k}(\epsilon^{-k})$.
\end{enumerate}

We note that
$$
\pr(p_0(X)\leq 0) \leq \E[f(q_i(X))]
$$
and that
$$
\E[f(q(Y))] \leq \pr(p_0(Y)\leq 0) +O_{d,m}(\epsilon N \log(\epsilon^{-1})^{dm/2+1}).
$$
Hence it suffices to prove that
\begin{equation}\label{totalReplacementEqn}
\left| \E[f(q(X))] - \E[f(q(Y))]\right| = O_{d,m,k}(\epsilon^{-k}T).
\end{equation}

For $0\leq j \leq \ell$, let
$$
Z^{(j)} := (Y^1,\ldots,Y^j,X^{j+1},\ldots,X^\ell).
$$
In particular, $Z^{(0)}=X$, $Z^{(\ell)}=Y$, and $Z^{(j)}$ is obtained from $Z^{(j-1)}$ by changing the $j^{\tth}$ block of coordinates from $X^j$ to $Y^j$.  We will attempt to bound the left hand side of Equation (\ref{totalReplacementEqn}) by bounding
\begin{equation}\label{replacementEqn}
\left| \E[f(q(Z^{(j)}))] - \E[f(q(Z^{(j-1)}))]\right|
\end{equation}
for each $j$.

Consider the expression in Equation (\ref{replacementEqn}) for fixed values of $Y^1,\ldots,Y^{j-1},X^{j+1},\ldots,X^\ell$.  We approximate $f(q(Z^{(j)}))$ and $f(q(Z^{(j-1)}))$ by Taylor expanding $f$ about $(\overline{q_1},\ldots,\overline{q_m})$ where
\begin{align*}
\overline{q_i}(Y^1,\ldots,Y^{j-1},Z,X^{j+1},\ldots,X^\ell) & := Q_{i,j}(Y^1,\ldots,Y^{j-1},X^{j+1},\ldots,X^\ell)\\
& = \E[q_i(Z^{(j)})]\\
& = \E[q_i(Z^{(j-1)})].
\end{align*}
Thus, for some polynomial $g$ of degree $k-1$ we have that:
\begin{align*}
f(q(Z)) & = g(q(Z)) + O\left(\sum_{i_1,\ldots,i_k} \frac{\partial^k f}{\partial q_{i_1}\cdots \partial q_{i_k}}\prod_{a=1}^k (q_{i_a}(Z) - \overline{q_{i_a}}(Z)) \right)\\
& = g(q(Z)) +O_{d,m,k}\left(\sum_{i_1,\ldots,i_k} \epsilon^{-k}\sum_{a=1}^k |q_{i_a}(Z) - \overline{q_{i_a}}(Z)|^k \right)\\
& = g(q(Z)) +O_{d,m,k}\left(\epsilon^{-k}\sum_{i=1}^m |q_{i}(Z) - \overline{q_{i}}(Z)|^k \right).
\end{align*}
By assumption
$$
\E[g(q(Z^{(j)}))] = \E[g(q(Z^{(j-1)}))].
$$
Thus, the expression in Equation (\ref{replacementEqn}) is at most
\begin{align*}
\epsilon^{-k}O_{d,m,k}&\left(\sum_{i=1}^m \E[|q_{i}(Z^{(j)}) - \overline{q_{i}}(Z^{(j)})|^k]+\E[|q_{i}(Z^{(j-1)}) - \overline{q_{i}}(Z^{(j-1)})|^k] \right)\\ & = O_{d,m,k}\left(\epsilon^{-k} \sum_{i=1}^m T_{i,j}\right).
\end{align*}
Summing over $j$ yields Equation (\ref{totalReplacementEqn}), proving the first part of this proposition.

Changing our normalization so that $|p|_2=1$, we have that
$$
\pr(p(X) \leq 0) \leq \pr(p_0(X) - \delta \leq 0) + O(\eta)
$$
Notice that $p_0-\delta$ has the diffuse decomposition $(h-\delta,q_1,\ldots,q_m)$.  Therefore, applying our previous result to this decomposition of $p_0-\delta$, we have that
$$
\pr(p_0(X) - \delta \leq 0) \leq \pr(p_0(Y) - \delta \leq 0) + O_{d,m,k}\left(\epsilon^{-k}T +\epsilon N \log(\epsilon^{-1})^{dm/2+1} \right).
$$
On the other hand we have that
$$
\pr(p_0(Y)-\delta \leq 0) \leq \pr(p(Y)-2\delta \leq 0) + O(\eta).
$$
Finally, by Lemma \ref{anticoncentrationLem} we have that
$$
\pr(p(Y)-2\delta \leq 0) \leq \pr(p(Y)\leq 0) + O(d\delta^{1/d}).
$$
Combining the above inequalities we find that
$$
\pr(p(X)\leq 0) \leq \pr(p(Y)\leq 0) + O_{d,m,k}\left(\epsilon^{-k}T +\epsilon N \log(\epsilon^{-1})^{dm/2+1} + \delta^{1/d} + \eta \right).
$$
The other direction of the inequality follows analogously, and this completes our proof.
\end{proof}

\section{Application to PRGs for PTFs with Gaussian Inputs}\label{GaussianPRGSec}

In \cite{GPRG}, the author introduced a new pseudorandom generator for polynomial threshold functions of Gaussian inputs.  In particular, for appropriately chosen parameters $N$ and $k$ he lets
$$
X = \frac{1}{\sqrt{N}}\sum_{i=1}^N X^i
$$
where the $X^i$ are independently chosen from $k$-independent families of Gaussians.  He shows that for some $k=O(d/c)$ and $N=2^{O_c(d)}\epsilon^{-4-c}$ that for any such $X$, if $Y$ is a random Gaussian and $f$ any degree-$d$ polynomial threshold function then
\begin{equation}\label{GPRGEquation}
|\E[f(X)] - \E[f(Y)]| < \epsilon.
\end{equation}

The proof of this is by the replacement method.  In particular, $f$ is replaced by a smooth approximation $g$, and bounds are proved on the change in the expectation of $g(X)$ as the $X^i$ are replaced by random Gaussians one at a time.  The power of this method is highly dependent on ones ability to find a $g$ that is close to $f$ with high probability and yet has relatively small higher derivatives.  If $f(x)=\sgn(p(x))$, a naive attempt to use the replacement method would use $g=\rho(p(x))$ for $\rho$ a smooth approximation to the sign function.  Unfortunately, this approach will have difficulty proving Equation (\ref{GPRGEquation}) unless $N>\epsilon^{-2d}$.  In \cite{GPRG}, the author uses a version of Proposition \ref{StrongAnticoncentrationProp} and constructs a $g$ which approximates $f$ as long as an appropriate analogue of
$$
|g(x)| \geq \epsilon |D_{i_1} g(x)|_2 \geq \epsilon^2 |D_{i_1} D_{i_2} g(x)|_2 \geq \ldots
$$
holds.  The analysis of this is somewhat complicated, involving the development of the theory of the so-called ``noisy derivative''.  Furthermore, for technical reasons this method has difficulty dealing with $N$ smaller than $\epsilon^{-4}$.  As a first application of our theory of diffuse decompositions we provide a relatively simple analysis of this generator that works with $N$ as small as $\epsilon^{-2-c}$.  In particular we show:

\begin{thm}\label{GPRGThm}
Given, an integer $d>0$ and real numbers $c,\epsilon>0$, there exist integers $k=O(d/c)$ and $N=O_{c,d}(\epsilon^{-2-c})$ so that for any random variable
$$
X = \frac{1}{\sqrt{N}}\sum_{i=1}^N X^i
$$
where the $X^i$ are chosen independently from $k$-independent distributions of Gaussians, and for any degree-$d$ polynomial threshold function $f$,
$$
|\E[f(X)] - \E_{Y\sim \mathcal{N}}[f(Y)]| < \epsilon.
$$
\end{thm}
\begin{proof}
We begin by making a few reductions to produce a more amenable case.  We assume throughout that $\epsilon$ is sufficiently small.  Note that it is sufficient to prove that for $N=\epsilon^{-2-c}$ that the error is $O_{c,d}(\epsilon^{1-2c})$, since making appropriate changes to $c$ and $\epsilon$ will yield the necessary result.  Secondly, we may let $f(x) = \sgn(p(x))$ for $p$ a degree-$d$ polynomial with $|p|_2=1$.

By Theorem \ref{DDTheorem}, there exists a degree-$d$ polynomial $p_0$ with $|p-p_0|_2=O_{c,d}(\epsilon^{d+1})$, and so that $p_0$ has an $(\epsilon,\epsilon^{-c/2})$-diffuse decomposition $(h,q_1,\ldots,q_m)$.  It should be noted that by $2$-independence,
$$
\E[|p(X)-p_0(X)|^2] = \E[|p(Y)-p_0(Y)|^2] = O_{c,d}(\epsilon^{2d+2}).
$$
Therefore, by the Markov bound we have with probability at least $1-\epsilon^2$ that
$$
|p(X)-p_0(X)|, |p(Y)-p_0(Y)| < \epsilon^d.
$$

We note that we may write $Y=\frac{1}{\sqrt{N}}\sum_{j=1}^N Y^j$, where the $Y^j$ are independent Gaussians.  We define the polynomial
$$
p'(Y^1,\ldots,Y^N) := p\left(\frac{1}{\sqrt{N}}\sum_{j=1}^N Y^j\right).
$$
We note that
$$
p(X) = p'(X^1,\ldots,X^N)
$$
and
$$
p(Y) = p'(Y^1,\ldots,Y^N).
$$

It is clear that if we define $p_0'$ and $q_i'$ analogously, that $p_0'$ has an $(\epsilon,\epsilon^{c/2})$-diffuse decomposition $(h,q_1',\ldots,q_m')$, and that with probability at least $1-\epsilon^2$ that
$$
|p'(X)-p_0'(X)|, |p'(Y)-p_0'(Y)| < \epsilon^d.
$$
We may thus apply Proposition \ref{replacementProp} to $p',p_0'$ with $\eta=\epsilon^2$ and $\delta=\epsilon^d$.

Let $K$ be an even integer less than $k/d$ and more than $6/c$.  By $k$-independence of the $X^j$, any polynomial $g$ of degree less than $K$ in the $q_i'$ will have the same expectation evaluated at $X^1,\ldots,X^N$ as at $Y^1,\ldots,Y^N$.  Hence by Proposition \ref{replacementProp},
\begin{align}\label{Z9Eqn}
\left|\E[f(X)] - \E[f(Y)] \right| & = 2\left|\pr(p(X^1,\ldots,X^N)\leq 0) - \pr(p(Y^1,\ldots,Y^N)\leq 0) \right| \nonumber \\
& = O_{d,m,c}\left(\epsilon^{1-c}\log(\epsilon^{-1})^{dm/2+1} + \epsilon^{-K}T + \epsilon \right).
\end{align}
Where by the $K$-independence of $X$, the $T$ above is
$$
2\sum_{i=1}^m\sum_{j=1}^N \E\left[\left(q_i(Y) -\E_{Y^j}[q_i'(Y^1,\ldots,Y^N)] \right)^K \right].
$$
By Lemma \ref{hypercontractiveLem}, this is
$$
O_{c,d}\left( \sum_{i=1}^m\sum_{j=1}^N \E\left[\left(q_i(Y) -\E_{Y^j}[q_i'(Y^1,\ldots,Y^N)] \right)^2 \right]^{K/2}\right).
$$
Letting $Z=\frac{1}{\sqrt{N-1}}\sum_{i\neq j} Y^i$ (which is a random Gaussian), the expectations in question are
$$
\E_Z\left[\var_Y\left(q_i\left(\sqrt{\frac{N-1}{N}}Z + \frac{1}{\sqrt{N}} Y\right)\right) \right].
$$
This in turn is at most
$$
\E\left[\left(q_i\left(\sqrt{\frac{N-1}{N}}Z + \frac{1}{\sqrt{N}} Y\right) - q_i(Z)\right)^2 \right].
$$
We bound this with the following lemma, which follows immediately from Claim 4.1 of \cite{sense3}:

\begin{lem}
For $q$ any degree-$d$ polynomial we have that
$$
\E\left[\left|q(Z) - q\left(\sqrt{\frac{N-1}{N}}Z + \frac{1}{\sqrt{N}}Y\right)\right|^2 \right] = O(d^2|q|_2^2 / N).
$$
\end{lem}

Thus, $T$ is at most
\begin{align*}
O_{c,d}\left(\sum_{i=1}^m \sum_{j=1}^N N^{-K/2} \right) & = O_{c,d,m}(N^{-K/2+1})\\ & = O_{c,d,m}(\epsilon^{K-2+Kc/2-c})\\ & = O_{c,d,m}(\epsilon^{K+1-c}).
\end{align*}
Thus, by Equation (\ref{Z9Eqn}),
$$
|\E[f(X)]-\E[f(Y)]|\leq O_{c,d,m}(\epsilon^{1-2c}),
$$
as desired.
\end{proof}

\section{The Diffuse Invariance Principle and Regularity Lemma}\label{InvarianceSec}

While the case of Gaussian inputs is very convenient for proving theorems such as the Decomposition Theorem, many interesting questions involve evaluation of polynomials on random variables from other distributions.  Perhaps the most studied of these is the Bernoulli, or hypercube distribution.

\begin{defn}
The $n$-dimensional Bernoulli distribution is the probability distribution on $\R^n$ where each coordinate is randomly and uniformly chosen from the set $\{-1,1\}$.  Equivalently, it is the uniform distribution on the set $\{-1,1\}^n.$
\end{defn}

As we have been using $X,Y,Z,\etc.$ to represent Gaussian random variables, we will attempt to use $A,B,\etc.$ for Bernoulli random variables.

A powerful tool for dealing with Bernoulli variables is the use of \emph{invariance principles}.  These are theorems which state that if $p$ is a sufficiently regular polynomial (for some definition of regularity) that the distributions of $p(X)$ and $p(B)$ are similar to each other (generally that they are close in cdf distance).  This allows one to make use of results in the Gaussian setting and apply them to the Bernoulli setting (at least for sufficiently regular polynomials).  Since not all polynomials will be regular, in order to make use of this idea in a more general context, one also needs a \emph{regularity lemma}.  These are structural results that allow us to write arbitrary polynomials of Bernoulli random variables in terms of regular ones.

In this section, we will discuss some of the existing invariance principles and regularity lemmas, and make use of the theory of diffuse decompositions to provide some new ones that will deal better with high degree polynomials.  In Section \ref{BernFactsSec}, we discuss some background information about polynomials of Bernoulli random variables and give a brief overview of existing invariance principles and regularity lemmas.  In Section \ref{InvarianceSubSec}, we state and prove the Diffuse Invariance Principle, and in Section \ref{RegSec} prove the corresponding regularity lemma.

\subsection{Basic Facts about Bernoulli Random Variables}\label{BernFactsSec}

\subsubsection{Multilinear Polynomials}

For a Bernoulli random variable $B$, we have that any coordinate, $b_i$, satisfies $b_i^2=1$ with probability 1.  This, of course, does not hold in the Gaussian case.  Thus, if there is going to by any hope of comparing polynomials on Gaussian and Bernoulli inputs, we must restrict ourself to polynomials that have no term that is degree more than 1 in any variable.  In particular, we must restrict ourselves to the case of multilinear polynomials:

\begin{defn}
A polynomial $p:\R^n\rightarrow \R$ is multilinear if its degree with respect to any coordinate variable is at most 1.
\end{defn}

To clarify the relationship between general polynomials and multilinear polynomials we mention the following lemma:
\begin{lem}\label{linearizationLem}
For every polynomial $p:\R^n\rightarrow\R$, there exists a unique multilinear polynomial $q:\R^n\rightarrow\R$ so that $q$ agrees with $p$ on $\{-1,1\}^n$.  Furthermore, $\deg(q)\leq \deg(p)$.
\end{lem}
\begin{proof}
To prove the existence of $q$, it suffices to show that the result holds for every monomial $p=\prod x_i^{\alpha_i}$.  It is clear that this monomial agrees on the hypercube with the multilinear monomial $\prod x_i^{\alpha_i \pmod{2}}$.

Uniqueness will follow from the fact that any non-zero multilinear polynomial on $\R^n$ is non-vanishing on the hypercube.  This follows from the fact that the map from a multilinear polynomial to its vector of values on $\{-1,1\}^n$ is a surjective linear map of vector spaces of dimension $2^n$.
\end{proof}

\begin{defn}
For any polynomial $p(x)$, let $L(p(x))$ be the corresponding multilinear polynomial as described by Lemma \ref{linearizationLem}.
\end{defn}

\subsubsection{$L^p$ Norms and Hypercontractivity}

As the $L^p$ norms for polynomials of Gaussians have been useful to us, the corresponding norms for the Bernoulli distribution will also be useful.

\begin{defn}
Let $p:\R^n\rightarrow\R$ we for $t\geq 1$, we define $|p|_{B,t}$ as
$$
|p|_{B,t} = \left(\E_B[|p(B)|^t]\right)^{1/t}.
$$
Where above $B$ is an $n$-dimensional Bernoulli random variable.
\end{defn}

We also have the analogue of Lemma \ref{hypercontractiveLem}.  In particular, we have that:
\begin{lem}[Bonami \cite{BHyp}]\label{BhypercontractiveLem}
For $p:\R^n\rightarrow\R$ a degree-$d$ polynomial, and $t\geq 2$ we have that
$$
|p|_{B,t} \leq \sqrt{t-1}^d |p|_{B,2}.
$$
\end{lem}

Yielding the Corollary
\begin{cor}\label{BConcCor}
For $p:\R^n\rightarrow \R$ a degree-$d$ polynomial $N>0$, then
$$
\pr_B(|p(B)| > N|p|_{B,2}) = O\left(2^{-(N/2)^{2/d}} \right).
$$
\end{cor}
The proof is analogous to that of Corollary \ref{ConcentrationCor}.

We will also need a result combining Lemmas \ref{hypercontractiveLem} and \ref{BhypercontractiveLem}
\begin{lem}\label{mixedHypercontractiveLem}
Let $p$ be a degree-$d$ polynomial, $B$ a Bernoulli random variable, $G$ a Gaussian random variable, and $t \geq 2$ a real number.  Then
$$
\E[|p(G,B)|^t] \leq (t-1)^{td/2} \E[p(G,B)^2]^{t/2}
$$
\end{lem}
\begin{proof}
For integers $N$, let $G^N$ be a random variable defined by $G^N = \frac{1}{\sqrt{N}} \sum_{j=1}^N A^j$ where the $A^j$ are independent Bernoulli random variables.  Clearly the coordinates of $G^N$ are independent and by the central limit theorem, as $N\rightarrow \infty$, their distributions converge to Gaussians in cdf distance.  This implies that for and $\epsilon>0$ and for sufficiently large $N$ that we can have correlated copies of the random variables $G$ and $G^N$ so that $|G-G^N|<\epsilon$ with probability $1-\epsilon$.  Furthermore, with probability $1-\epsilon$, $|G|=O_n (\log(\epsilon^{-1}))$ (here $n$ in the number of coordinates of $G$).  Therefore, for sufficiently large $N$ we have that with probability $1-O(\epsilon)$ that $|p(G,B)-p(G^N,B)| = O_p(\epsilon \log(\epsilon^{-1})^d)$ (this follows from considering every possible value of $B$ separately).  Furthermore, note that $\E[|p(G,B)|^{2t}]$ and $\E[|p(G^N,B)|^{2t}]$ are both finite.  Thus, for any indicator random variable, $I$ we have that
$$
\E[|p(G,B)|^t I] \leq \E[|p(G,B)|^{2t}]^{1/2} \E[I]^{1/2} = O_{p,t}(\E[I]^{1/2}).
$$
Similarly,
$$
\E[|p(G^N,B)|^t I] = O_p(\E[I]^{1/2}).
$$
Therefore, if $|p(G,B)-p(G^N,B)| = O_p(\epsilon \log(\epsilon^{-1})^d)$ with probability $1-O(\epsilon)$, let $I$ be the indicator random variable for the event that this fails.  Then
\begin{align*}
\E & [|p(G,B)|^t] -\E[|p(G^N,B)|^t]\\ & \leq O_t(1) \E[|p(G,B) - p(G^N,B)| |p(G,B) + p(G^N,B)|^{t-1} ]\\
& \leq O_t(1) \E[|p(G,B) - p(G^N,B)|]^{1/2} \E[|p(G,B)|^{2t-1/2} + |p(G^N,B)|^{2t-1/2} ]^{1/2}\\
& = O_{p,t}(1)\left( \E[|p(G,B) - p(G^N,B)|(1-I)] + \E[|p(G,B) - p(G^N,B)|I]\right)^{1/2} \\
& = O_{p,t}(1)\left( O_p(\epsilon \log(\epsilon^{-1})^d) + O_p(\sqrt{\epsilon}) \right)^{1/2} \\
& = O_{p,t}(\epsilon^{1/4}).
\end{align*}
Therefore
$$
\lim_{N\rightarrow\infty} \E[|p(G^N,B)|^t] = \E[|p(G,B)|^t].
$$

On the other hand, note that $p(G^N,B)$ can be thought of as a polynomial evaluated at a Bernoulli random variable in perhaps a greater number of dimensions.  Hence by Lemma \ref{BhypercontractiveLem},
$$
\E[|p(G^N,B)|^t] \leq (t-1)^{dt/2}\E[p(G^N,B)^2]^{t/2}.
$$
Thus, 
\begin{align*}
\E[|p(G,B)|^t] & = \lim_{N\rightarrow\infty} \E[|p(G^N,B)|^t] \\ &  \leq \lim_{N\rightarrow\infty}(t-1)^{dt/2}\E[p(G^N,B)^2]^{t/2} \\ & = (t-1)^{td/2} \E[p(G,B)^2]^{t/2}.
\end{align*}

\end{proof}

We also note the following relationship between the Gaussian and Bernoulli norms
\begin{lem}\label{L2eqLem}
If $p:\R^n\rightarrow\R$ is a multilinear polynomial then $|p|_2 = |p|_{B,2}$.
\end{lem}
\begin{proof}
This follows immediately after noting that the basis $\prod x_i^{\alpha_i}$ for $\alpha\in\{0,1\}^n$ is an orthonormal basis of the set of multilinear polynomials with respect to both the Bernoulli and Gaussian measures.
\end{proof}

\subsubsection{Influence and Regularity}

The primary obstruction to a multilinear polynomial behaving similarly when evaluated at Bernoulli inputs rather than Gaussian inputs is when some single coordinate has undo effect on the output value of the polynomial.  In such a case, the fact that this coordinate is distributed as a Bernoulli rather than a Gaussian may cause significant change to the resulting distribution.  In order to quantify the extent to which this can happen we define the $i^{\tth}$ influence of a coordinate as follows:

\begin{defn}
For $p:\R^n\rightarrow \R$ a function we define the $i^{\tth}$ influence of $p$ to be
$$
\Inf_i(p) := \left| \frac{\partial p}{\partial x_i} \right|_2^2.
$$
It should be noted that for multilinear polynomials $p$, this is equivalent to the more standard definition
$$
\Inf_i(p) = \E_A [ \var_{a_i}(p(A))].
$$
This is the expectation over uniform independent $\{-1,1\}$ choices for the coordinates other than the $i^{\tth}$ coordinate of the variance of the resulting function over a Bernoulli choice of the $i^{\tth}$ coordinate.  Equivalently, it is
$$
\frac{1}{4} \E\left[\left|p(a_1,\dlots,a_{i-1},-1,a_{i+1},\ldots,a_n) - p(a_1,\dlots,a_{i-1},1,a_{i+1},\ldots,a_n) \right|^2 \right].
$$
\end{defn}

We now prove some basic facts about the influence.
\begin{lem}\label{infCoefLem}
If $p:\R^n\rightarrow\R$ is a polynomial $\Inf_i(p)$ is $\sum_{a} a_i |c_a(p)|^2$, where $c_a(p)$ are the Hermite coefficients of $p$.
\end{lem}
\begin{proof}
Recall that
$$
p(x) = \sum_a c_a(p)h_a(x).
$$
Therefore, we have that
$$
\frac{\partial p}{\partial x_i} = \sum_a \sqrt{a_i} c_a(p) h_{a-e_i}(x).
$$
Thus, 
$$
\left| \frac{\partial p}{\partial x_i}\right|_2^2 = \sum_a a_i |c_a(p)|^2.
$$
\end{proof}

From this we have
\begin{cor}
For $p$ a degree-$d$ polynomial in $n$ variables,
$$
\sum_{i=1}^n \Inf_i(p) = \sum_{k=1}^d k |p^{[k]}|_2^2 = \Theta_d(\var(p(X))).
$$
\end{cor}

We now make the following definition (which agrees with the standard ones up to changing $\tau$ by a factor of $\Theta_d(1)$):
\begin{defn}
Let $p$ be a degree-$d$ multilinear polynomial.  We say that $p$ is $\tau$-regular if for each $i$
$$
\Inf_i(p) \leq \tau \var_A(p).
$$
\end{defn}

In terms of this notion of regularity, the standard invariance principle, proved in \cite{MOO}, can be stated as follows:
\begin{thm}[The Invariance Principle (Mossel, O'Donnell, and Oleszkiewicz)]\label{Z4Thm}
If $p$ is a $\tau$-regular, degree-$d$ multilinear polynomial, $A$ and $X$ are Bernoulli and Gaussian random variables respectively and $t\in \R$, then
$$
\left| \pr(p(X) \leq t) - \pr(p(A)\leq t)\right| = O(d\tau^{1/(8d)}).
$$
\end{thm}

It should be noted that the dependence on $\tau^{1/d}$ in the error of Theorem \ref{Z4Thm} is necessary.  In particular, if $d$ is even and $N$ is a sufficiently large integer consider the polynomial $p:\R^{n+1}\rightarrow\R$ defined by
$$
p(x_0,\ldots,x_{N}) = \tau x_0 + \left(\frac{1}{\sqrt{N}}\sum_{i=1}^{N} x_i \right)^d.
$$
Let $q=L(p)$.  It is not hard to see that by making $N$ sufficiently large, one can make $|p-q|_2$ arbitrarily small, and thus, by Lemma \ref{anticoncentrationLem} and Corollary \ref{ConcentrationCor}, we can make the probability distributions for $p(X)$ and $q(X)$ arbitrarily close.  It is also not hard to see that $q$ is $\Theta_d(\tau^2)$ regular.  This is because $\Inf_0(q)=\tau^2$, $\Inf_i(q)=O_d(N^{-1})$ for $i\neq 0$, and $\var_A(q(A)) = \Theta_d(1)$.  On the other hand, it is clear that for Bernoulli input $A$ we have that
$$
q(A)=p(A) \geq -\tau.
$$
On the other hand considering the distribution of values of $p(X)$ (which as stated can be arbitrarily close to that of $q(X)$), if we let $y=\frac{1}{\sqrt{N}}\sum_{i=2}^{N+1} x_i$, we note that $x_0$ and $y$ are independent Gaussians.  Thus, with probability $\Theta(\tau^{1/d})$ we have that $x_0 < -2$ and $|y| \leq \tau^{1/d}$.  If these occur, then $p(X)<-\tau$.  Thus, for $N$ sufficiently large the difference between the probabilities that $q(A) < -\tau$ and that $q(X) < -\tau$ can be as large as $\Omega(\tau^{1/d})$.

The essential problem in the above example is that although the first coordinate has low influence, there is a reasonable probability that the size of $q(X)$ will be comparable to $\tau$, and in the case when $|q(X)|$ is small, the relative effect of the first coordinate is much larger.  We get around this problem by introducing a new concept of regularity involving the idea of a diffuse decomposition.  The problem above came from the fact that the probability distribution of $q(X)$ was too clustered near 0.  Since the analogue of this cannot happen for a diffuse set of polynomials, we expect to obtain better bounds.

\begin{defn}
For $p$ a degree-$d$ multilinear polynomial, we say that $p$ has a \emph{$(\tau,N,m,\epsilon)$-regular decomposition} if there exists a polynomial $p_0$ of degree-$d$ so that
\begin{itemize}
\item $|p-p_0|^2_{B,2} \leq \epsilon^2 \var(p_0(X))$.
\item $p_0$ has a $(\tau^{1/5},N)$-diffuse decomposition of size $m$, $(h,q_1,\ldots,q_m)$ so that $q_i$ is multilinear for each $i$ and $\Inf_j(q_i)\leq \tau$ for each $i,j$.
\end{itemize}
\end{defn}

\begin{thm}[The Diffuse Invariance Principle]\label{DInvPrinThm}
If $p$ is a degree-$d$ multilinear polynomial that has a $(\tau,N,m,\epsilon)$-regular decomposition for $1/2>\epsilon,\tau>0$, $A$ and $X$ and random Bernoulli and Gaussian variables respectively and $t$ is a real number, then
$$
|\pr(p(A)\leq t) - \pr(p(X)\leq t)| = O_{d,m}( \tau^{1/5}N \log(\tau^{-1})^{dm/2+1} + \epsilon^{1/d}\log(\epsilon^{-1})^{1/2}).
$$
\end{thm}

\begin{rmk}
We can derive a statement very similar to that of Theorem \ref{Z4Thm} from Theorem \ref{DInvPrinThm}.  In particular, if $p$ is multilinear, and $\tau$-regular, we may normalize $p$ so that $\E_X[p(X)]=0,\E_X[p(X)^2]=1.$  Then by Lemma \ref{anticoncentrationLem}, we have that for $h=\textrm{Id}$ and $q=p$, $(h,q)$ is a $(\tau^{1/5},O(d\tau^{(1/d-1)/5}))$-diffuse decomposition of $p$.  Furthermore, by assumption $q$ is multilinear and has all influences at most $\tau$.  Therefore, this is a $(\tau,O(d\tau^{(1/d-1)/5}),1,0)$-regular decomposition of $p$.  Thus, we obtain
$$
|\pr(p(A)\leq t) - \pr(p(X)\leq t)| = O_d (\tau^{1/(5d)}\log(\tau^{-1})^{d/2+1}).
$$
\end{rmk}

Neither invariance principle on its own is very useful for dealing with general polynomial threshold functions which might not satisfy the necessary regularity conditions. Fortunately, in both cases if regularity fails it will be because some small number of coordinates have undo effect on the value of the polynomial.  If this is the case, we can hope to make things better by fixing the values of these coordinates and considering the resulting polynomial over the remaining coordinates, hoping that it is regular.  Theorems confirming this intuition have been known as \emph{regularity lemmas}.  For the standard notions of regularity, various regularity lemmas have appeared in \cite{reg} and \cite{sensitivity} as well as other places.  As an example, \cite{reg} proved the following:

\begin{thm}[Diakonikolas, Servedio, Tan, Wan]
Let $f(x) = sign(p(x))$ be any degree-$d$ PTF. Fix any $\tau > 0$. Then $f$ is equivalent to a decision
tree $T$ , of depth
$$
\textrm{depth}(d,\tau) = \frac{1}{\tau}\cdot (d \log (\tau^{-1}))^{O(d)}
$$
with variables at the internal nodes and a degree-$d$ PTF $f_\rho = sign(p_\rho)$ at each leaf $\rho$, with the following
property: with probability at least $1 -\tau$ , a random path from the root reaches a leaf $\rho$ such that $f_\rho$ is
$\tau$ -close to some $\tau$-regular degree-$d$ PTF.
\end{thm}

Along similar lines, we prove the following:
\begin{thm}[Diffuse Regularity Lemma]\label{difRegThm}
Let $p$ be a degree-$d$ polynomial with Bernoulli inputs.  Let $\tau,c,M>0$ with $\tau<1/2$.  Then $p$ can be written as a decision tree of depth at most
$$
O_{c,d,M}\left(\tau^{-1} \log(\tau^{-1})^{O(d)}\right)
$$
with variables at the internal nodes and a degree-$d$ polynomial at each leaf, with the following property: with probability at least $1-\tau$, a random path from the root reaches a leaf $\rho$ so that the corresponding polynomial $p_\rho$ either satisfies $\var(p_\rho) < \tau^M |p_\rho|_2^2$ or $p_\rho$ has an $(\tau,\tau^{-c},O_{c,d,M}(1),O_{c,d,M}(\tau^M))$-regular decomposition.
\end{thm}

\subsection{The Diffuse Invariance Principle}\label{InvarianceSubSec}

In this section, we prove Theorem \ref{DInvPrinThm}.  We begin with the following proposition:

\begin{prop}\label{invProp}
Let $p$ be a degree-$d$ polynomial with a $(\tau^{1/5},N)$-diffuse decomposition (for $1/2>\tau>0$) $(h,q_1,\ldots,q_m)$ with $q_i$ multilinear so that $\Inf_i(q_j)\leq \tau$ for all $i,j$.  Then if $A$ is a Bernoulli random variable, $X$ a Gaussian random variable and $t$ a real number then
$$
\left| \pr(p(A)\leq t) - \pr(p(X)\leq t)\right| = O_{d,m}(N\tau^{1/5} \log(\tau^{-1})^{dm/2+1}).
$$
\end{prop}

\begin{proof}

It suffices to prove this statement for $t=0$.   We proceed via Proposition \ref{replacementProp}.  We note that for each $i$ the first three moments of $A_i$ agree with the corresponding moments of $X_i$.  Therefore, since the $q_i$ are multilinear, any degree-$3$ polynomial in the $q_i$ has the same expectation under $A$ as under $X$.  Thus, we may apply Proposition \ref{replacementProp} with $k=4$.  We have that
\begin{equation}\label{Z10Eqn}
\left| \pr(p(A)\leq 0) - \pr(p(X)\leq 0)\right| = O_{d,m}(N\tau^{1/5} \log(\tau^{-1})^{dm/2+1} + \tau^{-4/5}T).
\end{equation}

Recall that $T_{i,j}$ is
\begin{align*}
&\E\left[\left(q_i(X_1,\ldots,X_{j-1},A_j,\ldots,A_n) - \E_Y[q_i(X_1,\ldots,X_{j-1},Y,A_{j+1},\ldots,A_n)] \right)^4 \right]\\
+&\E\left[\left(q_i(X_1,\ldots,X_{j},A_{j+1},\ldots,A_n) - \E_Y[q_i(X_1,\ldots,X_{j-1},Y,A_{j+1},\ldots,A_n)] \right)^4 \right]
\end{align*}
By Lemma \ref{mixedHypercontractiveLem} this is at most
$$
O_d\left(\E_{X_1,\ldots,X_{j-1},A_{j+1},\ldots,A_n}[\var_Y(q_i(X_1,\ldots,X_{j-1},Y,A_{j+1},\ldots,A_n))]^2 \right).
$$
Since the polynomial in expectation is at most quadratic in each $X_i$, this is
$$
O_d\left(\E_{A}[\var_Y(q_i(A_1,\ldots,A_{j-1},Y,A_{j+1},\ldots,A_n))]^2 \right)=O_d(\Inf_j(q_i)^2).
$$
Thus, 
\begin{align*}
T & = \sum_{i=1}^m \sum_{j=1}^n T_{i,j}\\
& = O_d\left(\sum_{i=1}^m \sum_{j=1}^n \Inf_j(q_i)^2 \right)\\
& \leq O_d\left(\sum_{i=1}^m \sum_{j=1}^n \tau\Inf_j(q_i) \right)\\
& = O_d\left(\sum_{i=1}^m \tau \right)\\
&= O_{d,m}(\tau).
\end{align*}

Thus, by Equation (\ref{Z10Eqn}),
$$
\left| \pr(p(A)\leq 0) - \pr(p(X)\leq 0)\right| = O_{d,m}(N\tau^{1/5} \log(\tau^{-1})^{dm/2+1}),
$$
as desired.

\end{proof}

Proposition \ref{invProp} is the main analytic tool used in our proof of Theorem \ref{DInvPrinThm}.  From it we can quickly derive the following theorem:

\begin{thm}\label{DInv2Thm}
Let $p$ be a degree-$d$ multilinear polynomial with a $(\tau,N,m,\epsilon)$-regular decomposition (for $1/2>\epsilon,\tau>0$) given by $(h,q_1,\ldots,q_m)$.  Let $p_0(x):=h(q_1(x),\ldots,q_m(x))$.  Let $A$ be a Bernoulli random variable, $X$ a Gaussian random variable, and $t$ a real number.  Then
$$
\left| \pr(p(A)\leq t) - \pr(p_0(X)\leq t) \right| = O_{d,m}(N\tau^{1/5} \log(\tau^{-1})^{dm/2+1} + \epsilon^{1/d}\log(\epsilon^{-1})^{1/2})
$$
\end{thm}

\begin{rmk}
For most applications Theorem \ref{DInv2Thm} will be as good as Theorem \ref{DInvPrinThm} as it shows that the regular polynomial of Bernoullis behaves similarly to a polynomial of Gaussians.  As we shall see later, it will take some work to show that it will necessarily behave like the same polynomial of Gaussians.  This is because although $|p-p_0|_{2,B}$ is small, this does not immediately imply that $|p-p_0|_2$ is sufficiently small for the proof to work.
\end{rmk}

\begin{proof}
As in the proof of Proposition \ref{invProp}, we may assume that $t=0$ and prove the inequality
$$
\pr(p(A)\leq 0) \leq \pr(p_0(X)\leq 0) + O_{d,m}(N\tau^{1/5} \log(\tau^{-1})^{dm/2+1} + \epsilon^{1/d}\log(\epsilon^{-1})^{1/2})
$$

By Corollary \ref{BConcCor} we have with probability $1-\epsilon$ that
$$|p(A)-p_0(A)| \leq O(\epsilon\log(\epsilon^{-1})^{d/2})\sqrt{\var(p_0(X))}\leq O(\epsilon\log(\epsilon^{-1})^{d/2})|p_0|_2.$$
Thus, we have that
\begin{align*}
\pr(p(A)\leq 0) & \leq \epsilon + \pr(p_0(A)\leq O(\epsilon\log(\epsilon^{-1})^{d/2})|p_0|_2)\\
& \leq O_{d,m}(N\tau^{1/5} \log(\tau^{-1})^{dm/2+1} + \epsilon) + \pr(p_0(X) \leq O(\epsilon\log(\epsilon^{-1})^{d/2})|p_0|_2)\\
& \leq O_{d,m}(N\tau^{1/5} \log(\tau^{-1})^{dm/2+1} + \epsilon^{1/d}\log(\epsilon^{-1})^{1/2}) + \pr(p_0(X) \leq 0).
\end{align*}
The second line above is by Proposition \ref{invProp} and the third is by Lemma \ref{anticoncentrationLem}.

The lower bound on $\pr(p(A)\leq 0)$ is proved analogously.
\end{proof}

In order to complete the proof of Theorem \ref{DInvPrinThm} we need the following:

\begin{prop}\label{linErrProp}
If $p$ is a degree-$d$ polynomial with a $(\tau,N,m,\epsilon)$-regular decomposition (for $1/2>\epsilon,\tau>0$) given by $p_0(x)=h(q_1(x),\ldots,q_m(x))$, then for $X$ a Gaussian random variable, and $t$ a real number,
\begin{align*}
|\pr(p(X)\leq t) & - \pr(p_0(X)\leq t)|  \leq O_{d,m}\left( \tau^{1/4}N \log(\tau^{-1})^{d(m+1)/2+1} + \epsilon^{1/d}\log(\epsilon^{-1})^{1/2}\right).
\end{align*}
\end{prop}

The biggest difficulty with proving this Proposition will be dealing with the discrepancy between $p_0$ and $L(p_0)$.  To deal with this, we make the following definition:

\begin{defn}
Let $p_1,\ldots,p_k$ be multilinear polynomials.  Define
$$
A(p_1,\ldots,p_k) = \sum_{S\subseteq \{1,2,\ldots,k\}} (-1)^{|S|} \left(\prod_{i\in S} p_i\right) L\left(\prod_{i\not \in S} p_i \right).
$$
\end{defn}

We note the following:
\begin{lem}\label{Z6Lem}
Let $q_1,\ldots,q_m$ be multilinear polynomials and let $h$ be a degree-$d$ polynomial in $m$ variables then $L(h(q_1(x),\ldots,q_m(x)))$ is
$$
\sum_{k=0}^d \sum_{i_1,\ldots,i_k=1}^m \frac{\partial^k h}{\partial q_{i_1}\cdots \partial q_{i_k}} A(q_{i_1},\ldots,q_{i_k}).
$$
\end{lem}
\begin{proof}
As the above expression is linear in $h$, we may assume that $h$ is a monomial of degree $d$.  In particular we may assume that $h = q_{i_1}q_{i_2}\cdots q_{i_d}$ (note that some of the indices $i_j$ might coincide).  The expression in question then becomes:
\begin{align*}
&\sum_{T=\{t_1,\ldots,t_k\}\subseteq \{1,\ldots,d\}} \left(\prod_{j\not\in T} q_{i_j}\right) A(q_{i_{t_1}},\ldots,q_{i_{t_k}})\\
= & \sum_{T\subseteq \{1,\ldots,d\}} \left(\prod_{j\not\in T} q_{i_j}\right) \sum_{S\subseteq T} (-1)^{|S|} \left(\prod_{i\in S} q_i\right) L\left(\prod_{i\in T\backslash S} q_i \right)\\
= & \sum_{S\subseteq T\subseteq \{1,\ldots,d\}}\left( \prod_{ j \not \in T\backslash S } q_{i_j}\right)  L\left(\prod_{j\in T\backslash S} q_{i_j} \right).
\end{align*}
Letting $R=T\backslash S$, this is
\begin{align*}
&\sum_{R\subseteq \{1,\ldots,d\}} L\left( \prod_{j \in R} q_{i_j}\right) \left(\prod_{j \not\in R} q_{i_j}\right) \sum_{S\in \{1,\ldots,d\}\backslash R} (-1)^{|S|}\\
= & \sum_{R = \{1,\ldots,d\}} L\left( \prod_{j \in R} q_{i_j}\right) \left(\prod_{j \not\in R} q_{i_j}\right) \\
= & L\left( \prod_{j \in \{1,\ldots,d\}} q_{i_j}\right)\\
= & L(h),
\end{align*}
as desired.
\end{proof}

To control the discrepancy between $p_0$ and $L(p_0)$ it now suffices to prove the following:

\begin{prop}\label{ABoundProp}
Let $p_1,\ldots,p_k$ be multilinear, degree at most $d$ polynomials with $\Inf_i(p_j) \leq \tau$ for all $i,j$,  $|p_j|_2 \leq 1$ for all $j$.  Then
$$
|A(p_1,\ldots,p_k)|_2 = O_{k,d}(\tau^{k/4}).
$$
\end{prop}
\begin{proof}
We proceed by bounding the expected value of $A(p_1,\ldots,p_k)^2$.  In particular, we show that if $p_j$ are multilinear degree-$d$ polynomials of norm at most 1 with all influences at most $\tau$ then
$$
\E[A(p_1,\ldots,p_k)(X)A(p_{k+1},\ldots,p_{2k})(X)] = O_{k,d}(\tau^{k/2}).
$$
We note that the above expression is linear in the $p_j$.  We may therefore rewrite it as a sum over sequences of monomials $m_1,\ldots,m_{2k}$ where $m_j$ is a monomial of $p_j$, of
$$
\E[A(m_1,\ldots,m_k)(X)A(m_{k+1},\ldots,m_{2k})(X)].
$$

To each such sequence of monomials $m_1,\ldots,m_{2k}$ we associate a \emph{repeat pattern}, which is the multiset of non-empty subsets of $\{1,2,\ldots,2k\}$ whose elements correspond to $\{j:x_i \textrm{ appears in monomial }m_j\}$ for all $i$ so that $x_i$ appears in any of the monomials $m_j$.  We break up the above sum into parts based on the repeat pattern satisfied by $m_1,\ldots,m_{2k}$, since there are $O_{k,d}(1)$ such possible patterns, it suffices to prove our bound for the sum of all terms coming from each such pattern.  It particular we need to show that for any repeat pattern $P$ that
\begin{align}\label{PatternSumEqn}
\sum_{\substack{m_j\textrm{ a monomial from }p_j \\ (m_1,\ldots,m_{2k}) \textrm{ has repeat pattern }P}}&\E[A(m_1,\ldots,m_k)(X)A(m_{k+1},\ldots,m_{2k})(X)] = O_{k,d}(\tau^{k/2}).
\end{align}

Note that if the repeat pattern contains any subset of odd size that the resulting sum will be 0.  This is because for any $m_1,\ldots,m_{2k}$ with this repeat pattern, there will be some $x_i$ appearing in an odd number of the $m_j$.  This means that the product of the $m_j$ will be an odd function of $x_i$.  Since $L$ of an odd polynomial is still odd, this means that $A(m_1,\ldots,m_k)A(m_{k+1},\ldots,m_{2k})$ will be an odd function of $x_i$ and thus, have expectation 0.

Furthermore, suppose that given $P$, there is some $1\leq j\leq 2k$ so that $j$ does not appear in any element of $P$ of size greater than 2.  We claim again that for any $m_1,\ldots,m_{2k}$ satisfying $P$ that $\E[A(m_1,\ldots,m_k)(X)A(m_{k+1},\ldots,m_{2k})(X)]=0$.  To show this we assume without loss of generality that $j=1$.  We expand out the $A$'s to get that the expression in question is the expectation of
$$
\sum_{S\subseteq \{1,2,\ldots,k\}} \sum_{T\subseteq \{k+1,\ldots,2k\}} (-1)^{|S|+|T|}\left( \prod_{j\in S\cup T} p_j\right) L\left( \prod_{j\in \{1,\ldots,k\}\backslash S} p_j \right)L\left( \prod_{j\in \{k+1,\ldots,2k\}\backslash T} p_j \right).
$$
We claim that if we toggle whether $1$ is in $S$ in the above sum, it has no effect on the expectation of the resulting product other than to negate the $(-1)^{|S|+|T|}$ term.  This is because adding $1$ to $S$ can only have the effect of removing some $x_i^2$ terms from the resulting monomial.  On the other hand since $\E[1]=\E[X_i^2]$, this does not effect the resulting expectation.  Thus, the expectations of the terms with $1$ in $S$ cancel the expectations of the terms with $1$ not in $S$, leaving us with expectation 0.

It thus suffices to consider Equation (\ref{PatternSumEqn}) when all elements of $P$ have even order and so that for each $1\leq j\leq 2k$ there is some element of $P$ of order at least 4 containing $j$.  For such $P$ we upper bound the left hand side of Equation (\ref{PatternSumEqn}) by
\begin{equation}\label{PatternSumEqn2}
\sum_{\substack{m_j\textrm{ a monomial from }p_j \\ (m_1,\ldots,m_{2k}) \textrm{ has repeat pattern }P}}O_{k,d}\left(\prod_{j=1}^{2k} |m_j|_2\right).
\end{equation}

We will now prove the following statement, which will imply our desired bound.  Let $p_1,\ldots,p_{2k}$ be multilinear polynomials with $|p_j|\leq 1$ and $T\subseteq \{1,2,\ldots,2k\}$ some set so that $\Inf_i p_j \leq \tau$ for all $i$ and all $j\in T$.  Furthermore, let $P$ be a repeat pattern all of whose elements have even order and so that each element of $T$ appears in some element of $P$ of order at least 4, then the expression in Equation (\ref{PatternSumEqn2}) is at most $O_{k,d}(\tau^{|T|/4})$.  We prove this by induction on $|P|$.  The base case where $|P|=0$ is trivial since then we are considering only the term where all of the $m_j$ are constants.

If $|P|>0$, we consider an element of $P$ of maximal size.  In particular, if $T\neq \emptyset$, this implies that this element is of size at least 4.  Without loss of generality this element is $\{1,2,\ldots,2\ell\}$.  We break our sum into pieces based on which coordinate is shared by all of $m_1,\ldots,m_{2\ell}$ (if more than one coordinate is shared by each of these elements we will count all of them leading to a strictly larger sum).  If we wish to compute the sum over all terms where they share a coordinate $x_i$ we find that it is
$$
\sum_{\substack{m_j\textrm{ a monomial from }p_j' \\ (m_1,\ldots,m_{2k}) \textrm{ has repeat pattern }P'}}O_{k,d}\left(\prod_{j=1}^2k |m_j|_2\right).
$$
Where above $p_j'=p_j$ for $j>2\ell$ and for $j\leq 2\ell$, $p_j'$ consists of the sum of the monomials in $p_j$ containing $x_i$ divided by $x_i$, and $P'$ is $P$ minus $\{1,2,\ldots,2\ell\}$.  Furthermore note that $|p_j'|_2 = \sqrt{\Inf_i(p_j)}$ for $j\leq 2\ell$.  Letting $p_j''$ be the normalized version of $p_j'$, the above is at most
$$
\prod_{j=1}^{2\ell} \sqrt{\Inf_i(p_j)}\sum_{\substack{m_j\textrm{ a monomial from }p_j" \\ (m_1,\ldots,m_{2k}) \textrm{ has repeat pattern }P'}}O_{k,d}\left(\prod_{j=1}^2k |m_j|_2\right).
$$
Letting $T'$ = $T\backslash\{1,2,\ldots,2\ell\}$, we note that this sum is of the form specified for the value $T'$, hence we have by the inductive hypothesis that the above sum is
$$
O_{k,d}\left( \tau^{|T'|/4} \prod_{j=1}^{2\ell} \sqrt{\Inf_i(p_j)} \right).
$$
It thus suffices to prove that
$$
\sum_i \prod_{j=1}^{2\ell} \sqrt{\Inf_i(p_j)} = O_{k,d}\left(\tau^{(|T|-|T'|)/4} \right) =  O_{k,d}\left(\tau^{(|T\cap\{1,2,\ldots,2\ell\}|)/4} \right).
$$
We assume without loss of generality that $T\cap\{1,2,\ldots,2\ell\} = \{1,2,\ldots,a\}$.  We note by Cauchy-Schwarz that
$$
\sum_i \prod_{j=1}^{2\ell} \sqrt{\Inf_i(p_j)} \leq \left(\prod_{j=1}^{2\ell-2} \max_i \sqrt{\Inf_i(p_j)} \right)\left(\prod_{j=2\ell-1}^{2\ell} \sum_i \Inf_i(p_j) \right)^{1/2}.
$$
We note that for each of the last two terms that
$$
\sum_i \Inf_i(p_j) = O_d(|p_j|_2^2) = O_d(1).
$$
Furthermore, we have that
$$
\prod_{j=1}^{2\ell-2} \max_i \sqrt{\Inf_i(p_j)} \leq \prod_{j=1}^{\min(a,2\ell-2)} \tau^{1/2} \prod_{j=a+1}^{2\ell-2} 1 = \tau^{\min(a,2\ell-2)/2}.
$$
Thus, we have that
$$
\sum_i \prod_{j=1}^{2\ell} \sqrt{\Inf_i(p_j)}  \leq O_d(\tau^{\min(a,2\ell-2)/2}) = O_d(\tau^{a/4}).
$$
With the last step following from the observation that either $a=0$ or $\ell\geq 2$.  This completes our inductive step and proves our proposition.

\end{proof}

We are now prepared to prove Proposition \ref{linErrProp} and thus, Theorem \ref{DInvPrinThm}.

\begin{proof}
We may clearly assume that $t=0$.  We will give a series of high probability statements that together imply that
$$
\sgn(p(X)) = \sgn(p_0(X)).
$$
Let $V = \var(p_0)$.

First note that by assumption
$$
|p-L(p_0)|_2 = |p-p_0|_{2,B} \leq \epsilon V.
$$
Thus, by Corollary \ref{ConcentrationCor} we have for some sufficiently large $C$ that with probability $1-\epsilon$ that
$$
|p(X)-L(p_0)(X)| \leq C \epsilon\log(\epsilon^{-1})^{d/2} V.
$$

Additionally, by Lemma \ref{anticoncentrationLem}, we have with probability $1-O(d\epsilon^{1/d}\log(\epsilon^{-1})^{1/2})$ that
$$
|p_0(X)| \geq 2 C \epsilon\log(\epsilon^{-1})^{d/2} |p_0|_2 \geq 2 C \epsilon\log(\epsilon^{-1})^{d/2} V.
$$

By Proposition \ref{ABoundProp} and Corollary \ref{ConcentrationCor} we have that for $C$ a sufficiently large number given $d$ that with probability $1-O_{d,m}(\tau)$ that for all $1\leq i_1,i_2,\ldots,i_k \leq m$ for $k\leq d$ that
$$
|A(q_{i_1},\ldots,q_{i_k})(X)| \leq C \tau^{k/4}\log(\tau^{-1})^{dk/2}.
$$

Finally, by Lemma \ref{DDDerLem} we have that with probability $1-O_{d,m}(\tau^{1/4}N \log(\tau^{-1})^{d(m+1)/2+1})$, letting $x=(q_1(X),\ldots,q_m(X))$ that
\begin{align*}
|h(x)| & \geq 3 C m \tau^{1/4} \log(\tau^{-1})^{d/2} |D_{i_1} h(x)|_2 \geq 3^2 C^2 m^2 \tau^{2/4} \log(\tau^{-1})^{d2/2} |D_{i_1}D_{i_2} h(x)|_2 \\ &\geq \ldots \geq 3^d C^d m^d \tau^{d/4} \log(\tau^{-1})^{d^2/2} |D_{i_1}\cdots D_{i_d} h(x)|_2.
\end{align*}

Assuming that all of the above hold, then
$$
|p(X)-p_0(X)| \leq |p(X) - L(p_0)(X)| + |p_0(X) - L(p_0)(X)| \leq |p_0(X)|/2 + |p_0(X) - L(p_0)(X)|.
$$
By Lemma \ref{Z6Lem}, we have that letting $x=(q_1(X),\ldots,q_m(X))$
\begin{align*}
|L(p_0)(X) - p_0(X)| & =
\left|\sum_{k=1}^d \sum_{i_1,\ldots,i_k=1}^m A(q_{i_1},\ldots,q_{i_k})(X) D_{i_1}\cdots D_{i_k} h(x)\right|\\
& \leq \sum_{k=1}^d \sum_{i_1,\ldots,i_k=1}^m C \tau^{k/4}\log(\tau^{-1})^{dk/2} |D_{i_1}\cdots D_{i_k} h(x)|_2\\
& \leq \sum_{k=1}^d \sum_{i_1,\ldots,i_k=1}^m 3^{-k} m^{-k} |h(x)|\\
& \leq \sum_{k=1}^d 3^{-k} |p_0(X)|\\
& < |p_0(X)|/2.
\end{align*}
Combining this with the above we find that
$$
|p(X)-p_0(X)| < |p_0(X)|/2 + |p_0(X)|/2 = |p_0(X)|.
$$
Thus, with probability at least
$$
1-O_{d,m}\left( \tau^{1/4}N \log(\tau^{-1})^{d(m+1)/2+1} + \epsilon^{1/d}\log(\epsilon^{-1})^{1/2} \right)
$$
that $\sgn(p(X))=\sgn(p_0(X)).$
\end{proof}

\subsection{The Regularity Lemma}\label{RegSec}

In this section, we will prove Theorem \ref{difRegThm}.  Much of it will be along the lines of the proof of Theorem \ref{DDTheorem} with some extra work being done to ensure that the resulting $q_i$ are regular.  We begin with a lemma on the regularity of restrictions of polynomials.

\begin{lem}\label{regResLem}
Let $p$ be a degree-$d$ multilinear polynomial with $|p|_2\leq 1$.  Let $1/2>\epsilon>0$ be a real number.  Then there exists an $M=O_d( \epsilon^{-1} \log(\epsilon^{-1})^{d})$ so that for any set $S$ of coordinates containing the $M$ coordinates of highest influence for $p$, if we let $A$ be a random Bernoulli variable over the coordinates in $S$ and let $p_A$ be the polynomial over the remaining coordinates upon plugging these values into the coordinates of $S$ then with probability $1-\epsilon$
$$
\max_i(\Inf_i(p_A)) \leq \epsilon.
$$
\end{lem}
\begin{proof}
We assume throughout that $\epsilon$ is sufficiently small.  Note that the sum of the influences of $p$ is $O_d(1)$, therefore if $M$ is a sufficiently large multiple of $\epsilon^{-1}\log(\epsilon^{-1})^{d}$, we have that the largest influence of a coordinate not in $S$ is at most a small constant times $\epsilon \log(\epsilon^{-1})^{-d}$.  Note that for each $i\not\in S$, there is a polynomial $p_i$ of degree at most $d$ so that $\Inf_i(p_A) = p_i(A)^2$.  Furthermore, it is easy to check that $\E[p_i(A)^2] = \Inf_i(p)$.  Applying Corollary \ref{ConcentrationCor} we find that if $M$ were chosen to be sufficiently large, then with probability at most $\epsilon^{4}/2$ is any given $\Inf_i(p_A)$ more than $\epsilon$.  Taking a union bound over $i$, we find that with probability at most $\epsilon/2$ is some $\Inf_i(p_A)>\epsilon$ for any $i$ with $\Inf_i(p) > d\epsilon^3$.  Consider the polynomial
$$
q(A) = \sum_{j:\Inf_j(p)\leq d\epsilon^3} \Inf_j(p_A)^2.
$$
Note that $|\Inf_j(p_A)^2|_2 = O_d(\Inf_j(p)^2)$ by Lemma \ref{hypercontractiveLem}.  Thus, 
$$
|q|_2 \leq O_d(1)\sum _{j:\Inf_j(p)\leq d\epsilon^3} \Inf_j(p)^2 \leq O_d(\epsilon^3)\sum _{j:\Inf_j(p)\leq d\epsilon^3} \Inf_j(p) = O_d(\epsilon^3).
$$
Thus, by Corollary \ref{ConcentrationCor} $q(A) > \epsilon^2$ with probability at most $\epsilon/2$.  On the other hand, if $q(A)\leq \epsilon^2$, it implies that $\Inf_j(p_A)\leq \epsilon$ for all $j$ so that $\Inf_j(p)\leq d\epsilon^3$.  Thus, with probability at most $\epsilon$ is any $\Inf_j(p_A)$ more than $\epsilon$.
\end{proof}

\begin{lem}\label{resBoundLem}
Let $p$ be a degree-$d$ multilinear polynomial.  Let $S$ be a set of coordinates and $A$ a Bernoulli random variable over those coordinates.  Let $p_A$ be the restricted polynomial when the coordinates of $A$ are plugged into $p$.  Then
$$
\pr(|p_A|_2 \geq N |p|_2) = O_d\left( 2^{\log(N)^{1/d}} \right).
$$
\end{lem}
\begin{proof}
Note that $|p_A|_2^2$ is a polynomial in $A$ of degree at most $2d$.  Note that the squared $L_2$ norm of this polynomial is
$$
\E_A[\E_B[p(A,B)^2]^2] \leq \E_{A,B}[p(A,B)^4] = |p|_{4,B}^4 \leq O_d(|p|_2^4).
$$
The result now follows from Corollary \ref{BConcCor}.
\end{proof}

The main parts of the proof of Theorem \ref{difRegThm} are contained in the following proposition

\begin{prop}\label{Z7Prop}
Let $p$ be a degree-$d$ multilinear polynomial and let $\epsilon,c,M>0$ for $1/2>\epsilon$.  Then $p$ can be written as a decision tree of depth
$$
O_{c,d,M}(\epsilon^{-1}\log(\epsilon^{-1})^{O(d)})
$$
with coordinate variables for internal nodes and polynomials for leaves so that for a random leaf $p_\rho$ we have with probability $1-O_{c,d,M}(\epsilon)$ that there exists a $p_0$ with $|p-p_0|_{2,B} \leq \epsilon^N |p|_2$ and so that $p_0$ has an $(\epsilon,\epsilon^{-c})$-diffuse decomposition $(h,q_1,\ldots,q_m)$ with $m=O_{c,d,N}(1)$, $q_i$ multilinear and so that $\Inf_j(q_i) \leq \epsilon$ for each $i,j$.
\end{prop}
\begin{proof}
The proof is along the same lines as the proof of Theorem \ref{DDTheorem}, with some extra work done to ensure that the influences can be controlled.  We assume that $|p|_2=1$ and assume throughout that $\epsilon$ is sufficiently small.

We define:
a \emph{partial decomposition} of our polynomial $p$ to be a set of the following data:
\begin{itemize}
\item A positive integer $m$.
\item A polynomial $h:\R^m\rightarrow \R$.
\item A sequence of multilinear polynomials $(q_1,\ldots,q_m)$ each on $\R^n$ with $|q_i|_2=1$ for each $i$.
\item A sequence of integers $(a_1,\ldots,a_m)$ with $a_i$ between 0 and $4 \cdot 3^i (N+1) / c - 1$.
\end{itemize}
Furthermore, we require that each $q_i$ is non-constant, and that for any monomial $\prod x_{i}^{\alpha_1}$ appearing in $h$ that $\sum \alpha_i \deg(q_i) \leq d$.

We say that such a partial decomposition has complexity at most $C$ if the following hold:
\begin{itemize}
\item $m \leq C$.
\item $|h|_2 \leq C \epsilon^{-1+C^{-1}}$.
\item $|p(A) - h(\epsilon^{a_i c/(2\cdot 3^i)}q_i(A))|_{2,B} \leq C\epsilon^{N+1}\log(\epsilon^{-1})^C$.
\end{itemize}

We define the weight of a partial decomposition as follows.  First we define the polynomial
$$
w(x)=\sum_{i=1}^m x^{\deg(q_i)} (4\cdot 3^i (N+1)/c - a_i).
$$
We then let the weight of the decomposition be $w(\omega)$.

We prove by ordinal induction on $w$ that if $p$ has a partial decomposition of weight $w$ and complexity $C$, then there is a decision tree of depth
$O_{c,C,d,N,w}(\epsilon^{-1}\log(\epsilon^{-1})^{O(d)})$ so that with probability $1-O_{c,C,d,w,N}(\epsilon)$ a random leaf has such a $p_0$ with a diffuse decomposition into multilinear polynomials whose influences are at most $\epsilon$.

Again the idea of the proof is to show that after a decision tree of appropriate depth and with appropriate probability, that we either have such a $p_0$ or that we have a partial decomposition with smaller weight.  By Lemma \ref{regResLem}, if we restrict to random values of the $O_d(\epsilon^{-1}\log(\epsilon^{-1})^{O(d)})$ highest influence coordinates of each of the $q_i$, we will have all influences of all of the $q_i$ at most $\epsilon$ with probability $1-O_{d,m}(\epsilon)$.  Applying Lemma \ref{resBoundLem} to the $q_i$ and $p-h(q_1,\ldots,q_m)$ we find that with probability $1-O_{d,m}(\epsilon)$ that the restricted values of $q_i$ have norm at most $\log(\epsilon^{-1})^{O(d)}$ and that the $L^2$ norm of $p-h(q_1,\ldots,q_m)$ increased by at most a similar factor.  Thus, rescaling the $q_i$ and modifying $h$ appropriately, we find that with probability $1-O_{d,m}(\epsilon)$ over our restrictions, we have a new partial decomposition of weight $w$ and complexity $O_C(1)$ so that $\Inf_i(q_j)\leq \epsilon$ for each $i$ and $j$.  As in Lemma \ref{Z2Lem}, we show that either $(h,q_1,\ldots,q_m)$ is an $(\epsilon,\epsilon^{-c})$-diffuse set or that we have a partial decomposition of strictly smaller weight and with complexity $O_{c,C,d,N}(1)$.  The proof follows through identically to the proof in Lemma \ref{Z2Lem} with the additional caveat that the $A_\ell,B_\ell$ can be chosen to be multilinear.  This is because the $q_i$ are multilinear, so keeping only the multilinear parts of the $A_\ell,B_\ell$ only reduces the error produced by the approximation.  This completes the proof.

\end{proof}

We will need one more lemma about decision trees of polynomials before we proceed.

\begin{lem}\label{largeResLem}
Let $p$ be a multilinear, degree-$d$ polynomial.  Let $T$ be some decision tree over it's coordinates.  If $T$ is evaluated making random, independent choices at each step, and the restricted function is called $p_\rho$, then with probability at least $2^{O(d)}$ over these choices we have that
$$
|p_\rho|_2 \geq |p|_2/2.
$$
\end{lem}
\begin{proof}
Given a partially filled-in decision tree $T'$ define $V(T') = \E[p(A)^2|T']$.  It is clear that $V$ is a martingale.  Therefore $V^2$ is a submartingale.  In particular, this means that the expectation of $V^2$ over some decision tree is at most the expectation over an extended decision tree that eventually decides values for all coordinates.  This latter expectation is $|p|_{4,B}^4 = 2^{O(d)}|p|_2^4$.  Therefore, the expectation over fills of $T$ of $V$ is $|p|_2^2$ and the expectation of $V^2$ is at most $2^{O(d)}|p|_2^4$.  Therefore by the Paley-Zygmund inequality with probability at least $2^{O(d)}$ we have that $V\geq |p|_2^2/4$, proving our lemma.
\end{proof}

We are now prepared to prove Theorem \ref{difRegThm}.
\begin{proof}
We claim that for $\tau$ sufficiently small that a correctly constructed decision tree of depth $O_{c,d,M}(\tau^{-1}\log(\tau^{-1})^{O(d)})$ yields a restriction with the desired property with probability at least $2^{O(d)}$.  Repeating this process up to $2^{O(d)}\log(\tau^{-1})$ many times upon failure will guarantee an aggregate success probability of $1-\tau$.

To do this we construct the decision tree given by Proposition \ref{Z7Prop} for $N=M+d+2$ and $\epsilon=\tau$.  We claim that if the restricted polynomial has $L^2$ norm at least $|p|_2/2$ (which happens with probability $2^{O(d)}$ by Lemma \ref{largeResLem}), then the resulting polynomial has the desired property.

Let $P$ be the resulting polynomial.  We have a polynomial $p_0$ with an appropriate diffuse decomposition into multilinear polynomials with sufficiently small influences and so that $|P-p_0|_{2,B} = O_{c,d,M}(\tau^{M+d+2})|P|_2$.  If $\var(p_0) \geq \tau^{M+d}|P|_2^2$, we have an appropriate regular decomposition.  Otherwise, $\var(p_0)\leq \tau^{M+d}|P|_2^2$.  This implies that for some $\mu$ that $|p_0-\mu|_2^2 \leq \tau^{M+d}|P|_2^2$.  Thus, by Lemma \ref{hSmallLem} we have that $|h-\mu|_2^2 \leq O_{c,d,M}(\tau^{M})|P|_2^2$.  From this it is easy to see that the sum of the squares of the coefficients of $h-\mu$ is $O_{c,d,M}(\tau^{M})|P|_2^2.$  From this it is easy to verify that the variance of $p_0$ over Bernoulli inputs is $O_{c,d,M}(\tau^M)|P|_2^2$.  Therefore, due to the small difference between $p$ and $p_0$ under Bernoulli inputs, we have that $\var(P)\leq O_{c,d,M}(\tau^M)|P|_2^2$, which satisfies one of the necessary conditions.

\end{proof}

\section{Application to Noise Sensitivity of Polynomial Threshold Functions}\label{GLSec}

\subsection{Background of Noise Sensitivity Results}

\subsubsection{Definitions}

If $f:\R^n\rightarrow\{-1,1\}$ is a boolean function, the noise sensitivity of $f$ is a measure of the likelihood that a small change in the input value to $f$ changes the output.  There are several different notions of noise sensitivity, suitable for slightly different contexts.  We present their definitions here.

\begin{defn}
For $f:\R^n\rightarrow\{-1,1\}$ a boolean function, we define its \emph{average sensitivity} to be
$$
\as(f):=\sum_{i=1}^n\pr_{A\sim_u \{-1,1\}^n}( f(A) \neq f(A^{(i)}) ),
$$
where $A^{(i)}$ is obtained from $A$ by flipping the sign of the $i^{\tth}$ coordinate.  In other words, the average sensitivity is the expected number of coordinates of $A$ that could be changed in order to change the value of $f$.
\end{defn}

We also define the average sensitivity in the Gaussian setting:
\begin{defn}
For $f:\R^n\rightarrow\{-1,1\}$ a boolean function, we define its \emph{Gaussian average sensitivity} to be
$$
\gas(f):=\sum_{i=1}^n \pr(f(X) \neq f(X^{(i)})),
$$
where above $X$ is a Gaussian random variable and $X^{(i)}$ is obtained from $X$ by replacing the $i^{\tth}$ coordinate by an independent random Gaussian.
\end{defn}

A related notion is that of noise sensitivity in the Bernoulli or Gaussian context.  Whereas average sensitivity counts the expected number of coordinates that could be changed to alter the sign of $f$, noise sensitivity measures the probability that the sign of $f$ changes if each coordinate is changed by a small amount.  In particular we define:
\begin{defn}
For $f:\R^n\rightarrow\{-1,1\}$ a boolean function, and $1\geq \delta \geq 0$ we define the \emph{noise sensitivity of $f$ with parameter $\delta$} to be
$$
\ns_\delta(f):=\pr(f(A)\neq f(B)),
$$
where $A$ and $B$ are Bernoulli random variables with $B$ obtained from $A$ by flipping the sign of each coordinate randomly and independently with probability $\delta$.
\end{defn}

\begin{defn}
For $f:\R^n\rightarrow\{-1,1\}$ a boolean function, and $1\geq \delta \geq 0$ we define the \emph{Gaussian noise sensitivity of $f$ with parameter $\delta$} to be
$$
\gns_\delta(f):=\pr(f(X)\neq f(Y)),
$$
where $X$ and $Y$ are Gaussian random variables that together form a joint Gaussian with $$\cov(X_i,Y_j)=\begin{cases}(1-\delta) & \textrm{ if }i=j\\ 0 &\textrm{ otherwise}\end{cases}.$$
\end{defn}

\subsubsection{Previous Work}

The main Conjecture about the noise sensitivity of polynomial threshold functions was given in \cite{gl}

\begin{conj}[Gotsman-Linial]\label{GLConj}
Let $f$ be a degree-$d$ polynomial threshold function in $n$ variables, then
$$
\as(f) \leq 2^{-n+1} \sum_{k=0}^{d-1} \binom{n}{\lfloor (n-k)/2 \rfloor}(n-\lfloor (n-k)/2 \rfloor).
$$
\end{conj}
\begin{rmk}
It should be noted that the upper bound conjectured above is actually obtainable.  In particular, if $f$ is the polynomial threshold function associated to the polynomial
$$
\prod_{i=1}^d \left(\sum_{j=1}^n A_j - d +2i - 1/2 \right)
$$
achieves this bound.
\end{rmk}

In particular, Conjecture \ref{GLConj} implies that
$$
\as(f) = O(d\sqrt{n}).
$$
By the work of \cite{sens2}, this implies bounds on other notions of sensitivity.  In particular it would imply that
$$
\ns_\delta(f) = O(d\sqrt{\delta})
$$
and
$$
\gns_\delta(f) = O(d\sqrt{\delta}).
$$
Furthermore, this would imply the following bound on the Gaussian average sensitivity
$$
\gas(f) = O(d\sqrt{n}).
$$
In particular, we have:
\begin{lem}\label{GasASBoundLem}
The largest Gaussian average sensitivity of any degree-$d$ polynomial threshold function in $n$ variables is at most the largest average sensitivity of a degree-$d$ polynomial threshold function in $n$ variables.
\end{lem}
\begin{proof}
We will show that if $f$ is a degree-$d$ PTF in $n$ variables, then $\gas(f)$ can be written as an expectation over the average sensitivities of certain other degree-$d$ PTFs in $n$ variables.  The key to this argument is to produce the correct distribution on pairs of Gaussians that differ in exactly one coordinate in an unusual way.  In particular, we define $n$-variable Gaussians $Z$ and $Z'$ as follows:
$$
Z_i = \frac{1}{\sqrt{2}}(X_i + A_i Y_i), Z_i' = \frac{1}{\sqrt{2}}(X_i + B_i Y_i)
$$
where $X_i,Y_i$ are independent Gaussian random variables, and $A=(A_1,\ldots,A_n),B=(B_1,\ldots,B_n)$ are Bernoulli random variables that differ only in a single random coordinate.  It is clear that $Z$ and $Z'$ are random Gaussians that agree in all but one of their coordinates, and that they are independent in the coordinate on which they differ.  Thus, 
$$
\gas(f) = \pr(f(Z) \neq f(Z')).
$$
On the other hand, after fixing values of $X$ and $Y$, we may define a new degree-$d$ PTF $f_{X,Y}$ by $$f_{X,Y}(A) := f\left(\frac{1}{\sqrt{2}}(X_i + A_i Y_i) \right).$$ Therefore, we have that
\begin{align*}
\gas(f) &= \pr(f(Z) \neq f(Z'))\\
&= \E_{X,Y}\left[\pr(f_{X,Y}(A) \neq f_{X,Y}(B)) \right]\\
&= \E_{X,Y}[\as(f_{X,Y})].
\end{align*}
This is at most the maximum possible average sensitivity of a degree-$d$ PTF in $n$ variables.
\end{proof}

Proving the conjectured bounds for the various notions of sensitivity has proved to be quite difficult.  The degree-1 case of Conjecture \ref{GLConj} was known to Gotsman and Linial.  The first non-trivial bounds for higher degrees were obtained independently by \cite{sens2} and \cite{sense3}, who later combined their papers into \cite{sensitivity}.  They essentially proved bounds on average sensitivities of $O_d(n^{1-1/O(d)})$ and bounds on noise sensitivities of $O_d(\delta^{1/O(d)})$.  For the special case of Gaussian noise sensitivity, the author proved essentially optimal bounds in \cite{GNSBound} of $O(d\sqrt{\delta})$.  In this section, we improve on these bounds and in particular show that $\as(f) = O_{c,d}(n^{5/6+c}).$ Our basic technique will be to compare $\ns_\delta(f)$ to $\gns_{2\delta}(f)$ using an appropriate invariance principle.  It should be noted that this idea could have been applied using traditional means, but that the bound obtained would not have been better than $\delta^{1-O(1/d)}$.

\subsection{Noise Sensitivity Bounds}

In this section, we prove the following three theorems:

\begin{thm}\label{NSBoundThm}
If $f$ is a degree-$d$ polynomial threshold function, and if $c,\delta>0$, then
$$
\ns_\delta(f) = O_{c,d}(\delta^{1/6-c}).
$$
\end{thm}

\begin{thm}\label{ASBoundThm}
If $f$ is a degree-$d$ polynomial threshold function in $n$ variables, and if $c>0$, then
$$
\as(f) = O_{c,d}(n^{5/6+c}).
$$
\end{thm}

\begin{thm}\label{gasBoundCor}
For $f$ a degree-$d$ polynomial threshold function in $n$ variables and $c>0$,
$$
\gas(f) = O_{c,d}(n^{5/6+c}).
$$
\end{thm}

We begin with the proof of Theorem \ref{NSBoundThm} in the case of regular polynomial threshold function.

\begin{prop}\label{regNSBoundProp}
Let $f=\sgn\circ p$ be a polynomial threshold function for $p$ a degree-$d$ polynomial with a $(\tau,N,m,\epsilon)$-regular decomposition for $1/2>\epsilon,\tau>0$.  Let $1>\delta>0$, then
$$
\ns_\delta(f) = O(d\sqrt{\delta}) + O(d\epsilon^{1/2d}\log(\epsilon^{-1})) + O_{d,m}(N\tau^{1/5}\log(\tau^{-1})^{dm/2+1}).
$$
\end{prop}

The proof of Proposition \ref{regNSBoundProp} will be to use the replacement method to show that $\ns_\delta(f)$ is approximately $\gns_{2\delta}(f)$, which we bound using the main theorem of \cite{GNSBound}.  Unfortunately, we will not be able to apply Proposition \ref{replacementProp} directly, but many of the techniques will be similar.

\begin{proof}

Let $A^1,A^2$ be a pair of Bernoulli random variables so that for each coordinate $i$, $A^1_i$ and $A^2_i$ are equal with probability $1-\delta$ independently over different $i$.  $\ns_\delta(f)=\pr(f(A^1)\neq f(A^2)) = 2 \pr(f(A^1)=1, f(A^2)=-1)$.  We wish to bound this later probability.

Let $X^1$ and $X^2$ be Gaussian random variables so that the joint distribution $(X^1,X^2)$ is a Gaussian with
$$\textrm{Cov}(X^1_i,X^2_j) = \begin{cases} 1-2\delta &\textrm{ if } i=j\\ 0 &\textrm{ Otherwise}\end{cases} $$
Note that all of the first three moments of $(A^1,A^2)$ are identical to the corresponding moments of $(X^1,X^2).$

We are given that there exists a polynomial $p_0$ with $|p-p_0|_{2,B}^2 < \epsilon \var(p_0)$ so that $p_0$ has a $(\tau,N)$-diffuse decomposition $(h,q_1,\ldots,q_m)$ with $q_i$ multilinear and $\Inf_i(q_j)\leq \tau$ for all $i,j$.  After rescaling these polynomials, we may assume that $\var(p_0)\leq |p_0|_2^2 = 1$. Note that by Corollary \ref{BConcCor} that with probability $1-O(\epsilon)$ that $|p(A^i) - p_0(A^i)| < \epsilon^{1/2} \log(\epsilon^{-1})^d $ for each of $i=1,2$.  By Proposition \ref{ApproximationProp} there exist functions $f^1,f^2:\R^m\rightarrow [0,1]$ so that:
\begin{itemize}
\item $f^1(x)=1$ if $h(x)+\epsilon^{1/2}\log(1+\epsilon^{-1})^d>0$.
\item $f^2(x)=1$ if $h(x)-\epsilon^{1/2}\log(1+\epsilon^{-1})^d<0$.
\item \begin{align*}\left| \E[f^1(q_1(X^1),\ldots,q_m(X^1))]  - \E[I_{(0,\infty)}(h(q_1(X^1),\ldots,q_m(X^1))+\epsilon^{1/2}\log(1+\epsilon^{-1})^d)]\right| \\ = O_{d,m}(N\tau^{1/5}\log(\tau^{-1})^{dm/2+1}).\end{align*}
\item \begin{align*}\left| \E[f^2(q_1(X^2),\ldots,q_m(X^2))]  - \E[I_{(-\infty,0)}(h(q_1(X^1),\ldots,q_m(X^1))-\epsilon^{1/2}\log(1+\epsilon^{-1})^d)]\right| \\  = O_{d,m}(N\tau^{1/5}\log(\tau^{-1})^{dm/2+1}).\end{align*}
\item $|(f^i)^{(k)}|_\infty = O_m(\tau^{-k/5})$ for $1\leq k \leq 4$.
\end{itemize}
We then have that

\begin{align*}
\ns_\delta(f) & = 2\pr(f(A^1)=1, f(A^2)=-1)\\
& \leq \E[f^1(q_1(A^1),\ldots,q_m(A^1))f^2(q_1(A^2),\ldots,q_m(A^2))] + O(\epsilon).
\end{align*}

We would like to relate $$\E[f^1(q_1(A^1),\ldots,q_m(A^1))f^2(q_1(A^2),\ldots,q_m(A^2))]$$ to $$\E[f^1(q_1(X^1),\ldots,q_m(X^1))f^2(q_1(X^2),\ldots,q_m(X^2))].$$
In particular, we have that with respect to the Gaussian distribution, $f^i(q_1(X),\ldots,q_m(X))$ differs from
$I_{(0,\infty)}(\pm (p_0(X) - \epsilon^{1/2}\log(\epsilon^{-1})^d))$
with probability at most $O_{d,m}(N\tau^{1/5}\log(\tau^{-1})^{dm/2+1})$.
This in turn differs from $I_{(0,\infty)}(\pm p_0(X))$ with probability at most $O(d\epsilon^{1/2d}\log(\epsilon^{-1}))$ by Lemma \ref{anticoncentrationLem}.  Hence we have that

\begin{align*}
\E & [f^1(q_1(X^1), \ldots,q_m(X^1))f^2(q_1(X^2),\ldots,q_m(X^2))] \\ & = O(d\epsilon^{1/2d}\log(\epsilon^{-1})) + O_{d,m}(N\tau^{1/5}\log(\tau^{-1})^{dm/2+1}) + \E[I_{(0,\infty)}(p(X^1))I_{(-\infty,0)}(p(X^2))\\
& = O(d\epsilon^{1/2d}\log(\epsilon^{-1})) + O_{d,m}(N\tau^{1/5}\log(\tau^{-1})^{dm/2+1}) + \gns_{2\delta}(f)\\
& = O(d\epsilon^{1/2d}\log(\epsilon^{-1})) + O_{d,m}(N\tau^{1/5}\log(\tau^{-1})^{dm/2+1}) + O(d\sqrt{\delta}).
\end{align*}
Where the bound on the Gaussian noise sensitivity comes from the main Theorem of \cite{GNSBound}.

Thus, we are left with the task of bounding the difference between $\E[f^1(q_i(A^1))f^2(q_i(A^2))]$ and $\E[f^1(q_i(X^1))f^2(q_i(X^2))]$.  We do this with the replacement method.  We let $Z^{i,\ell}$ be the random vector whose $j^{\tth}$ component is $A^i_j$ if $j>\ell$ and $X^i_j$ otherwise.  We note that $Z^{i,0}=A^{i}$ and $Z^{i,n}=X^{i}$.  We proceed to bound the difference
\begin{equation}\label{z8Eqn}
|\E[f^1(q_i(Z^{1,j-1}))f^2(q_i(Z^{2,j-1}))] - \E[f^1(q_i(Z^{1,j}))f^2(q_i(Z^{2,j}))] |.
\end{equation}

We note that $Z^{i,j-1}$ and $Z^{i,j}$ agree in all but the $j^{\tth}$ coordinate.  Thus, in bounding the difference above we may consider all but the $j^{\tth}$ coordinate fixed.  We then approximate the resulting function of $Z^{1}_j,Z^2_j$ by it's Taylor series.  In particular, if we let $z_i = Z^i_j$, then for appropriate functions $g_1$ and $g_2$ (depending on the other coordinates of $Z$) we need to consider $\E[g_1(z_1)g_2(z_2)]$.  We have that $g_1(z_1)g_2(z_2)$ equals a degree 3 polynomial in $z_1$ and $z_2$ plus an error of at most
\begin{align*}
z_1^4g_1''''(t_1)g_2(0)/24 & + z_1^3z_2g_1'''(t_2)g_2'(t_3)/6 + z_1^2z_2^2g_1''(t_4)g_2''(t_5)/4 \\ & + z_1z_2^3 g_1'(t_6)g_2'''(t_7)/6 + z_2^4 g_1(0)g_2''''(t_8)/24
\end{align*}
for some points $t_i$.  Since the expectations of the degree 3 polynomials in $z_1$ and $z_2$ are the same in the Bernoulli and Gaussian case, and since the fourth moments are bounded, we have that the difference in Equation (\ref{z8Eqn}) is
$$
O\left(\E\left[|g_1''''|_\infty + |g_1'''g_2'|_\infty + |g_1''g_2''|_\infty + |g_1'''g_2'|_\infty + |g_2''''|_\infty\right] \right).
$$
Now the $k^{\tth}$ derivative of $g_i$ can be written as
$$
\sum_{\substack{i_1,\ldots,i_k=1}}^m \frac{\partial^k f^i}{\partial q_{i_1}\cdots \partial q_{i_k}}\prod_{\ell=1}^k \frac{\partial q_{i_\ell}(Z^i)}{\partial z_j}.
$$
On the other hand, by assumption, this partial derivative of $f^i$ is at most $\tau^{-k/5}$, and the product is at most
$$
\left(\max_\ell \frac{\partial q_{\ell}}{\partial x_j} \right)^{k}.
$$
Thus, the total error in Equation (\ref{z8Eqn}) is at most
$$
O\left(m^4 \tau^{-4/5} \E\left[\sum_{\ell=1}^m \sum_{i=1}^2 \left( \frac{\partial q_{\ell}(Z^i)}{\partial z_j}\right)^4 \right] \right).
$$
It is clear that
$$
\E\left[\left( \frac{\partial q_{\ell}(Z^i)}{\partial z_j}\right)^2\right] = \Inf_j(q_\ell).
$$
Thus, $\frac{\partial q_{\ell}(Z^i)}{\partial z_j}$ is a polynomial in independent Bernoulli and Gaussian random variables with second moment $\Inf_j(q_\ell)$.  Therefore by Lemma \ref{mixedHypercontractiveLem} its fourth moment is $O_d(\Inf_j(q_\ell)^2)$.  Therefore we have that the expression in Equation (\ref{z8Eqn}) is at most
$$
O_{d,m}\left(\tau^{-4/5} \sum_{\ell} \Inf_j^2(q_\ell)\right).
$$
Therefore, summing this over $j$, we get that
$$
\left|\E[f^1(q_i(A^1))f^2(q_i(A^2))]-\E[f^1(q_i(X^1))f^2(q_i(X^2))] \right|
$$
is at most
$$
O_{d,m}\left(\tau^{-4/5} \sum_{j,\ell} \Inf_j^2(q_\ell) \right).
$$
On the other hand, for fixed $\ell$ we have that $\sum_j \Inf_j(q_\ell)= O_d(1)$ and that for each $j$ that $\Inf_j(q_\ell)\leq \tau$.  Therefore, $\sum_j \Inf_j^2(q_\ell) = O_d(\tau)$.  Thus, we have that
$$
\left|\E[f^1(q_i(A^1))f^2(q_i(A^2))]-\E[f^1(q_i(X^1))f^2(q_i(X^2))] \right|=O_{d,m}(\tau^{1/5}).
$$
Recall though that
$$
\ns_\delta \leq \E[f^1(q_i(A^1))f^2(q_i(A^2))] + O(\epsilon)
$$
and that
\begin{align*}
\E[f^1(q_i(X^1)) & f^2(q_i(X^2))]  = O(d\epsilon^{1/2d}\log(1+\epsilon^{-1})) + O_{d,m}(N\tau^{1/5}\log(\tau^{-1})^{dm/2+1}) + O(d\sqrt{\delta}).
\end{align*}
Combining these yields our result.

\end{proof}

We are now prepared to prove Theorem \ref{NSBoundThm}.

\begin{proof}
Write $f=\sgn\circ p$ for $p$ a degree-$d$ polynomial.
We will reduce to the case of Proposition \ref{regNSBoundProp} by use of Theorem \ref{difRegThm}.  In particular, we may write $p$ as a decision tree of depth $O_{c,d}(\delta^{-5/6}\log(\delta^{-1})^{O(d)})$ so that a $1-\delta^{5/6}$ fraction of the leaves are polynomials with either a $(\delta^{5/6},\delta^{-c/2},O_{c,d}(1),\delta^{2d})$-regular decomposition or with variance less than $\delta^2$ times their squared mean.

Consider $A^1$ and $A^2$ random Bernoulli variables that differ in each coordinate independently with probability $\delta$.  Consider the path on the decision tree above followed by $A^1$.  With probability at least $1-\delta^{5/6}$ the resulting leaf satisfies one of the two cases specified by Theorem \ref{difRegThm}.  Furthermore, with probability at least $1-O_{c,d}(\delta^{1/6}\log(\delta^{-1})^{O(d)})$ $A^2$ agrees with $A^1$ on all coordinates queried by the decision tree.  Conditioned on this occurrence, the probability that $p(A^1)$ and $p(A^2)$ have different signs is equal to the noise sensitivity with parameter $\delta$ of the polynomial threshold function defined by the leaf.  If the leaf has a $(\delta^{5/6},\delta^{-c/2},O_{c,d}(1),\delta^{2d})$-regular decomposition, this is $O_{c,d}(\delta^{1/6-c})$ by Proposition \ref{regNSBoundProp}.  If this polynomial has low variance compared to its mean, then both $p(A^1)$ and $p(A^2)$ are the same sign as the mean of $p$ with high probability by Corollary \ref{BConcCor}.  Thus, we have that
$$
\ns_\delta(f) \leq \delta^{5/6} + O_{c,d}(\delta^{1/6}\log(\delta^{-1})^{O(d)}) + O_{c,d}(\delta^{1/6-c}) = O_{c,d}(\delta^{1/6-c}).
$$
\end{proof}

Theorem \ref{ASBoundThm} now follows immediately by Lemma 8.1 of \cite{sens2}.  And Theorem \ref{gasBoundCor} follows from Theorem \ref{ASBoundThm} and Lemma \ref{GasASBoundLem}.

\section{Application to PRGs for PTFs with Bernoulli Inputs}\label{BernoulliPRGSec}

In \cite{MZ}, Meka and Zuckerman developed a relatively small pseudo-random generator of polynomial threshold functions with Bernoulli inputs.  Their generator was defined as follows.  Let $h:[n]\rightarrow [a]$ be a hash function picked from a 2-independent family.  Let $A^1,\ldots,A^a:[n]\rightarrow \{-1,1\}$ be chosen independently from a $k$-independent hash family.  Meka and Zuckerman's generator is given by $A_i = A^{h(i)}_i$.  Meka and Zuckerman show that for appropriate chosen $m=\tilde{O}(\epsilon^{-2})$ and $a=O(\epsilon^{-O(d)})$ that this generator fools all degree-$d$ polynomial threshold functions to within $\epsilon$.

Meka and Zuckerman's proof is essentially to think of $h$ as constant and to use the replacement method to bound the expected errors as the $A^i$ are replaced by random Gaussians vectors one at a time.  If the polynomial in question is sufficiently regular, then these errors will be small, and thus, the expected value of the PTF in question over the PRG will be close to the expected value of the PTF over random Gaussian inputs, and by the Invariance Principle, the expected value at random Bernoulli inputs will also be close.  Unfortunately, this technique had been limited by the classical Invariance Principle and Regularity Lemma, and thus, could not produce a PRG of seed length less than $\epsilon^{-O(d)}$.  In this section, we will show how our Diffuse Invariance Principle and Regularity Lemma can improve this to produce a PRG of seed length $O_d(\log(n)\epsilon^{-O(1)})$.

We begin by producing a pseudorandom generator that works in the case of regular polynomials, and then reducing the general case to this one.

\subsection{The Regular Case}

\begin{prop}\label{regBPRGProp}
Let $p$ be a degree-$d$ polynomial in $n$ variables with a $(\tau,N,m,\epsilon)$-regular decomposition.  Let $a$ be a positive integer.  Let $h:[n]\rightarrow[a]$ be picked randomly from a 2-independent hash family and for each $h$ let $A^1,\ldots,A^a:[n]\rightarrow\{-1,1\}$ be picked indecently from $4d$-independent hash families.  Define the $n$-variable function $A$ in terms of $h$ and $A^i$ as $A_i = A^{h(i)}_i$.  Then if $B$ is a Bernoulli random variable, $|\E[\sgn(p(A))] - \E[\sgn(p(B))]|$ is at most
$$
O_{d,m}(N\tau^{1/5}\log(1+\tau^{-1})^{dm/2+1}) + O(d\epsilon^{1/d}\log(\epsilon^{-1})^{1/2}) + O(a^{-1}\tau^{-1}).
$$
\end{prop}

We begin by showing that a similar statement holds for an appropriate choice of $h$.
\begin{lem}
Let $p$ and $p_0$ be degree-$d$ polynomials with $|p-p_0|_{2,B}^2 \leq \epsilon^2 \var(p_0)$ so that $p_0$ has a $(\tau,N)$-diffuse decomposition $(g,q_1,\ldots,q_m)$ with $q_i$ multilinear ($1/2>\epsilon,\tau>0$).  Suppose furthermore that $h:[n]\rightarrow [a]$ is a function so that
$$
\sum_{j=1}^a \left( \sum_{\ell:h(\ell)=j} \sum_i \Inf_\ell(q_{i}) \right)^2 \leq \tau.
$$
Let $A^1,\ldots,A^a:[n]\rightarrow\{-1,1\}$ be picked independently from a $4d$-independent hash family.  Define the random variable $A$ so that its $i^{\tth}$ coordinate is the $i^{\tth}$ coordinate of $A^{h(i)}$.  Then for $G$ a random Gaussian we have that
$$
|\E[\sgn(p(A))] - \E[\sgn(p_0(G))]| \leq O_{d,m}(N\tau^{1/5}\log(\tau^{-1})^{dm/2+1}) + O(d\epsilon^{1/2d}).
$$
\end{lem}
\begin{proof}
We show that
$$
\pr(p(A)\leq 0) \leq \pr(p_0(G)\leq 0) +O_{d,m}(N\tau^{1/5}\log(\tau^{-1})^{dm/2+1}) + O(d\epsilon^{1/2d}).
$$
The other direction will follow analogously.

First, we note that by Corollary \ref{BConcCor} that with probability $1-O(\epsilon)$ that $|p(A)-p_0(A)|<\epsilon^{1/2}\sqrt{\var(p_0)} \leq \epsilon^{1/2}|p_0|_2$.  Therefore,
$$
\pr(p(A)\leq 0) \leq \pr(p_0(A)\leq -\epsilon^{1/2}|p_0|_2) + O(\epsilon).
$$
On the other hand,
$$
\pr(p_0(G)\leq -\epsilon^{1/2}|p_0|_2) = \pr(p_0(G)\leq 0) + O(d\epsilon^{1/2d})
$$
by Lemma \ref{anticoncentrationLem}.  Hence it will suffice to prove that
$$
\pr(p_0(A)\leq -\epsilon^{1/2}|p_0|_2) \leq \pr(p_0(G)\leq -\epsilon^{1/2}|p_0|_2) +O_{d,m}(N\tau^{1/5}\log(\tau^{-1})^{dm/2+1}).
$$
Modifying $p_0$ by $\epsilon^{1/2}|p_0|_2$, it suffices to prove under the same hypothesis that
$$
\pr(p_0(A)\leq 0) \leq \pr(p_0(G)\leq 0) +O_{d,m}(N\tau^{1/5}\log(\tau^{-1})^{dm/2+1}).
$$

The proof is by Proposition \ref{replacementProp}.  Let $B^i$ be the vector of entries $A^i_j$ of $A^i$ for which $h(j)=i$.  Reordering, the coordinate variables we can make it so that $A=(B^1,\ldots,B^a)$. Similarly, let $G=(G^1,\ldots,G^a)$.  Note that since the $q_i$ are multilinear and degree at most $d$, that any degree-3 polynomial in the $q_i$ has the same expectation under the $B^i$ as under the $G^i$.  We may thus apply Proposition \ref{replacementProp} with $k=4$.  We have that
\begin{equation}\label{Z11Eqn}
|\pr(p_0(A)\leq 0) -\pr(p_0(G)\leq 0)| = O_{d,m}(N\tau^{1/5}\log(\tau^{-1})^{dm/2+1}+\tau^{-4/5}T).
\end{equation}
Recall that $T$ above is
$$
\sum_{i,j} T_{i,j}
$$
where $T_{i,j}$ is
\begin{align*}
&\E\left[\left(q_i(B^1,\ldots,B^{j},G^{j+1},\ldots,G^a) - \E_Y[q_i(B^1,\ldots,B^{j-1},Y,G^{j+1},\ldots,G^a)] \right)^4 \right] \\ +&\E\left[\left(q_i(B^1,\ldots,B^{j-1},G^{j},\ldots,G^a) - \E_Y[q_i(B^1,\ldots,B^{j-1},Y,G^{j+1},\ldots,G^a)] \right)^4 \right].
\end{align*}
By the $4d$-independence of the $B^i$, this expectation is the same as it would be if they were fully independent Bernoulli variables.  Thus, by Lemma \ref{mixedHypercontractiveLem}, this is at most
\begin{align*}
&O_d\left(\E\left[\left(q_i(B^1,\ldots,B^{j},G^{j+1},\ldots,G^a) - \E_Y[q_i(B^1,\ldots,B^{j-1},Y,G^{j+1},\ldots,G^a)] \right)^2 \right]\right)^2 \\ +&O_d\left(\E\left[\left(q_i(B^1,\ldots,B^{j-1},G^{j},\ldots,G^a) - \E_Y[q_i(B^1,\ldots,B^{j-1},Y,G^{j+1},\ldots,G^a)] \right)^2 \right]\right)^2.
\end{align*}
Since the terms in the expectations above are at most quadratic in any coordinate, the expectation is unchanged by replacing Gaussian inputs with Bernoullis and hence
$$
T_{i,j} = O_d\left(\E_{B^1,\ldots,\hat{B^j},\ldots,B^a}[\var_{B^j}(q_i(B))]^2 \right).
$$
The variance above is clearly the sum of the squares of the coefficients of the non-constant terms of the polynomial obtained by substituting the values of $B^1,\ldots,\hat{B^j},\ldots,B^a$ into $q_i$.  The expectation of this is easily seen to be the sum of the squares of the coefficients of the monomials in $q_i$ containing at least one of the $B^j$ variables.  This in turn is clearly at most $\sum_{\ell:h(\ell)=j} \Inf_\ell(q_i)$.  Thus,
\begin{align*}
T &= \sum_{i=1}^m\sum_{j=1}^a T_{i,j}\\
& \leq \sum_{i=1}^m\sum_{j=1}^a O_d\left( \sum_{\ell:h(\ell)=j} \Inf_\ell(q_i)\right)^2\\
& \leq \sum_{j=1}^a \left( \sum_{\ell:h(\ell)=j} \sum_i \Inf_\ell(q_{i}) \right)^2\\
& \leq \tau.
\end{align*}
Thus, by Equation (\ref{Z11Eqn}),
$$
|\pr(p_0(A)\leq 0) -\pr(p_0(G)\leq 0)| = O_{d,m}(N\tau^{1/5}\log(\tau^{-1})^{dm/2+1}),
$$
completing our proof.
\end{proof}

We can now prove Proposition \ref{regBPRGProp}
\begin{proof}
Let $q_1,\ldots,q_m$ be as given in the $(\tau,N,m,\epsilon)$-regular decomposition of $p$.

By the above Lemma, it suffices to prove that with probability $1-O(a^{-1}\tau^{-1})$ over $h$ that
$$
\sum_{j=1}^a \left( \sum_{i:h(i)=j} \sum_\ell \Inf_i(q_{\ell}) \right)^2 = O_{d,m}(\tau).
$$
On the other hand this is at most
$$
m \sum_\ell \sum_i \Inf_i(q_\ell)^2 + m\sum_\ell \sum_{i\neq i': h(i)=h(i')} \Inf_i(q_\ell)\Inf_{i'}(q_\ell).
$$
Since $\sum_i \Inf_i(q_\ell) = O_d(1)$ for each $\ell$ and since each $\Inf_i(q_\ell)$ is at most $\tau$, the first term above is $O_{d,m}(\tau)$.  The expectation of the latter term above is
\begin{align*}
m/a \sum_\ell \sum_{i \neq i'} \Inf_i(q_\ell)\Inf_{i'}(q_\ell) &\leq O_m\left(a^{-1} \sum_\ell \left(\sum_i \Inf_i(q_\ell)\right)^2\right)\\
& = O_{d,m} (a^{-1}).
\end{align*}
Our result follows from the Markov bound on this random variable.
\end{proof}

\subsection{The General Case}

We are now prepared to state our conclusions in the general case.

\begin{thm}\label{BPRGThm}
Let $A$ be a random variable defined as follows.  Let $h:[n]\rightarrow[a]$ be picked randomly from a 2-independent hash family for $a=\epsilon^{-6}$.  Let $A^1,\ldots,A^a:[n]\rightarrow\{-1,1\}$ be picked independently from $k$-independent hash families for $k=\epsilon^{-5}+4d$.  Let $A_i=A^{h(i)}_i$ for $1\leq i \leq n$.  Note that $A$ can be generated from a seed of length $O(\log(n) \epsilon^{-11})$.  Let $B$ be a random $n$-dimensional Bernoulli random variable, and let $f$ be any degree-$d$ polynomial threshold function in $n$ variables.  Then for any $c>0$
$$
|\E[f(A)] - \E[f(B)]| = O_{c,d}(\epsilon^{1-c}).
$$
\end{thm}
\begin{rmk}
Note that by changing the values of $a$ and $k$ above we can find a PRG with seed length $O_{c,d}(\log(n) \epsilon^{-11-c})$ that fools degree-$d$ PTFs to within $\epsilon$.
\end{rmk}
\begin{proof}
Note that the coordinates of $A$ are $k$-independent (since they are for each possible value of $h$).  Assume that $\epsilon$ is sufficiently small (since otherwise there is nothing to prove).  By Theorem \ref{difRegThm} we know that $f$ can be written as a decision tree of depth $\epsilon^{-5}$ so that with probability $1-O(\epsilon)$ a randomly chosen leaf is of the form $\sgn\circ p$ where either $\var(p(B)) < \epsilon^2 |\E[p(B)]|$ or $p$ has an $(\epsilon^{5},\epsilon^{-c/5},O_{c,d}(1),\epsilon^{2d})$-regular decomposition.  For each such decision-tree path, condition on $A$ and $B$ on having the appropriate values on the appropriate $\epsilon^{-5}$ coordinates defining this branch of the decision tree.  Note that the conditional distribution on $A$ can be written in the same form as $A$ was originally written only with the $A^i$ perhaps only being $4d$-independent.

There is a probability of $1-O(\epsilon)$ that $p$ satisfies one of the two conditions outlined above.  If the former condition holds, both $p(A)$ and $p(B)$ have the same sign as $\E[p(B)]$ with probability $1-O(\epsilon)$.  In the latter case, by Proposition \ref{regBPRGProp}, we have that for an appropriate $p_0$
$$
\E[\sgn(p(A))] = \E[\sgn(p_0(G))] + O_{d,m}(\epsilon^{1-c}) = \E[\sgn(p(B))] + O_{d,m}(\epsilon^{1-c})
$$
(since $B$ is also of the form specified in Proposition \ref{regBPRGProp}).  This completes our proof.
\end{proof}

\section{Conclusion}\label{ConclusionSec}

We have introduced the notion of a diffuse decomposition of a polynomial and proved that they exist for reasonable parameters.  This in turn has allowed us to make improvements on known bounds for several major problems relating to polynomial threshold functions.  There are several directions in which this work might be expanded.  Perhaps most importantly is that the theory introduced in this paper may well have applications to other problems of interest in the field.  On the other hand, Theorem \ref{DDTheorem} still has room for improvement.  In particular, I believe that such a diffuse decomposition should exist with size merely polynomial in $dN/c$.  Producing such a technical improvement, would allow one to noticeably improve the $d$-dependence in all of the applications presented in this paper.

\section*{Acknowledgements}

This research was done with the support of an NSF postdoctoral fellowship.


\begin{thebibliography}{[99]}

\bibitem{curciutApp} Richard Beigel \emph{The polynomial method in circuit complexity}, Proc. of 8th Annual Structure in Complexity Theory Conference (1993), pp. 82-95.

\bibitem{BHyp} Aline Bonami \emph{{\'E}tude des coefficients {F}ourier des fonctions de {$L^{p}(G)$}}, Annales de l'Institute Fourier Vol. 20(2), p. 335-402, 1970.

\bibitem{anticoncentration} A. Carbery, J. Wright \emph{Distributional and $L^q$ norm inequalities for polynomials over convex bodies in $\R^n$}
Mathematical Research Letters, Vol. 8(3), pp. 233–248, 2001.

\bibitem{sensitivity} 	Ilias Diakonikolas, Prahladh Harsha, Adam Klivans, Raghu Meka, Prasad Raghavendra, Rocco A. Servedio, Li-Yang Tan \emph{Bounding the average sensitivity and noise sensitivity of polynomial threshold functions} Proceedings of the 42nd ACM symposium on Theory of computing (STOC), 2010.

\bibitem{sense3} Ilias Diakonikolas, Prasad Raghavendra, Rocco A. Servedio, Li-Yang Tan \emph{Average sensitivity and noise sensitivity of polynomial threshold functions} http://arxiv.org/abs/0909.5011.

\bibitem{reg} Ilias Diakonikolas, Rocco Servedio, Li-Yang Tan, Andrew Wan \emph{A Regularity Lemma, and Low-Weight Approximators, for Low-Degree Polynomial Threshold Functions}, 25th Conference on Computational Complexity (CCC), 2010

\bibitem{Freplacement} W. Feller \emph{An introduction to probability theory and its applications} Vol. II. Second edition.
    John Wiley \& Sons Inc., New York, 1971.

\bibitem{gl} Craig Gotsman, Nathan Linial \emph{Spectral properties of threshold functions} Combinatorica,
    Vol. 14(1), p. 35–50, 1994.

\bibitem{GT} Ben Green, Terence Tao, \emph{The distribution of polynomials over Fnite Felds, with applications to the
Gowers norms},  Contrib. Discrete Math Vol. 4(2), p. 1-36, 2009.

\bibitem{sens2} Prahladh Harsha, Adam Klivans, Raghu Meka \emph{Bounding the Sensitivity of Polynomial Threshold Functions} http://arxiv.org/abs/0909.5175.

\bibitem{GPRG} Daniel M. Kane \emph{A Small PRG for Polynomial Threshold Functions of Gaussians} Symposium on the Foundations Of Computer Science (FOCS), 2011.

\bibitem{GNSBound} Daniel M. Kane \emph{The Gaussian Surface Area and Noise Sensitivity of Degree-$d$ Polynomial Threshold Functions}, in Proceedings of the 25th annual IEEE Conference on Computational Complexity (CCC 2010), pp. 205-210.

\bibitem{KL} Tali Kaufman, Shachar Lovett \emph{Worst Case to Average Case Reductions for Polynomials}, The 49th Annual IEEE Symposium on Foundations of Computer Science (FOCS 2008).

\bibitem{learningApp} Adam R. Klivans, Rocco A. Servedio \emph{Learning DNF in time $2^{O(n^{1/3})}$}, J. Computer and System Sciences Vol. 68, p. 303-318, 2004.

\bibitem{LReplacement} J. W. Lindeberg  \emph{Eine neue herleitung des exponential-gesetzes in der wahrscheinlichkeitsrechnung}, Math. Zeit., Vol. 15, p.211-235, 1922.

\bibitem{MZ} Raghu Meka, David Zuckerman \emph{Pseudorandom generators for polynomial threshold functions}, Proceedings of the 42nd ACM Symposium on Theory Of Computing (STOC 2010).

\bibitem{MOO} E. Mossel, R. O'Donnell, and K. Oleszkiewicz \emph{Noise stability of functions with low influences:
invariance and optimality} Proceedings of the 46th Symposium on Foundations of Computer Science
(FOCS), pages 21–30, 2005.

\bibitem{hypercontractivity} Nelson \emph{The free {M}arkov field}, J. Func. Anal. Vol. 12(2), p. 211-227, 1973.

\bibitem{PZ} R.E.A.C.Paley and A.Zygmund, \emph{A note on analytic functions in the unit circle}, Proc. Camb. Phil. Soc. Vol. 28, p. 266–272, 1932.

\bibitem{commApp} Alexander A. Sherstov \emph{Separating AC0 from depth-2 majority circuits}, SIAM J. Computing Vol. 38, p. 2113-2129, 2009.

\end{thebibliography}
\end{document}